\documentclass[12pt]{article}

\usepackage{bold-extra}
\usepackage{amsmath}             
\usepackage{amssymb}             
\usepackage{amsfonts}            
\usepackage{mathrsfs}            
\usepackage{amsthm}
\usepackage{mathtools}
\usepackage{commath}
\usepackage{bm}
\usepackage{graphicx}
\usepackage{geometry}
\usepackage{float}
\usepackage{listings}
\usepackage{subfigure}
\usepackage{multirow}
\usepackage{diagbox}
\usepackage[export]{adjustbox}
\usepackage[all]{xy}
\usepackage{tikz-cd}
\usepackage{fancyhdr}
\usepackage{hyperref}
\usepackage{enumerate}

\usepackage{titlesec} 
\titleformat{\section}[block]{\Large\bfseries}{\arabic{section}}{0.7em}{}[]
\titleformat{\subsection}[block]{\large\bf}{\arabic{section}.\arabic{subsection}}{0.5em}{}[]
\titleformat{\subsubsection}[runin]{\small\bfseries}{\arabic{subsection}-\alph{subsubsection}}{0.5emem}{}[]

\theoremstyle{definition}
\newtheorem{defi}{Definition}[section]
\newtheorem{ex}[defi]{Example}
\theoremstyle{plain}
\newtheorem{theo}[defi]{Theorem}
\newtheorem{lemma}[defi]{Lemma}
\newtheorem{prop}[defi]{Proposition}
\newtheorem{cor}[defi]{Corollary}

\newtheorem{conj}{Conjecture}

\theoremstyle{remark}
\newtheorem*{rmk}{Remark}

\newcommand{\inte}[1]{\mathcal{O}_{#1}}

\newcommand{\Bdrp}[1]{\mathbb{B}_{\mathrm{dR}}^+(#1)}
\newcommand{\Bpp}[1]{\mathbb{B}_{#1}}
\newcommand{\Bp}[1]{B_{#1}}
\newcommand{\act}[2]{\ ^{[#1]}#2}

\newcommand{\cyc}{\mathrm{cyc}}
\newcommand{\mic}{\mathrm{mic}}

\newcommand{\bB}{\mathbb{B}}

\newcommand{\bL}{\mathbb{L}}

\newcommand{\bQ}{\mathbb{Q}}

\newcommand{\calO}{\mathcal{O}}

\newcommand{\calX}{\mathcal{X}}

\newcommand{\rf}{\mathrm{f}}
\newcommand{\rg}{\mathrm{g}}

\def\Mat{{\mathrm{Mat}}}
\def\B{\mathbb{B}_{\mathrm{dR}}^+}
\def\bdrp{{B_{\mathrm{dR}}^+}}

\def\ep{\epsilon}
\def\dR{\mathrm{dR}}
\def\vpi{\varpi}
\def\pet{v}
\def\et{\textrm{\'et}}
\def\Q{\mathbb{Q}}
\def\m{\mathfrak{m}}

\def\obdr{\mathcal{O}\mathbb{B}_{\mathrm{dR}}}

\def\T{\mathbb{T}}
\def\X{\mathbb{X}}
\def\Spa{\mathrm{Spa}}
\def\Z{\mathbb{Z}}
\def\Hom{\mathrm{Hom}}
\def\End{\mathrm{End}}
\def\M{\mathbf{M}}
\def\Perf{\mathcal{P}er\! f}
\def\Ainf{\mathbb{A}_{\mathrm{inf}}}
\def\ainf{A_{\mathrm{inf}}}
\def\Rep{\mathcal{R}{e p}}
\def\bk{\mathbf{k}}

\title{On smooth adic spaces over $\B$ and sheafified $p$-adic Riemann--Hilbert correspondence}
\author{Jiahong Yu\thanks{Academy of Mathematics and Systems Science, Chinese Academy of Sciences} \thanks{Morningside Center of Mathematics, CAS}}

\makeatletter
\renewcommand{\@maketitle}{
\newpage
 \null
 \vskip 2em%
 \centering
  {\large \bfseries{\MakeUppercase{\@title}} \par}%
 \par
 \vspace{0.5cm}
 \centering{\@author\par}
 } 
 
\makeatother

\begin{document}

\pagestyle{fancy}
\fancyhf{}
\fancyhead[L]{\textsl\leftmark}
\fancyfoot[C]{\thepage}
\renewcommand{\headrulewidth}{0.2pt}
\renewcommand{\footrulewidth}{0.2pt}

\maketitle

\abstract{Let $C$ be a completely algebraic closed non-archimedean field over $\mathbb{Q}_p$ and $\alpha,r$ be two positive integers. Denote by $B_\alpha$ the ring $\mathbb{B}_{\mathrm{dR}}^+(C)/(\ker\theta)^\alpha$. This paper first constructs a sheafified $p$-adic Riemann--Hilbert correspondence. Specifically, we construct a canonical sheaf isomorphism on $X_{\mathrm{\acute{e}t}}$,
\[
R^1\nu_*\big( \mathrm{GL}_r(\mathbb{B}_{\mathrm{dR}}^+/(\ker\theta)^\alpha) \big) \cong \mathrm{MIC}_{r}(X)\{-1\},
\]
where the first term is identified with the sheaf of isomorphism classes of $v$-vector bundles with coefficients in $\mathbb{B}_{\mathrm{dR}}^+/(\ker\theta)^{\alpha}$, and the second term is defined as the sheaf of isomorphism classes of integrable connections of rank $r$.

We then define the moduli space of integrable connections on $X$ and the moduli space of $v$-vector bundles on $X$ with coefficients in $\mathbb{B}_{\mathrm{dR}}^+/(\ker\theta)^{\alpha}$, and prove that they are small $v$-stacks in the sense of Scholze.

These constructions generalize Heuer's work on $p$-adic Simpson correspondence.
}

\tableofcontents

\section{Introduction}

Let $X$ be a smooth analytic variety over $\mathbb{C}$. The classical complex analytic Riemann--Hilbert correspondence establishes a canonical equivalence of categories between integrable connections on $X$ and local systems with coefficients in $\mathbb{C}$. The key to its proof is the Poincar\'e's lemma, which states that for any point on $X$, there exists an open neighborhood such that every integrable connection on it is trivial.

However, in $p$-adic geometry, due to the lack of a $p$-adic analogue of the Poincar\'e's lemma, a Riemann--Hilbert correspondence has not been established. Nevertheless, in \cite{scholze2012padic}, Scholze defined the so-called de Rham period sheaf (also known as $\mathbb{B}_{\mathrm{dR}}^+$) for all smooth rigid analytic varieties over $\mathbb{Q}_p$ and demonstrated that it admits an analogue of the Poincar\'e's lemma. Furthermore, in \cite{Liu_2016}, Ruochuan Liu and Xinwen Zhu defined, for each smooth analytic variety $X$ over $\mathbb{Q}_p$, a new ringed space $(\mathcal{X},\mathcal{O}_{\mathcal{X}}^+)$ and defined a Riemann--Hilbert functor for the category of nilpotent $\mathbb{B}_{\mathrm{dR}}^+$-local systems on it (for details, see Subsection \ref{subsec: liuzhu intro}). In fact, $(\mathcal{X},\mathcal{O}_{\mathcal{X}}^+)$ is a smooth adic space defined over Fontaine's period ring $\mathbb{B}_{\mathrm{dR}}(\mathbb{C}_p)^+$, along with its structure sheaf. This motivates the conjecture that for a smooth adic space $X$ defined over $\mathbb{B}_{\mathrm{dR}}(\mathbb{C}_p)^+$, there exists an equivalence of categories between $\mathbb{B}_{\mathrm{dR}}^+$-local systems on $X$ and integrable connections on $X$. We name this conjecture the geometric $p$-adic Riemann--Hilbert conjecture (see Conjecture \ref{conj:geometric riemann hilbert}). A natural consequence of this conjecture is a sheaf-theoretic $p$-adic Riemann--Hilbert correspondence (see Theorem \ref{theorem: sheafRH}). The first theme of this article is to prove this consequence, which provides indirect evidence for the geometric $p$-adic Riemann--Hilbert conjecture.

Back to the setting of the complex geometry. For a complex smooth analytic variety $X$, Simpson further defined the moduli space of $\mathbb{C}$-local systems on $X$ (also known as the Betti moduli space) and the moduli space of integrable connections on $X$. From this perspective, the classical Riemann--Hilbert correspondence can be generalized as an isomorphism between these two moduli spaces.

We hope that Conjecture \ref{conj:geometric riemann hilbert} can also be generalized to a statement about moduli spaces. Therefore, we first need to define the moduli space of $\mathbb{B}_{\mathrm{dR}}^+$-local systems and the moduli space of integrable connections on $X$ for a smooth adic space $X$ over $\mathbb{B}_{\mathrm{dR}}^+$. The second goal of this article is to define these two moduli spaces using the language of small $v$-stacks introduced by Scholze in \cite{scholze2022etale}.

\subsection{Geometric $p$-adic Riemann--Hilbert correspondence}\label{subsec: liuzhu intro}

Ruochuan Liu and Xinwen Zhu constructed an analogue of the classical complex Riemann--Hilbert functor in \cite{Liu_2016}. Specifically, let $k$ be a finite extension of $\Q_p$ and $X$ be a smooth rigid space over $k$. Let $K$ be the completion of the algebraic closure of $k$. As constructed in \cite{scholze2012padic}, there exists a pro-\'etale sheaf ${\obdr}_{,X}$ equipped with a natural differential operator
$$\mathrm{d}_{\dR}:{\obdr}_{,X}\to {\obdr}_{,X}\otimes_{\nu^{-1}\inte{X}}\nu^{-1}\Omega_{X/k},$$ 
where $\nu$ denotes the projection from the pro-\'etale topos to the \'etale topos.
Define the ringed topology 
$$(\mathcal{X},\inte{\mathcal{X}})=(X_{\text{\'et}},\nu_*{\obdr}_{,X})$$ 
and the operator $\mathrm{d}_{\dR}$ descends to a differential operator $\mathrm{d}_{\dR}:\inte{\mathcal{X}}\to \inte{\mathcal{X}}\otimes_{\inte{X}}\Omega_{X/k}$.  The structure sheaf $\inte{\mathcal{X}}$ admits a natural filtration.

In the following, we will use $\Omega_\calX$ (resp. $\Omega_{\calX^+}$) to denote the sheaf $\calO_\calX\otimes_{\calO_{X}}\Omega_{X/k}$ (resp. $\calO_\calX^+\otimes_{\calO_{X}}\Omega_{X/k}$).

For any $\Q_p$-local system $\mathbb{L}$, define $$\mathcal{RH}(\mathbb{L})=\nu_*\big({\obdr}_{,X}\otimes_{\widehat{\Q}_p}\widehat{\mathbb{L}}\big)$$ and $\nabla_{\mathbb{L}}$ be the push-forward of $\mathrm{d}_{\dR}\otimes_{\Q_p}id_{\widehat{\mathbb{L}}}$\footnote{We follow the notation in \cite{scholze2012padic} that for any local system $\bL$, $\widehat{\bL}$ is the associated pro-\'etale sheaf on $X$.}. They proved the following result (cf. \cite[Theorem 3.8]{Liu_2016}):
\begin{enumerate}[(1)]
    \item For any $\mathbb{L}$, $\mathcal{RH}(\mathbb{L})$ is a vector bundle of the rank equal to that of $\mathbb{L}$.

    \item There is a natural filtration on $\mathcal{RH}(\mathbb{L})$ which makes $\mathcal{RH}(\mathbb{L})$ a filtered module over $\inte{\mathcal{X}}$. The connection $\nabla_{\mathbb{L}}:\mathcal{RH}(\mathbb{L})\to \mathcal{RH}(\mathbb{L})\otimes_{\inte{\mathcal{X}}}\Omega_{\mathcal{X}}$ satisfies the Griffith transversality with respect to this filtration.

    \item The $0$-th graded piece of $\nabla_{\mathbb{L}}:\mathcal{RH}(\mathbb{L})\to \mathcal{RH}(\mathbb{L})\otimes_{\inte{\mathcal{X}}}\Omega_{\mathcal{X}}$ coincides with the $p$-adic Simpson correspondence of $\mathbb{L}$.
\end{enumerate}

In fact, they proved a slightly stronger result. Let $\obdr^{[0,\infty)}$ denote the $0$-th filtration of $\obdr$, and let $\inte{\mathcal{X}}^+$ be the sheaf $\nu_*\obdr^{[0,\infty)}$ on $\mathcal{X}$. The restriction of $\mathrm{d}_{\dR}$ induces an operator $\mathrm{d}_{\dR}^+:\inte{\mathcal{X}}^+\to \Omega_{\mathcal{X}^+}\{-1\}$, where $\{-1\}$ means tensoring with the module $\ker \theta^{-1}\Bdrp{K}$, see Definition \ref{definition: BKgeneral}. Their theorem can be strengthened as follows:

Let $$\mathcal{RH}^{[0,\infty)}(\mathbb{L})=\nu_*\big({\obdr}_{,X}^{[0,\infty)}\otimes_{\widehat{\Q}_p}\widehat{\mathbb{L}}\big).$$
\begin{enumerate}[(1)]
    \item For any $\mathbb{L}$, $\mathcal{RH}(\mathbb{L})^{[0,\infty)}\subseteq \mathcal{RH}(\mathbb{L})$ is a vector bundle on the ringed topology $(\mathcal{X},\inte{\mathcal{X}}^+)$ of the same rank as $\mathbb{L}$.

    \item The restriction of $\nabla_{\mathbb{L}}$ induces a connection
    $$\nabla_{\mathbb{L}}:\mathcal{RH}(\mathbb{L})^{[0,\infty)}\to \mathcal{RH}(\mathbb{L})^{[0,\infty)}\otimes_{\inte{\mathcal{X}}^+}\Omega_{\mathcal{X}^+}\{-1\}.$$
\end{enumerate}

The strengthened form implies that the $p$-adic geometric Riemann--Hilbert correspondence should be concerned as a deformation of the $p$-adic Simpson correspondence.

However, the construction in \cite{Liu_2016} highly depends on the choice of the rational structure of $X$, and the local system $\mathbb{L}$ must be defined on $X$ rather than merely defined on $X_{K}$. In analogy with the $p$-adic Simpson correspondence introduced by Faltings (for curves: see \cite{Faltings_2005}) and Heuer (general cases, see \cite{heuer2023padic}), we conjecture the following.

\begin{conj}[Geometric $p$-adic Riemann--Hilbert correspondence]\label{conj:geometric riemann hilbert}
    Let $C$ be a complete algebraically closed $p$-adic field, and let $\bdrp$ denote its de Rham period ring. Let $X/\bdrp$ be a smooth adic space and $\overline{X}=C\otimes_{\bdrp}X$. There exists a canonical rank-preserving equivalence between categories
    $$\mathcal{RH}:\mathcal{V}ect(\mathbb{B}_{\mathrm{dR}}^+)\cong\mathcal{MIC}(X),$$
    where
    \begin{itemize}
        \item $\mathcal{V}ect(\mathbb{B}_{\mathrm{dR}}^+)$ is the category of locally free pro-\'etale sheaves of $\mathbb{B}_{\mathrm{dR}}^+$-modules on $\overline{X}$;
        \item $\mathcal{MIC}(X)$ is the category of pairs 
        $$(M,\nabla:M\to M\otimes_{\inte{X}}\Omega_{X}\{-1\}),$$ where $M$ is an analytic vector bundle and $\nabla$ is an integrable connection.
    \end{itemize}
\end{conj}

\subsection{Sheafified $p$-adic Simpson correspondence and Riemann--Hilbert correspondence}

To be more explicit, we introduce the following two notations and keep them in the paper.
\begin{itemize}
    \item Let $\Q_p^\cyc$ be the completion of the extension of $\Q_p$ by adding all $p$-power roots of unity.
    \item For any adic space $X$, denote by $\nu$ the projection from the $v$-topology $X_{\pet}$ of $X$ to the \'etale topology $X_\et$ of $X$.
\end{itemize}

\begin{rmk}Here we explain the choice of topology. Unlike in the previous subsection, throughout this article, we will use the $v$-topology instead of the pro-\'etale topology. In fact, the categories of locally free sheaves with coefficients in $\mathbb{B}_{{\alpha}}$ (see Theorem \ref{theorem: sheafRH}) under these two topologies are equivalent; therefore, all instances of the $v$-topology in this paper, especially Theorems \ref{theorem: sheafRH} and \ref{theorem: modulismall}, can be replaced by the pro-\'etale topology. We use the $v$-topology because its definition is simpler. This choice leads to a potential issue: $\mathcal{O}\mathbb{B}_{\mathrm{dR}}^+$ is currently only defined on the pro-\'etale topology. However, we will not use this notion in this paper.\end{rmk}

Having fixed these conventions, we now return to the main narrative. First, we briefly recall the sheafified $p$-adic Simpson correspondence constructed by Heuer in \cite{heuer2022moduli}. He defined a new class of adic spaces — smoothoid spaces — characterized by being locally smooth over perfectoid spaces. For any smoothoid space $X$ over a perfectoid field $K/\Q_p^{\cyc}$, he established a canonical isomorphism between sheaves
$$R^1\nu_*\mathrm{GL}_r(\widehat{\mathcal{O}}_{X_\pet})\cong \mathrm{Higgs}_r(X)$$
where $\mathrm{Higgs}_r(X)$ is the \'etale sheaf associated to the presheaf of isomorphism classes of rank $r$ Higgs bundles.

\begin{rmk}
    Heuer assumed that $K$ is algebraically closed, which is unnecessary.
\end{rmk}

In light of the discussion in Subsection \ref{subsec: liuzhu intro}, the $p$-adic Riemann--Hilbert correspondence should be regarded as a deformation of the $p$-adic Simpson correspondence. This naturally leads us to ask whether the sheafified $p$-adic Simpson correspondence can be extended to a sheafified $p$-adic Riemann–Hilbert correspondence. We provide a positive answer to this question. Let $(K,K^+)$ be a perfectoid Tate--Huber pair over $\Q_p^\cyc$, $\alpha$ a positive integer and $X$ a smooth adic space over $\Bp{\alpha}=\B(K)/(\ker\theta)^\alpha$. Define $\overline{X}=X\otimes_{\Bp{\alpha},\theta}K$.

Our first theorem is the following.

\begin{theo}\label{theorem: sheafRH}
    Let $r$ be a positive integer. There is a canonical isomorphism of sheaves
    $$\mathrm{RH}_{X,r}:R^1\nu_*\big(\mathrm{GL}_r({\mathbb{B}}_{\alpha,\overline{X}})\big)\cong \mathrm{MIC}_r(X)\{-1\},$$
    where
    \begin{itemize}
        \item ${\mathbb{B}}_{\alpha,\overline{X}}$ is the $v$-sheaf $\bB^+_{\dR,\overline{X}}/(\ker\theta)^\alpha$ on $\overline{X}_{\pet}$ where $\bB^+_{\dR,\overline{X}}$ is the de Rham period sheaf in the sense of \cite[Definition 6.1]{scholze2012padic}.\footnote{Scholze assumed that $X$ is locally noetherian. However, this assumption is not necessary.}

        \item $\mathrm{MIC}_r(X)\{-1\}$ is the sheaf on $X_{\et}$ associated to the presheaf on $X_\et$ which sends each $U\in X_{\et}$ to the isomorphic classes of pairs
        $$(M,\nabla:M\to M\otimes_{\calO_U}\Omega_{U/B_\alpha}\{-1\}),$$
        where $M$ is a vector bundle of rank $r$ on $U$ and $\nabla$ is an integrable connection with respect to the differential operator $\widetilde{d}$ (See the following for more explanation). 
    \end{itemize}
    Moreover, when $\alpha=1$, this isomorphism coincides with Heuer's sheafified $p$-adic Simpson correspondence (cf. \cite[Theorem 1.2]{heuer2022moduli}). 
\end{theo}

Here, the notation $\{-1\}$ is a generalization of Breuil--Kisin twist.

\begin{defi}\label{definition: BKgeneral}
    For any sheaf of $\Bp{\alpha}$ modules $\mathcal{F}$ on any topology and an integer $k$, define: 
    $$\mathcal{F}\{k\}=\mathcal{F}\otimes_{\bB_{\dR}^+(K)}(\ker\theta)^{k},$$
    (which is isomorphic to $\mathcal{F}$). Moreover, if $k\leq l$ are integers, there is a canonical homomorphism $\mathcal{F}\{l\}\to \mathcal{F}\{k\}$ induced by the canonical inclusion $(\ker\theta)^{l}\B\to (\ker\theta)^k\B$.
\end{defi}

\begin{ex}\label{eg:BKdiff}
    Keep the notations as above, and let $\Omega_{X}$ be the sheaf of differentials of $X$ over $B_\alpha$. Then $\Omega_X\{-1\}$ is a vector bundle. The canonical homomorphism
    $$i:\Omega_X\to \Omega_{X}\{-1\}$$
    (defined as above) is $\inte{X}$-linear with kernel $(\ker\theta)^{\alpha-1}\Omega_X$ and cokernel
    $\Omega_{\overline{X}}\{-1\}$ (the usual Breuil--Kisin twist). Let $\widetilde{d}:\inte{X}\to \Omega_X\{-1\}$ be the differential operator $i\circ d$.
    The similar construction provides a complex
    $$\widetilde{\mathrm{DR}}_{X}=[\inte{X}\xrightarrow{\widetilde{d}}\Omega_X\{-1\}\xrightarrow{\widetilde{d}}\Omega_X^2\{-2\}\to\cdots].$$
\end{ex}

We emphasize that the left hand side of the isomorphism in Theorem \ref{theorem: sheafRH} only depends on $\overline{X}$. The subscript $\alpha$ does not appear explicitly in the right hand side since it is already included in the base ring of $X$.

\subsection{Moduli spaces}

In Simpson's foundational work \cite{Simpson_1994}, he established the moduli spaces for Betti local systems and de Rham local systems on complex smooth projective varieties. He proved that the classical Riemann–Hilbert correspondence induces a complex analytic isomorphism between these two moduli spaces. In the $p$-adic setting, Heuer constructs analogues of these moduli spaces within the framework of $p$-adic nonabelian Hodge theory (cf. \cite[Theorem 1.4]{heuer2022moduli}). 

In this paper, we generalize these $p$-adic analogues within the context of $p$-adic Riemann–Hilbert theory. Again, they are deformations of those moduli spaces in $p$-adic Simpson theory. We will freely use the language of small $v$-stacks introduced by Peter Scholze in \cite{scholze2022etale}.

Fix a positive integer $\alpha$. Let $C$ be a perfectoid field over $\mathbb{Q}_p^\cyc$, and let $\Perf$ be the big $v$-topology of affinoid perfectoid spaces over $C$. Define $\Bp{\alpha}=\Bpp{\alpha}(C)$, and let $X$ be a smooth adic space over $\Bp{\alpha}$. Set $\overline{X}=X\otimes_{\Bp{\alpha}}C$.

\begin{defi}
    (1) For any $\Spa(A,A^+)\in\Perf$ and a positive integer $\alpha$, let $\Bpp{\alpha}(A)^+$ be the preiamge of $A^+$ along the Fontaine's $\theta$ map. Let $\Spa_{\dR,\alpha}(A,A^+)$ be the pre-adic space 
    $$\Spa\left(\Bpp{\alpha}(A),\Bpp{\alpha}(A)^+\right);$$
    this is indeed an adic space (cf. Corollary \ref{etshperfd}).

    (2) For any $\Spa(A,A^+)\in\Perf$, let 
    $$X_A=X\times_{\Spa_{\dR,\alpha}(C,C^+)}\Spa_{\dR,\alpha}(A,A^+).$$
\end{defi}

Following Heuer's method, we define the following moduli prestacks.

\begin{defi}
    For any integer $r>0$.
    
    (1) Let $\M_{\dR,\alpha,r,X}$ be the prestack sending $\mathrm{Spa}(A,A^+)\in \Perf$ to the groupoid of pairs $(M,\nabla)$ where $M$ is a rank $r$ vector bundle on $X_A$ and 
    $$\nabla:M\to M\otimes \Omega\{-1\}$$
    is a connection with respect to $\widetilde{d}$ (see Example \ref{eg:BKdiff}).

    (2) Let $\M_{{\mathbb{B}_{\alpha}},r,X}$ be the prestack sending $\mathrm{Spa}(A,A^+)\in \Perf$ to the groupoid of rank $r$ locally free $\B/(\ker\theta)^\alpha$- sheaves on the $v$-topology of $\overline{X}_{A}$. Here, $\overline{X}=X\otimes_{B_\alpha}C$ and $\overline{X}_A$ is the usual base change.
\end{defi}

We will prove the following theorem.

\begin{theo}\label{theorem: modulismall}
    The prestacks $\M_{\dR,\alpha,r,X}$ and $\M_{{\mathbb{B}_{\alpha}},r,X}$ are small $v$-stacks.
\end{theo}

\subsection{Decompleting toric towers over a de Rham period ring}

The key technique for proving Theorems \ref{theorem: sheafRH} and \ref{theorem: modulismall} is the decompletion of smooth algebras with toric charts over the de Rham period ring. Let $(K,K^+)$ be a perfectoid Tate--Huber pair over $\Q_p^\cyc$, $\alpha$ a positive integer and $X=\Spa(A,A^+)$ a smooth affinoid adic space over $\Bp{\alpha}=\B(K)/(\ker\theta)^\alpha$.

Assume that there exists a standard \'etale homomorphism
$$f:\Bp{\alpha}\langle  T_1^{\pm1},T_2^{\pm1},\dots,T_d^{\pm1}\rangle\to A.$$
Set $\overline{A}=A\otimes_{\Bp{\alpha}}K$ and 
$$\widehat{\overline{A}}_{\infty}=K\langle  T_1^{\pm\tfrac{1}{p^\infty}},T_2^{\pm\tfrac{1}{p^\infty}},\dots,T_d^{\pm\tfrac{1}{p^\infty}}\rangle\widehat{\otimes}_{K\langle  T_1^{\pm1},T_2^{\pm1},\dots,T_d^{\pm1}\rangle}\widehat{\overline{A}}.$$
This is a perfectoid ring. Denote $\widehat{A}_{\infty}=\Bpp{\alpha}(\widehat{\overline{A}}_\infty)$. Choose a system of $p^\infty$-th unit root $\{\zeta_{p^n}:n\geq0\}\in F$ (i.e. $\zeta_{p^n}^p=\zeta_{p^{n-1}}$ and $\zeta_p$ is primitive). There is a canonical action of $\Gamma=\Z_p^d$ on $\widehat{\overline A}_\infty$. This action can be extended to $\widehat{A}_\infty$ by the functorality of $\Bpp{\alpha}(-)$.

In \cite{berger2008familles}, Berger and Colmez introduced a method describing decompletion using locally analytic vectors, which was later analogized to a geometric setting in \cite{Pan_2022}. We first generalize their computations and obtain the following result.

\begin{prop}[Proposition \ref{prop:locallyan}]
    For any $N\geq 0$, there is a canonical isomorphism of rings
    $$\widehat{A}_\infty^{p^N\Gamma-\mathrm{an}}\cong A_N:= \Bp{\alpha}\langle  T_1^{\pm\tfrac{1}{p^N}},T_2^{\pm\tfrac{1}{p^N}},\dots,T_d^{\pm\tfrac{1}{p^N}}\rangle\otimes_{\Bp{\alpha}\langle  T_1^{\pm1},T_2^{\pm1},\dots,T_d^{\pm1}\rangle}A$$
    where the left hand side is the subset of $p^N\Gamma$-analytic vectors (cf. \cite[Subsection 2.1]{Pan_2022}). Moreover, the inclusion of the $p^N\Gamma$-analytic vectors into the $p^{N+1}\Gamma$-analytic vectors corresponds to the natural inclusion $A_N\to A_{N+1}$.
\end{prop}

\begin{rmk}
    In fact, the above isomorphism identifies \( T_i^{\frac{1}{p^N}} \) with the Teichmüller lift of \[ (T_{i}^{\frac{1}{p^N}}, T_{i}^{\frac{1}{p^{N+1}}}, T_i^{\frac{1}{p^{N+2}}}, \dots) \in \widehat{\overline A}_{\infty}^\flat .\]
\end{rmk}

\begin{rmk}
    This proposition induces a locally analytic action of $\Gamma$ on $A_\infty:=\varinjlim A_N$. However, unlike the $\alpha=1$ case, this action generally \underline{\textbf{differs}} from the Galois action of $\Gamma$. This difference is the main obstruction to deforming Heuer's theorem to Theorem \ref{theorem: sheafRH}. See Section \ref{Section:SheafRH} for details.
\end{rmk}

The decompletion theory can be described as the following theorem.

\begin{theo}[Theorem \ref{Sen}]
    There is a canonical rank preserving equivalence between categories
    $${\Rep}_{\Gamma}^{\mathrm{cnt}}(\widehat{A}_{\infty})\cong{\Rep}_{\Gamma}^{\mathrm{la}}(A_{\infty})$$
    where
    \begin{itemize}
        \item ${\Rep}_{\Gamma}^{\mathrm{cnt}}(\widehat{A}_{\infty})$ is the category of finite rank projective $\widehat{A}_\infty$-modules with continuous semi-linear $\Gamma$-actions.
        \item ${\Rep}_{\Gamma}^{\mathrm{la}}(A_{\infty})$ is the category of finite rank projective $A_\infty$-modules with locally analytic semi-linear $\Gamma$-actions.
    \end{itemize}
    Precisely, the functor from the left hand side to the right hand side is taking locally analytic vectors, and its inverse is obtained by taking completion.
\end{theo}

\subsection{Acknowledgement}

The author would like to thank Yupeng Wang and Tian Qiu for helpful discussions. We thank Liang Xiao and Yupeng Wang for carefully reading the preliminary draft and helping the author on writing.

\subsection{Notations}\label{Subsection: Notations}

\noindent (1) For any ring $A$ and elements $g,f_1,f_2,\cdots,f_k\in A$, use $$A\left[\tfrac{f_1,f_2,\cdots,f_k}{g}\right]$$ to denote the $A$-subalgebra of $A[\tfrac{1}{g}]$ generated by $\tfrac{f_1}{g},\tfrac{f_2}{g},\cdots,\tfrac{f_k}{g}$.
$$$$
\noindent (2) For any perfectoid Tate--Huber pair $(R,R^+)$, we will use $\B(R)$ to denote its De Rham period ring. We emphasis that this ring does \textbf{not} depend on the choice of $R^+$. We also use $\mathbb{B}_{\dR,X}^+$ to denote the de Rham period sheaf on $X_{v}$ for any adic space $X$ over $\Q_p$.
$$$$
\noindent (3) For any ring $R$, we will use $W(R)$ to denote it ring of Witt vectors and $W_N(R)$ to denote its ring of $N$-th truncated Witt vectors.

\section{Foundations on adic spaces}\label{Section: Foundation adic spaces}

The main goal of this section is to establish the theories of differential modules, étaleness, and smoothness for general adic spaces. We note that although these theories have been constructed in several previous works, most of them impose some special restrictions (most often requiring the base space to be strongly Noetherian or perfectoid). However, the spaces considered in this article are neither strongly Noetherian nor perfectoid; consequently, we are forced to redevelop the general theory.

\subsection{Foundations on functional analysis over Tate rings}

In this subsection, we fix a complete Tate ring $K$ and a pseudo-uniformizer $\vpi\in K$.

\begin{defi}
    Let $K_0\subseteq K$ be a definition subring containing $\varpi$. A topological $K$-module $M$ is said to be \emph{adic with respect to $K_0$} if there exists an open $K_0$-submodule $M_0\subseteq M$ for which the induced topology coincides with the $\varpi$-adic topology. Such an $M_0$ is called a \emph{definition submodule}. A topological module $M/K$ is called \textit{adic} if there exists a defining subring $K_0 \subseteq K$ such that $M$ is adic with respect to $K_0$.
\end{defi}

One of the most important property of adic modules is that the definition does not depend on the choice of the definition subring. 

\begin{lemma}\label{lem: adic independ K_0}
    Let $K_0,\ K_0'$ be two open bounded subrings of $K$ containing $\varpi$. Then any adic $K$-module with respect to $K_0$ is also an adic module with respect to $K_0'$.
\end{lemma}

\begin{proof}
    Let $M$ be an adic module over $K$ with respect to $K_0$. Choose a definition submodule $M_0$. We claim that there exists some $N$ such that the $K_0'$-submodule generated by $M$ contains in 
    $\vpi^{-N}M_0$.

    Indeed, since $K_0'$ is bounded, there exists some $N$ such that $K_0'\subseteq\vpi^{-N}K_0$. Hence, for any $a_1,a_2,\dots,a_r\in K_0'$ and $m_1,m_2\dots,m_r\in M_0$,
    \[\sum_{j=1}^r\vpi^Na_jm_j\in M_0\]
    since $M_0$ is a $K_0$-module. This implies that the $K_0'$-submodule generated by $M_0$ contains in 
    $\vpi^{-N}M_0$.
\end{proof}

From now on, we fix a definition subring $K_0\subseteq K$ containing $\vpi$.
Denote by $\mathrm{Mod}^{\mathrm{cadic}}_K$ the category of complete adic $K$-modules where morphisms are continuous homomorphisms. Lemma \ref{lem: adic independ K_0} shows that all adic $K$-module is adic with respect to $K_0$.

For any two complete adic modules $M$ and $N$, we write $\Hom_K(M,N)$ for the set of continuous homomorphisms from $M$ to $N$.
The category $\mathrm{Mod}^{\mathrm{cadic}}_K$ is in general not an abelian category, as it does not admit arbitrary cokernels. However, it retains several important properties reminiscent of abelian categories.

\begin{lemma}\label{lemma:foundation linear operators}
    Let $M,N\in \mathrm{Mod}_K^{\mathrm{cadic}}$ and fix definition submodules $M_0\subseteq M,\ N_0\subseteq N$.

    (1) There is a canonical identification $M_0[\tfrac{1}{\vpi}]\cong M$.
    
    (2) Let $f:M\to N$ be a continuous map such that there exists positive integers $m,n$ such that for any $x\in M$,
    $$f(\vpi^mx)=\vpi^nf(x).$$
    Then there exists an integer $r$ such that $f(M_0)\subseteq \vpi^rN_0$.
    
    (3) The natural homomorphism $\Hom_{K_0}(M_0,N_0)[\tfrac{1}{\vpi}]\to \Hom_{K}(M,N)$ is an isomorphism.

    (4) There is a topology on $\Hom_{K}(M,N)$ makes it a topological module such that the subset $$\{f\in \Hom_K(M,N):f(M_0)\subseteq N_0\}$$ is open and carries the \( \varpi \)-adic topology. This topology is independent of the choices of $K_0$, $M_0$ or $N_0$ and is complete.
\end{lemma}

\begin{proof}
    (1) Since $M_0$ is a submodule of $M$, it is $\vpi$-torsion free, the left-hand side is a submodule of the right-hand side. Moreover, $M_0$ is an open submodule in $M$, so for each $m\in M$, the condition $\lim_{n\to \infty}\vpi^nm=0$ implies that there exists some $n>0$ such that $\vpi^nm\in M_0$, showing $M_0[\tfrac{1}{\vpi}]=M$.

    (2) Assume for contradiction that for any $r\in\Z$, $f(M_0)\not\subseteq\vpi^rN_0$. Then we can choose for each $0<r\in\Z$ an element $x_r\in M_0$, such that $f(x_r)\not\in \vpi^{-nr}N_0$. Consider the sequence $\{\vpi^{mr} x_r:r\geq 1\}$: It converges to $0$ but the sequence $\{f(\vpi^{mr}x_r)\}$ dose not, contradicting the continuity of $f$.

    (3) Consider the natural map
    $$F:\Hom_{K_0}(M_0,N_0)[\tfrac{1}{\vpi}]\to \Hom_{K}(M,N).$$
    The injectivity of $F$ follows from the canonical map $\Hom_{K_0}(M_0,N_0)\to\Hom_K(M,N)$ is injective (by definition). 
    
    For surjectivity, it is equivalent to prove that for any $f:M\to N$ which is $K$-linear and continuous, there exists an $N\in \mathbb{Z}$ such that $f(M_0)\subseteq \vpi^{N}N_0$, which is a consequence of (2).

    (4) First, we verify that the definition of this topology is independent of the choices of \(K_0\), \(M_0\), or \(N_0\).

Let \(K_0'\) be another definition subring of $K$ containing \(\varpi\), \(M_0' \subseteq M\) and \(N_0' \subseteq N\) be definition submodules with respect to \(K_0'\). Let \(K_0'' = K_0 \cap K_0'\), then \(M_0, M_0'\) are both definition submodules of \(M\) with respect to \(K_0''\). Similarly, \(N_0, N_0'\) are both definition submodules of \(N\) with respect to \(K_0''\). Note that as submodules of \(\operatorname{Hom}_K(M, N)\), we have
\[
\operatorname{Hom}_{K_0}(M_0, N_0) = \{f \in \operatorname{Hom}_K(M, N) : f(M_0) \subseteq N_0\} = \operatorname{Hom}_{K_0''}(M_0, N_0).
\]
Similarly,
\[
\operatorname{Hom}_{K_0'}(M_0', N_0') = \operatorname{Hom}_{K_0''}(M_0', N_0').
\]
Therefore, we can assume \(K_0 = K_0'\).
    
    It suffices to prove that for any definition submodules $M_0'\subseteq M$ and $N_0'\subseteq N$, as submodules of $\Hom_K(M,N)$, there exists an integer $r$, such that 
    $$\vpi^r\Hom_{K_0}(M_0,N_0)\subseteq \Hom_{K_0}(M_0',N_0').$$
    In fact, we can choose $r$ such that $\vpi^rM_0\subseteq M_0'$ and $\vpi^rN_0'\subseteq N_0$, then by definition,
    $$\vpi^{2r}\Hom_{K_0}(M_0,N_0)= \Hom_{K_0}(\vpi^{-r}M_0,\vpi^rN_0)\subseteq \Hom_{K_0}(M_0',N_0').$$
    This proves that the topology does not depend on the choice of $K_0$, $M_0$ or $N_0$. 
    
    For completeness, we only need to prove that $\Hom_{K_0}(M_0,N_0)$ is $\vpi$-adically complete. This follows from the fact that $M_0$, $N_0$ are $\vpi$-adically complete and $\vpi$-torsion free.
\end{proof}

\begin{rmk}
    The topology on $\Hom_K(M,N)$ is the generalization of the strong topology in the functional analysis.
\end{rmk}

\begin{prop}\label{prop: complete tensor product}
    Let $M,N$ be objects in $\mathrm{Mod}_K^{\mathrm{cadic}}$. There exists an object $L\in\mathrm{Mod}_K^{\mathrm{cadic}}$ as well as a continuous bilinear map $u:M\times N\to L$ satisfying the universal property:
    \begin{itemize}
        \item Let $P\in\mathrm{Mod}_K^{\mathrm{cadic}}$ and $\phi:M\times N\to P$ be a continuous bilinear map, there exists a unique continuous linear homomorphism $f:L\to P$ such that $\phi=f\circ u$.
    \end{itemize}
\end{prop}

\begin{proof}
    Let $M\otimes_K N$ be the usual tensor product. Fix definition submodules $M_0\subseteq M$ and $\ N_0\subseteq N$ and define $(M\otimes_K N)_0$ to be the image of $M_0\otimes_{K_0} N_0$ in $M\otimes_KN$. Let $L=\widehat{(M\otimes_{K}N)_0}[\tfrac{1}{\vpi}]$ and the bilinear map $u:M\times N\to L$ be the natural one. 

    We check that $u:M\times N\to L$ satisfies the universal property. Let $P\in \mathrm{Mod}^{\mathrm{cadic}}_K$ and $\phi:M\times N\to P$ be a continuous bilinear map. By Lemma \ref{lemma:foundation linear operators} (2), there exists $P_0\subseteq P$ which is a definition submodule such that $\phi(M_0\times N_0)\subseteq P_0$. This induces a natural homomorphism $f_0:M_0\otimes_{K_0}N_0\to P_0$. Taking the completion of $f_0$ and inverting $\vpi$, we obtain the required $f:L\to P$. The uniqueness follows from the fact that the image of $M\otimes_KN$ (usual tensor product) in $L$ is dense.
\end{proof}

\begin{defi}
    For any $M,N\in\mathrm{Mod}_K^{\mathrm{cadic}}$, $x\in M$ and $y\in N$, we will use $M\widehat{\otimes}_{K}N$ as well as $x\widehat{\otimes}y$ to denote the object $L$ and $u(x,y)$ in Proposition \ref{prop: complete tensor product}.
\end{defi}

The complete tensor products also have the similar adjoint property with general tensor products.

\begin{lemma}[Adjoint properties of complete tensor products]\label{lemma: adj tensor}
    (1) Let $M,N,P$ be three objects in $\mathrm{Mod}_K^{\mathrm{cadic}}$. There is a natural $K$-linear continuous isomorphism
    $$\Hom_{K}(M\widehat{\otimes}_KN,P)\cong \Hom_K\left(M,\Hom_K(N,P)\right),$$
    where the $\Hom_K(-,-)$ are all topologized in the way in Lemma \ref{lemma:foundation linear operators}.

    (2) Let $A$ be a complete Tate algebra over $K$ and $M\in\mathrm{Mod}_K^{\mathrm{cadic}}$. Then there is a canonical $A$-module structure on $A\widehat{\otimes}_KM$ which makes $A\widehat{\otimes}_KM$ an object in $\mathrm{Mod}_{A}^{\mathrm{cadic}}$. Moreover, for any $N\in\mathrm{Mod}_{A}^{\mathrm{cadic}}$ there is a caonical topological isomorphism between complete adic $A$-modules:
    $$\Hom_{A}(A\widehat{\otimes}_KM,N)\cong \Hom_{K}(M,N).$$
\end{lemma}

\begin{proof}
    This can be proved in the same way as the proof to the usual tensor products. See for example, \cite[Chapter 2]{Atiyah_2018}.
\end{proof}

An important example of the complete tensor products is the rational localization of a complete adic module.

\begin{lemma}\label{lemma: topology on localisation}
    Let $g,f_1,f_2,\dots f_n\in K$ which generate the unit ideal, and $M\subseteq \mathrm{Mod}_K^{\mathrm{cadic}}$. Fix $M_0$ a definition submodule of $M$.
    Let $M_0[\tfrac{f_1,\dots,f_n}{g}]$ be the $K_0[\tfrac{f_1,\dots,f_n}{g}]$-submodule of $M[\tfrac{1}{g}]$ generated by $M_0$. Then there is a topological $K$-module structure of $M[\tfrac{1}{g}]$ makes $M_0[\tfrac{f_1,\dots,f_n}{g}]$ open and $\vpi$-adically topologized.
\end{lemma}

\begin{proof}
    Since $M_0[\tfrac{f_1,\dots,f_n}{g}]$ is a $K_0$-submodule of $M$, it suffices to prove that $$M_0[\tfrac{f_1,\dots,f_n}{g}][\tfrac{1}{\vpi}]=M.$$

    By the assumption, there exists $h_0,h_1,\dots,h_n\in K_0$ and a positive integer $N$ such that $h_0g+\sum_{i=1}^nh_if_i=\vpi^N$. Hence
    $$h_0+\sum_{i=1}^nh_i\tfrac{f_i}{g}=\tfrac{\vpi^N}{g}$$
    holds in $A[\tfrac{1}{g}]$. This implies that we only need to check
    $$M_0[\tfrac{\vpi^N}{g}][\tfrac{1}{\vpi}]=M[\tfrac{1}{g}],$$
    which holds obviously.
\end{proof}

\begin{defi}
    Let $g,f_1,f_2,\dots f_n\in K$ which generate the unit ideal, and $M\subseteq \mathrm{Mod}_K^{\mathrm{cadic}}$. Define the \emph{rational localization} to be the completion of $M[\tfrac{1}{g}]$ under the topology defined in Lemma \ref{lemma: topology on localisation}. Denote it by 
    $$M\langle \tfrac{f_1,f_2,\dots,f_n}{g} \rangle.$$
\end{defi}

\begin{lemma}\label{lemma: adj localization}
    Let $M$ be an object in $\mathrm{Mod}_K^{\mathrm{cadic}}$ which is equipped with a definition submodule $M_0$ and $N$ be an object in $\mathrm{Mod}_{K\langle 
    \tfrac{f_1,f_2,\dots,f_n}{g} \rangle}^{\mathrm{cadic}}$. Then each $K$-linear continuous homomorphism from $M$ to $N$ extends uniquely to a $K$-linear continuous homomorphism $M\langle 
    \tfrac{f_1,f_2,\dots,f_n}{g} \rangle\to N$. Moreover, this extension is $K\langle 
    \tfrac{f_1,f_2,\dots,f_n}{g} \rangle$-linear.
\end{lemma}

\begin{proof}
    Let $N$ be an object in $\mathrm{Mod}_{K\langle \tfrac{f_1.f_2,\dots,f_n}{g} \rangle}^{\mathrm{cadic}}$ and $N_0$ is a definition submodule as $K\langle 
    \tfrac{f_1,f_2,\dots,f_n}{g} \rangle$-module. Then by definition, $N_0$ is also a definition submodule as a $K$-module. Let $\phi:M\to N$ be a $K$-linear continuous homomorphism. This extends uniquely to a $K$-linear homomorphism $M[\tfrac{1}{g}]\to N$ since $g$ is invertible in $K\langle 
    \tfrac{f_1,f_2,\dots,f_n}{g} \rangle$. Still denote this extension by $\phi$. Moreover, we can choose $N_0$ such that $\phi(M_0)\subseteq N_0$ by Lemma \ref{lemma:foundation linear operators} (2). Hence, 
    $$\phi\left(M_0[\tfrac{f_1,f_2,\dots,f_n}{g}]\right)\subseteq N_0$$
    and $\phi$ extends to $M\langle\tfrac{f_1,f_2,\dots,f_n}{g} \rangle$ by taking completion. The uniqueness comes from the fact that $M[\tfrac{1}{g}]$ is dense in $M\langle\tfrac{f_1,f_2,\dots,f_n}{g} \rangle$ and the $K\langle\tfrac{f_1,f_2,\dots,f_n}{g} \rangle$ of the extension follows from the construction.
\end{proof}

\begin{cor}
    (1) The definition of the rational localizations does not depend on the choice of definition modules.

    (2) For any $f_1,f_2,\dots,f_n,g\in K$ which generate the unit ideal as well as $M\in \mathrm{Mod}_K^{\mathrm{cadic}}$, there is a canonical isomorphism
    $$M\langle\tfrac{f_1,f_2,\dots,f_n}{g}\rangle\cong K\langle\tfrac{f_1,f_2,\dots,f_n}{g}\rangle\widehat{\otimes}_KM.$$
\end{cor}

\begin{proof}
    (1) is a corollary of (2).

    For (2), this is just a combination of Lemma \ref{lemma: adj tensor} Lemma \ref{lemma: adj localization}.
\end{proof}

\begin{defi}
    For any $M\in\mathrm{Mod}_K^{\mathrm{cadic}}$, define the module
    $$M\langle t\rangle=\widehat{M_0[t]}[\tfrac{1}{\vpi}]=\{\sum_{i=0}^{\infty}m_it^i\subseteq M[[t]]:\lim_{n\to +\infty}m_n=0\}$$
    and
    $$M\langle t^{\pm1}\rangle=\widehat{M_0[t^{\pm1}]}[\tfrac{1}{\vpi}]=\{\sum_{i=-\infty}^{\infty}m_it^i\subseteq M[[t]]:\lim_{n\to \pm\infty}m_n=0\},$$
    where $M_0$ is a definition submodule.

    This definition obviously makes $M\langle t\rangle$ a complete adic $A\langle t\rangle$-module and $M\langle t^{\pm1}\rangle$ a complete adic $A\langle t^{\pm1}\rangle$-module.
\end{defi}

\begin{lemma}
    Let $A$ be a complete Tate-algebra over $K$, then for any $M\in\mathrm{Mod}_A^{\mathrm{cadic}}$,
    $$M\widehat{\otimes}_{K}K\langle t\rangle\cong M\langle t\rangle$$ and $$M\widehat{\otimes}_KK\langle t^{\pm 1}\rangle\cong M\langle t^{\pm 1}\rangle$$
    as topologically $A$-modules.
\end{lemma}

\begin{proof}
    This can be checked by the construction of the complete tensor product.
\end{proof}

Finally, we introduce the open mapping theorem in the setting of complete adic modules. As a first application, we will consider the finitely generated complete adic modules and then topologize all finitely rank projective modules canonically.

\begin{theo}[Open mapping theorem]\label{theo: open mapping theorem}
    Let $f$ be a morphism of complete adic $K$-modules. If $f$ is surjective, then $f$ is open.
\end{theo}

\begin{proof}
    This can be proved in the same way as the proof of the open mapping theorem of Banach spaces. See \cite{henkel2014openmappingtheoremrings} for details.

    We do not prove this theorem here. However, since all Tate rings we will be concerned with are over $\bQ_p$. We can view all complete adic modules as \(\mathbb{Q}_p\)-Banach spaces, in which case the statement reduces to the classical open mapping theorem for Banach spaces. 
\end{proof}

In the following, we call a complete adic module \emph{finitely generated (resp. finitely presented)} if the underlying module is finitely generated (resp. finitely presented).

\begin{cor}\label{cor: linear from finitely generated cadic is continuous}
    Let \(f: M \to N\) be a (not necessarily continuous) \(K\)-linear map between complete adic modules, where \(M\) is finitely generated. Then \(f\) is continuous.
\end{cor}

\begin{proof}
     Choose a surjective $K$-linear map $p:F\to M$ where $F$ is a finite rank free module.
    Note that for any topological $K$-module $E$ and $e\in E$, the map \[l_e(a):=ae \]
    from $K$ to $E$ is continuous by the definition of topological modules. By this fact, both $f$ and $g=f\circ p$ are continuous. By Theorem \ref{theo: open mapping theorem}, $f$ is open. Thus, for any open subset $U\subseteq N$, \[f^{-1}(U)=p(g^{-1}(U))\]
    is open, which implies that $f$ is continuous.
\end{proof}

In particular, we have the following corollary which says that there exists at most one way to topologize a finitely generated module to a complete adic module.

\begin{cor}
    Let $M$ be a finitely generated $K$-module, there exists at most one topology on $M$ making $M$ a complete adic $K$-module.
\end{cor}

\begin{proof}
    Let $M_1$ and $M_2$ be two complete adic module with underlying module $M$. Then $id_M$ is continuous a morphism from $M_i$ to $M_j$ for any choice of $i,j\in\{1,2\}$ by Corollary \ref{cor: linear from finitely generated cadic is continuous}. This implies that $M_1$ and $M_2$ are isomorphic as topological modules.
\end{proof}

For a finitely generated $K$-module, we say $M$ is \emph{complete adic} if there exists a (hence, unique) way to topologize $M$ as a complete adic module.

In general, not all finitely generated modules are complete adic. For example, the quotient of $K$ by a non-closed ideal is not complete adic. However, we will see that a direct summand of a finitely generated complete adic module is itself complete adic.

\begin{theo}
    Let $M$ be a finitely generated complete adic $K$-module, and let $M = N_1 \oplus N_2$ as $K$-modules. Then both $N_1$ and $N_2$ are complete adic.
\end{theo}

\begin{proof}
    Let $m_1,m_2,\dots,m_r$ be generators of $M$. For any $i$, let $m_i=p_i+q_i$ where $p_i\in N_1$ and $q_i\in N_2$. Obviously, $N_1$ is generated by $p_1,p_2,\dots,p_r$ and $N_2$ is generated by $q_1,q_2\dots,q_r$. Put
    \[N_1^\circ=K_0p_1+K_0p_2+\dots+K_0p_r;\ N_2^\circ=K_0q_1+K_0q_2+\dots+K_0q_r\]
    and $M^\circ=N_1^\circ\oplus N_2^\circ$ which is a finitely generated $K_0$-module.
    By Theorem \ref{theo: open mapping theorem} and Lemma \ref{lemma:foundation linear operators} (2), $M^\circ$ is a definition module; hence
    $N_1^\circ\oplus N_2^\circ$ is $\varpi$-adic complete. Consequently, both $N_1^\circ$ and $N_2^\circ$ are $\varpi$-adically complete, and the theorem follows.
\end{proof}

As a corollary, all finite rank projective modules are complete adic.

\subsection{The \'etale morphisms and the \'etale topology of adic spaces}

The material in this section is taken mainly from \cite[Subsection 8.2]{KEDLAYA_2018} and \cite[Appendix C]{Zavyalov_2025}.

All Tate rings are assumed to be complete. 

\begin{defi}[Finite \'etale morphisms]
    Let $f:(A,A^+)\to (B,B^+)$ be a morphism of Tate--Huber pairs. We say that $f$ is \emph{standard étale} if $f$ is a finite composition of rational localizations and finite étale maps.
\end{defi}

\begin{rmk}
    As we assume that $A$ is complete, then for any finite \'etale $A$ algebra $B$, there exists a unique topology on $B$ making $B$ a topologically $A$-algebra. Hence, the category of finite \'etale Tate--Huber pairs over $(A,A^+)$ is equal to the category of finite \'etale algebras over $A$.
\end{rmk}

\begin{defi}[Standard \'etale morphisms]\label{defi: standard etale}
     Let $f:(A,A^+)\to (B,B^+)$ be a morphism between complete Tate--Huber pairs. Call $f$ to be \emph{standard \'etale} if $f$ is a finite composition of rational localizations and finite \'etale maps.    
\end{defi}

\begin{rmk}
    The name `standard étale' is taken from \cite{Liu_2016}. In \cite{KEDLAYA_2018}, this is referred to as strictly étale; in \cite{Zavyalov_2025}, it is called strongly étale. 
\end{rmk}

It is clear that standard étaleness is stable under base change. The notion of standard étaleness is sufficiently robust that it also satisfies the following 2-out-of-3 property.

\begin{theo}\label{theo: 2outof3 st etale}
    Let $f:(A,A^+)\to (B,B^+)$ and $g:(B,B^+)\to (C,C^+)$ be morphisms of Tate--Huber pairs and $h:=g\circ f$.
    \begin{enumerate}[(1)]
        \item If both $f$ and $g$ are standard \'etale, so is $h$.
        \item If both $f$ and $h$ are standard \'etale, so is $g$.
    \end{enumerate}
\end{theo}

\begin{proof}
    The first claim is obvious by definition. We only need to prove the second one.

    By definition, there exist sequences of morphisms
    \[A=B_0\xrightarrow[]{f_0}B_1\xrightarrow[]{f_1}\dots\xrightarrow[]{f_{m-1}}B_m=B\]
    and
    \[A=C_0\xrightarrow[]{h_0}C_1\xrightarrow[]{h_1}\dots\xrightarrow[]{h_{n-1}}C_n=C\]
    such that each $f_i,h_j$ is either a rational localization or finite \'etale. Proceeding by induction on $m$ (the length of the sequence for $f$), it suffices to prove that the composition
    \[g_1:=g\circ f_{m-1}\circ f_{m-2}\circ\dots\circ f_1:B_1\to C\]
    is standard \'etale.

Consider the composition
\[B_1\xrightarrow[]{\alpha:=id\otimes 1} B_1\widehat{\otimes}_AC\xrightarrow[]{\beta:=g_1\otimes id}C.\]
Clearly, this composition coincides with $g$. Hence, it suffices to prove that both $\alpha$ and $\beta$ are standard étale. Since $\alpha$ is a base change of $g_1$, it is standard étale. If $f_0$ is a rational localization, then the universal property of rational localizations implies that $\beta$ is an isomorphism. If $f_0$ is finite étale, then $\alpha$ is finite étale. Since $\beta$ is a base change of $\alpha$ (by an automorphism of the target), it is also finite étale by standard facts in commutative algebra. 
\end{proof}

Consequently, fix a Tate--Huber pair $(A,A^+)$ and put $X=\Spa(A,A^+)$, the category of all standard \'etale Tate--Huber pairs over $(A,A^+)$ forms a topology, whose coverings are defined to be families of quasi-compact jointly surjective morphisms. This topology is denoted by $X_{\text{s-\'et}}$. Let $\inte{X,\text{s-\'et}}$ (resp. $\inte{X,\text{s-\'et}}^+$) be the presheaf on $X_{\text{s-\'et}}$ sending each $\Spa(B,B^+)$ to $B$ (resp. $B^+$). In the following, we will call $\inte{\text{s-\'et}}$ the \emph{structure presheaf}.

\begin{defi}
    Let $(A,A^+)$ be a Tate--Huber pair. We call $(A,A^+)$ \emph{\'etale sheafy} if the structure presheaf is a sheaf. If $A$ is a Tate ring and $(A, A^{+})$ is étale sheafy for every choice of $A^{+}$, then $A$ itself is called \emph{étale sheafy}. \footnote{Indeed, for a Tate ring $A$, if $(A,A^+)$ is \'etale sheafy for one choice $A^+$, it is \'etale sheafy as a Tate ring. But we will not use this fact.}
\end{defi}

By Theorem \ref{theo: 2outof3 st etale}, if $(A,A^+)$ is \'etale sheafy, then any standard \'etale Tate--Huber pairs over $(A,A^+)$ is \'etale sheafy.

An important observation is that we can prove the \'etale sheafiness by embedding into some algebras which are already known to be \'etale sheafy.

\begin{defi}\label{definition: sous-etale sheafy}
    We make the following definitions:
    \begin{enumerate}[(1)]
        \item Let $A$ be a complete Tate ring. We say that $A$ is \emph{sous-\'etale sheafy} if there exists a closed embedding of topological algebras $A\subset B$ such that:
    \begin{itemize}
        \item $A\subset B$ is a direct summand as topological $A$-modules.

        \item $B$ is \'etale sheafy.
    \end{itemize}

        \item Let $A$ be a complete Tate algebra. Then $A$ is called \emph{strongly \'etale (resp. sous-\'etale) sheafy} if $A\langle T_1,T_2,\dots,T_n\rangle$ is \'etale (resp. sous-\'etale) sheafy for any $n$.
    \end{enumerate}
\end{defi}

\begin{rmk}
    A closely related notion is that of \emph{sous-perfectoid} algebras; see, for example, \cite[Section 6.3]{Scholze_2020}.
\end{rmk}

\begin{prop}
    Let $A$ be a sous-\'etale sheafy complete Tate algebra. Then the following hold:
    \begin{enumerate}[(1)]
        \item If $B$ is a standard \'etale Tate algebra over $A$, then $B$ is sous-\'etale sheafy.

        \item $A$ is \'etale sheafy
    \end{enumerate}
\end{prop}

\begin{proof}
    (1) follows immediately from the definition.

    (2) Let $A \to \widetilde{A}$ be an embedding as in Definition \ref{definition: sous-etale sheafy}. Let $\widetilde{A}=A\oplus M$ be a decomposition of topological $A$ modules. Let $\{A\to A_i:i=1,2,\dots,n\}$ be a standard \'etale covering, then by the sheafiness of $\widetilde{A}$, the \v Cech sequence
    \begin{equation}\label{eq:sous etale sheafy}0\to \widetilde{A}\to \prod_{i=1}^{n} A_i\widehat{\otimes}_A\widetilde{A}\to\prod_{i,j=1}^{n} A_i\widehat{\otimes}_AA_j\widehat{\otimes}_A\widetilde{A}\end{equation}
    is exact. Therefore, the sequence 
    $$0\to A\to \prod_{i=1}^{n} A_i\to\prod_{i,j=1}^{n} A_i\widehat{\otimes}_AA_j$$
    is exact since it is a direct summand of (\ref{eq:sous etale sheafy}).
\end{proof}

We globalize the above definitions to adic spaces.

\begin{defi}
    We say that an adic space $X$ is \emph{\'etale sheafy} (resp. \emph{strongly \'etale sheafy}) if there exists an open covering $X=\cup_{i\in I}X_i$ such that each $X_i$ is isomorphic to some $\Spa(A_i,A_i^+)$, where
    $(A_i,A_i^+)$ is a \'etale sheafy (resp. strongly \'etale sheafy) Tate--Huber pair.
\end{defi}

In this paper, all adic spaces we consider are \'etale sheafy. Hence, we will only define the \'etale topology of \'etale sheafy adic spaces.

\begin{defi}\label{defi: etale morphism and topology}
    Let $X$ be an \'etale sheafy adic spaces.
    \begin{enumerate}[(1)]
        \item A morphism $f:Y\to X$ of adic spaces is called \emph{locally \'etale} if for any point $y\in Y$, there exists an open neighborhood $U$ of $y$ and an open neighborhood $V$ of $f(y)$ such that the following conditions hold:
        \begin{itemize}
            \item $f(U)\subseteq V$;
            \item Both $U$ and $V$ are isomorphic to some $\Spa(A,A^+)$ where $(A,A^+)$ is an \'etale sheafy Tate--Huber pair.
            \item The induced homomorphism of Tate--Huber pairs \[\left(\calO(V),\calO^+(V)\right)\to \left(\calO(U),\calO^+(U)\right)\] is standard \'etale.
        \end{itemize}
        \item The \emph{\'etale topology} of $X$ is defined as the topology whose underlying category consists of locally \'etale morphisms $f:Y\to X$ of adic spaces, and whose coverings are families of quasi-compact jointly surjective morphisms. Denote by $X_{\et}$ the \'etale topology of $X$.
    \end{enumerate}
\end{defi}

\begin{rmk}
We need to check that arbitrary finite products exist in the category $X_{\et}$. Locally, these are defined via complete tensor products. For general cases, one glues the local constructions, which is a technique-free but heavy imitation of the argument for schemes.
\end{rmk}

\begin{rmk}
    We do not impose any separateness condition here. That is the reason why we use `locally \'etale' instead of `\'etale'.
\end{rmk}

\begin{rmk}
    By Theorem \ref{theo: 2outof3 st etale}, for any \'etale sheafy Tate adic space $X$, an adic space that is locally \'etale over $X$ is also \'etale sheafy.
\end{rmk}

Finally, we prove that standard \'etale homomorphisms satisfy an analogue of the formally \'etale property.

\begin{theo}\label{theo: formally etale of standard etale}
    Let $A\to B$ be a standard \'etale homomorphism of Tate rings. Assume that we are given a diagram of Tate rings
    \[\begin{tikzcd}
        B\ar[r] & C/I\\ A\ar[u]\ar[r] &C\ar[u]
    \end{tikzcd},\]
    where $I\subseteq C$ is a closed ideal with $I^2=0$.
    Then there exists a unique arrow $B\to C$ making the diagram commute.
\end{theo}

\begin{proof}
    It suffices to consider the cases where $B$ is either a rational localization of $A$ or finite \'etale over $A$. The first case follows from the universal property of rational localizations. The second case is the usual property of finite \'etale homomorphisms.
\end{proof}

\begin{rmk}
    By Theorem \ref{theo: 2outof3 st etale}, for any \'etale sheafy Tate adic space $X$ any adic space which is locally \'etale over $X$ is also \'etale sheafy.
\end{rmk}

As a corollary, for any Tate ring $A$ and a closed ideal $I\subseteq A$ such that $I^n=0$ for some $n$, the category of standard \'etale algebras over $A$ is canonically equivalent to the category of standard \'etale alegbras over $A/I$.

\subsection{Smooth morphisms and complete differential modules}

The following definition takes from \cite[Definition 1.6.1]{Huber_1996}.

\begin{defi}
    Let $A\to B$ be a continuous morphism. For any $M\in \mathrm{Mod}^{\text{cadic}}_B$, a \emph{continuous $A$ linear derivation} from $B$ to $M$ is an $A$-linear derivation from $B$ to $M$ which is continuous. Denote by $\mathrm{Der}_A^{\text{ctn}}(B,M)$ the set of all continuous $A$-linear derivations from $B$ to $M$.
\end{defi}

\begin{theo}
    Let $f:A\to B$ be a morphism of complete Tate rings. The functor sending each $M\in\mathrm{Mod}^{\mathrm{cadic}}_B$ to the set $\mathrm{Der}_A^{\mathrm{ctn}}(B,M)$ is representable. We denote by $\widehat{\Omega}_{B/A}$ this universal object and call it \emph{the complete differential module}.
\end{theo}

\begin{proof}
    Let $B_0\subset B$ be a bounded open subring and $A_0\subset f^{-1}(B_0)$ be a bounded open subring. Let $\varpi\in A_0$ be a pseudo-uniformizer. Define $\Omega^+_{B_0/A_0}$ as the sub-$B_0$-module of $\Omega_{B/A}$, the usual differential module, generated by the set $\{dx:x\in B_0\}$. Define $\widehat\Omega_{B/A}$ as the module $\widehat{\Omega_{B_0/A_0}^{+}}[\frac{1}{\varpi}]$, where the completion is the $\varpi$-adic completion. Endow $\widehat{\Omega}_{B/A}$ with the usual topology and it becomes a complete adic $B$-module. By definition, there exists a derivation $d$ from $B_0$ to $\Omega_{B_0/A_0}^+$ and $d$ can extend continuously to a derivation from $B$ to $\widehat{\Omega}_{B/A}$. Denote by $d$ this extension from $B$ to $\widehat{\Omega}_{B/A}$.

    We need to check that such $\widehat\Omega_{B/A}$ satisfies the required universal property. Let $M$ be a complete adic $B$-module and $\partial:B\to M$ be a continuous $A$-linear derivation. By the universal property of the usual differential module, there exists a unique $B$-linear homomorphism $F:\Omega_{B/A}\to M$ such that $\partial=F\circ d$. We only need to prove that $F$ factors through a unique continuous homomorphism $\widehat{F}:\widehat{\Omega}_{B/A}\to M$. The uniqueness is obvious, as the image of $\Omega_{B/A}$ in $\widehat{\Omega}_{B/A}$ is dense.

    It remains to prove the existence. Let $M_0\subseteq M$ be a sub-$B_0$-module of definition. As $\partial$ is continuous, we can change $M_0$ by some $\varpi^{-n}M_0$ so that $\partial(B_0)\subseteq M_0$. By the definition of $\Omega_{B_0/A_0}^+$, $F(\Omega_{B_0/A_0}^+)\subseteq M_0$. By taking the $\varpi$-adic completion and then localizing, $F$ extends to a homomorphism from $\widehat{\Omega}_{B/A}$ to $M=M_0[\frac{1}{\varpi}]$. This proves the existence of $\widehat{F}$.
\end{proof}

\begin{ex}
    Let $A$ be a Tate ring. We claim that
    \[\widehat\Omega_{A\langle T_1,T_2,\dots,T_n \rangle/A}=A\langle T_1,T_2,\dots,T_n \rangle dT_1\oplus A\langle T_1,T_2,\dots,T_n \rangle dT_2\oplus \dots\oplus A\langle T_1,T_2,\dots,T_n \rangle dT_n.\]
    Indeed, let $\delta:A\langle T_1,T_2,\dots,T_n \rangle\to M$ be a continuous $A$-linear derivation. One can easily verify that
    \[\delta(f)=\sum_{j=1}^n \frac{\partial f}{\partial T_j} \delta(T_j).\]
\end{ex}

\begin{theo}\label{theo:stet base change dif module}
    Let $f:A\to B$ be a morphism of Tate rings, and $g:B\to C$ be either a rational localization or a finite \'etale morphism of Tate rings. Put $h=g\circ f$. There exists a unique continuous $C$-linear isomorphism $I:C\widehat{\otimes}_{B}\widehat{\Omega}_{B/A}\cong \widehat{\Omega}_{C/A}$ which fits into the following diagram
    \begin{equation}\label{eq: stet dif diagram}\begin{tikzcd}
        B\ar{d}{g} \arrow[r] & \widehat{\Omega}_{B/A}\ar{d}{I(1\otimes-)}\\
        C\ar{r} & \widehat{\Omega}_{C/A}.
    \end{tikzcd}\end{equation}
\end{theo}

\begin{proof}
    By the universal property of $\widehat{\Omega}_{B/A}$, there exists a unique continuous $C$-linear morphism $I$ that fits into diagram (\ref{eq: stet dif diagram}). It remains to prove that $I$ is an isomorphism.

    First, assume that $g$ is finite \'etale. Let $A_0,B_0,C_0$ be open bounded subrings of $A,B,C$ such that $f(A_0)\subset B_0$ and $g(B_0)\subset C_0$. Furthermore, since $C$ is a finite projective $B$-module, we can choose $C_0$ so that $C_0$ is finitely generated as a $B_0$-module. Fix a pseudo-uniformizer $\vpi\in A_0$. By standard commutative algebra, there exists a canonical map
    \[I_0:C_0\otimes_{B_0}\Omega_{B_0/A_0}\to \Omega_{C_0/A_0}.\]
    Note that $I_0[\frac{1}{\vpi}]$ is an isomorphism since $g$ is finite \'etale; moreover, $I_0$ induces an injection on the torsion-free parts. The cokernel of $I_0$ is isomorphic to $\Omega_{C_0/B_0}$, which is a finitely generated $C_0$-module since $C_0$ is finitely generated over $B_0$. Hence, the cokernel of $I$ is killed by some $\vpi^m$. Combining the above two properties of $I_0$, we conclude that the induced morphism $I$ is an isomorphism.

    It remains to prove the case where $g$ is a rational localization. We prove this by checking the universal property directly. Recall from Lemma \ref{lemma: adj tensor} that for any complete adic $B$-module $M$ and complete adic $C$-module $N$,
    \[\mathrm{Hom}_B(M,N)\cong \mathrm{Hom}_{C}(C\widehat{\otimes}_BM,N).\]
    We only need to check that for any complete adic $C$-module $M$ the restriction provides an equivalence between $\mathrm{Der}^{\mathrm{ctn}}_A(C,M)$ and $\mathrm{Der}^{\mathrm{ctn}}_A(B,M)$. This is a direct consequence of the next lemma.
\end{proof}

\begin{lemma}
    Let $A$ be a Tate ring and $B=A\langle\frac{f_1,f_2,\dots,f_n}{g}\rangle$ be a rational localization. Let $M$ be a complete adic module over $B$, then the restriction induces a bijection from $\mathrm{Der}^{\mathrm{ctn}}(B,M)$ and $\mathrm{Der}^{\mathrm{ctn}}(A,M)$.
\end{lemma}

\begin{proof}
    Since the image of $A[\frac{1}{g}]$ is dense in $B$, we only need to check that any continuous derivation $d$ from $A$ to $M$ can be continuously extended to $B$. By standard results in commutative algebra, $d$ extends to a derivation from $A[\frac{1}{g}]$ to $M$. Let $B_0$ be a definition subring containing all $\frac{f_i}{g}$ and $M_0$ be a definition $B_0$-module. Let $m_i=d(\frac{f_i}{g})$ for any $i$. There exists $N\geq 0$ such that $\vpi^Nm_i\in M_0$ for all $i$, hence we can change $M_0$ to $\vpi^{-N}M_0$ so that $m_i\subset M_0$. Hence $d$ sends $A_0[\frac{f_1,f_2,\dots,f_n}{g}]$ to $M_0$. Taking $\vpi$-adic completion and inverting $\vpi$, we obtain a continuous extension of $d$ to $B$.
\end{proof}

\begin{cor}\label{cor: st smooth free dif mod}
    Let $f:A\to B$ be a morphism of Tate rings such that $f$ factors as
    \[A\to A\langle T_1,T_2,\dots,T_d \rangle\xrightarrow[]{g} B,\]
    where the first arrow is the canonical inclusion and $g$ is standard \'etale.
    Then $\widehat{\Omega}_{B/A}$ is a free $B$-module with basis $\{dT_i:i=1,2,\dots,d\}$.
\end{cor}

\begin{proof}
    When $g$ is an isomorphism, this is proved by the construction of complete differential modules, For general $g$, this is proved by Theorem \ref{theo:stet base change dif module}.
\end{proof}

In this paper, we restrict our attention to differential modules in the context of smooth morphisms. We begin by defining smoothness in our setting.

\begin{defi}\label{defi: standard smooth}
    Let $f:(A,A^+)\to (B,B^+)$ be a morphism of Tate--Huber pairs. Call $f$ \emph{standard smooth} if $f$ factors as
    \[(A,A^+)\to (A\langle T_1,T_2,\dots,T_d \rangle, A^+\langle T_1,T_2,\dots,T_d \rangle)\xrightarrow[]{g} (B',B'^+),\]
    where the first arrow is the canonical inclusion and $g$ is standard \'etale.
\end{defi}

\begin{defi}\label{defi: smooth+st smooth}
    Let $f:(A,A^+)\to (B,B^+)$ be a morphism of \'etale sheafy Tate--Huber pairs. We say that $f$ is smooth if for any point $x\in\Spa(B,B^+)$, there exists a rational localization $x\in U=\Spa(B',B'^+)$ such that the induced homomorphism $(A,A^+)\to (B',B'^+)$ is standard smooth.
\end{defi}

\begin{ex}\label{ex: rational + stsmooth}
    Let $f:(A,A^+)\to (A_1,A_1^+)$ be a standard \'etale morphism and $g:(A_1,A_1^+)\to (B,B^+)$ a standard smooth morphism. Then the composition $g\circ f$ is standard smooth. Indeed, consider a factorization of $g$
    \[A_1\xrightarrow[]{\iota} A_1\langle T_1,T_2,\dots,T_d \rangle\xrightarrow[]{h} B,\]
    where the first arrow is the canonical inclusion and $h$ is standard \'etale.
    Then $g\circ f$ has a factorization
    \[A\to A\langle T_1,T_2,\dots,T_d \rangle\to A_1\langle T_1,T_2,\dots, T_d \rangle \xrightarrow[]{h} B,\]
    where the second arrow, induced by $f$, is standard \'etale by definition.
\end{ex}

The above theorems can also be generalized to Tate adic spaces.

\begin{defi}
    Let $X$ be an \'etale sheafy adic space and $f:Y\to X$ a morphism of adic spaces. Then $f$ is called \emph{locally smooth} if for any $y\in Y$ with $x=f(y)\in X$, there exist affinoid open neighborhoods $U=\Spa(B,B^+)$ of $y$ and $V=\Spa(A,A^+)$ of $x$ such that $f(U)\subseteq V$ and the induced morphism of Tate--Huber pairs \[\widetilde{f}:(A,A^+)\to (B,B^+)\] is standard smooth.
\end{defi}

\begin{rmk}
    Note that by definition, if $X$ is a strongly \'etale sheafy adic space and $Y$ is locally smooth over $X$, then $Y$ is \'etale sheafy.
\end{rmk}

Finally, we define the sheaf of differentials of a smooth morphism. In scheme theory, this can be defined directly from the diagonal morphism. However, in our setting, due to the lack of fiber products \footnote{If we assume that all adic spaces are strongly \'etale sheafy, it is possible to define fiber products of smooth morphisms. We will not use this fact in this paper}, we define it by gluing universal differential modules.

First, we define the module of differentials when the target is affinoid.

\begin{defi}\label{defi: affinoid tar dif mod sheaf}
    Let $X=\Spa(A,A^+)$, where $(A,A^+)$ is an \'etale sheafy Tate--Huber pair, and let $f:Y\to X$ be a smooth morphism between \'etale sheafy adic spaces. Define the sheaf of differentials as the sheafification of the presheaf that assigns to each affinoid object $\Spa(B,B^+)\in Y_{\et}$ the module $\widehat{\Omega}_{B/A}$.

    We will use $\Omega_{Y/X}$ to denote this sheaf.
\end{defi}

As we hope, the sheaf of differentials is locally free.

\begin{theo}\label{theo: local on target dif mod is free}
    Following the notation in Definition \ref{defi: affinoid tar dif mod sheaf}, there exists an analytic cover $\bigcup_{\lambda\in\Lambda}U_\lambda=Y$ such that the restriction ${\Omega}_{Y/X}|_{U_\lambda}$ is a free $\mathcal{O}_{Y,\et}$-module.
\end{theo}

\begin{proof}
    Let $y\in Y$ be a point. We need to prove that there exists a neighborhood of $y$ where ${\Omega}_{Y/X}$ is free. By definition, there exist affinoid neighborhoods $U$ of $y$ and $V$ of $f(y)$ such that $f(U)\subseteq V$ and the restriction $f:U\to V$ is induced by a standard smooth morphism of Tate--Huber pairs.

    Let $V'$ be a rational open neighborhood of $f(y)$ contained in $V$ and let $U'$ be the preimage of $V'$ in $U$. Then the induced morphism $f:U'\to V'$ is a rational localization of the morphism $f:U\to V$. Hence, the induced map $f:U'\to V'$ is standard smooth. By Example \ref{ex: rational + stsmooth}, the restriction $f:U'\to V$ is induced by a standard smooth morphism between Tate--Huber pairs. By Corollary \ref{cor: st smooth free dif mod}, the restriction ${\Omega}_{Y/X}$ on $U'$ is free.
\end{proof}

\begin{lemma}\label{lemma: compare dif mod rational on target}
    Let $X=\Spa(A,A^+)$, where $(A,A^+)$ is an \'etale sheafy Tate--Huber pair. Let $X'=\Spa(A_1,A_1^+)$ be a rational localization of $X$, and let $f:Y\to X$ be a smooth morphism between \'etale sheafy adic spaces. Then
    \[ {\Omega}_{Y/X}\cong  {\Omega}_{Y/X'}.\]
\end{lemma}

\begin{proof}
    We can check this locally on $Y$. Hence, we can assume that $Y=\Spa(B,B^+)$ and that the induced morphism $(A_1,A_1^+)\to (B,B^+)$ is standard \'etale by the proof of Theorem \ref{theo: local on target dif mod is free}. Hence, the induced map from $(A,A^+)$ to $(B,B^+)$ is standard smooth. The lemma follows from Corollary \ref{cor: st smooth free dif mod}.
\end{proof}

Using the above two results, we can define differential modules for general smooth morphisms.

\begin{theo}
    Let $f:Y\to X$ be a smooth morphism between \'etale sheafy adic spaces. Let $U=\Spa(A,A^+)$ and $V=\Spa(B,B^+)$ be two affinoid open subsets of $X$ where both $(A,A^+)$ and $(B,B^+)$ are \'etale sheafy. Then the sheaves $ {\Omega}_{f^{-1}(U)/U}$ and $ {\Omega}_{f^{-1}(V)/V}$ are canonically isomorphic on $U\cap V$. 
\end{theo}

\begin{proof}
    We construct the isomorphism locally on $U\cap V$ and then glue them. For any $x\in U\cap V$, choose a rational open subset $E$ of $U$ which both contains $x$ and is contained in $V$. Choose a rational open subset $E'$ of $V$ which is contained in $E$. Then $E'$ is rational as a open subset of both $U$ and $V$. The isomorphism is given by
    \[ {\Omega}_{f^{-1}(U)/U}|_{E'}\cong  {\Omega}_{f^{-1}(E')/E'}\cong  {\Omega}_{f^{-1}(V)/V}|_{E'}\]
    by Lemma \ref{lemma: compare dif mod rational on target}.

    This isomorphism is canonical and hence can exist globally on $U\cap V$.
\end{proof}

 As a consequence, all vector bundles $ {\Omega}_{f^{-1}(U)/U}$, where $U=\Spa(A,A^+)$ is an affinoid open such that $(A,A^+)$ is \'etale sheafy, can glue to a vector bundle on $Y$. We denote the vector bundle by $ {\Omega}_{Y/X}$ and call it \emph{the sheaf of differentials}.

\subsection{Special cases: strongly noetherian spaces and smoothoid spaces}

In this subsection, we will introduce our definitions of differential modules in two special cases, the strongly noetherian spaces and the smoothoid adic spaces.

\subsubsection*{The strongly noetherian case --} We first recall the following definition.

\begin{defi}
    Let $A$ be a complete Tate ring. Then $A$ is called \emph{strongly noetherian} if for any integer $n$, the Tate ring
    \[A\langle T_1,T_2,\dots,T_n \rangle\]
    is noetherian.
\end{defi}

An important example of strongly noetherian rings is complete non-archimedean fields, which is known as an application of the Weierstrass preparation theorem (cf. \cite[Proposition 14/ Section 2.2]{Bosch_1993}).

\begin{theo}
    Strongly noetherian rings are strongly \'etale sheafy.
\end{theo}

\begin{proof}
    See \cite[(2.2.2)]{Huber_1996}
\end{proof}

In \cite[Section 1.6]{Huber_1996}, Huber defined the differential modules, \'etale morphisms and smooth morphisms for strongly noetherian rings. We first compare Huber's definition with our definition.

\begin{theo}
    Let $X$ be an adic space which is locally isomorphic to $\Spa(A,A^+)$ where $A$ is strongly noetherian. Then the \'etale topos of $X$ in the sense of Huber is equal to the topos of $X_{\et}$ in Definition \ref{defi: etale morphism and topology}.
\end{theo}

\begin{proof}
    See \cite[Proposition 3.2.2]{deJong1996} for the case where $X$ is a rigid analytic space. This proof also works for general strongly noetherian varieties.
\end{proof}

The definitions of differential modules for smooth morphisms are also compatible. Indeed, these two definitions coincide in the strongly noetherian case (cf. \cite[Section 1.6]{Huber_1996}).

We end this part by introducing the following criterion of the noetherianness of a Tate algebra.

\begin{theo}\label{theo: reduction noetherian}
    Let $A$ be a Tate ring which is noetherian as an abstract ring. Then $A$ is strongly noetherian if and only if $A_{\mathrm{red}}$, the reduction of $A$, is strongly noetherian. 
\end{theo}

\begin{proof}
    It suffices to show the `if' part. That is, if $A_{\mathrm{red}}$ is strong noetherian, then $A$ is strongly noetherian.

    Assume that $A_{\mathrm{red}}$ is strongly noetherian. We claim that the reduction of $A\langle T\rangle$ is equal to $A_{\mathrm{red}}\langle T\rangle$ and that $A\langle T\rangle$ is noetherian.
    If this holds, the claim follows by induction on the number of variables.

    Let $I$ be the radical of $A$. The ideal
    \[I\langle T\rangle:=\{f=\sum_{j=0}^\infty a_jT^j\in A\langle T\rangle: a_j\in I,\ \forall j\},\]
    is nilpotent. In fact, under the inclusion $A\langle T\rangle\to A[[T]]$, $I\langle T\rangle$ is included in $I[[T]]$, which is nilpotent. A similar argument shows that $A_{\mathrm{red}}\langle T\rangle$ is reduced. Thus, the radical of $A\langle T\rangle$ is equal to $A_{\mathrm{red}}\langle T\rangle$.

    Next we show that $A\langle  T\rangle$ is noetherian. By \cite[Section 3.7.2, Proposition 2]{Bosch_1984}, any ideal of $A$ is closed. In particular, for each integer $N\ge 1$, $A/I^N$ is a noetherian Tate ring. Since $I$ is nilpotent, it suffices to show that each $A/I^N\langle T\rangle$ is noetherian. We prove this by induction on $N$. For $N=1$, this holds by the assumption that $A/I=A_{\mathrm{red}}$ is strongly noetherian. Assume that we have already proved the case for $N=n$. Note that $I^n/I^{n+1}$ is an $A/I^n$-module and we choose generators $x_1,x_2,\dots,x_r$. We claim that $I^n/I^{n+1}\langle T\rangle$ is generated by $x_1,x_2,\dots,x_r$ as an $A/I^n\langle T\rangle$-module. In fact, by the open mapping theorem (cf. Theorem \ref{theo: open mapping theorem}), for any bounded open subring $A_0\subseteq A$, the $A_0$-module generated by $x_1,x_2,\dots,x_r$ is open in $I^n/I^{n+1}$. Hence, for any sequence $\{y_j : j \geq 1\}$ in $M$ that tends to $0$, there exist elements $a_{ij} \in A$ for $1 \leq i \leq r$ and $j \geq 1$ such that 
    \[\lim_{j\to \infty}a_{ij}=0\]
    for all $i$ and
    \[y_j=\sum_{i=1}^ra_{ij}x_i\]
    for all $j$. This proves that $I^n/I^{n+1}\langle T\rangle$ is generated by $x_1,x_2,\dots,x_r$. By the induction hypothesis, any submodule of $I^n/I^{n+1}\langle T\rangle$ is finitely generated. Hence, for any ideal $J\subseteq A/I^{n+1}\langle T\rangle$, the sequence
    \[0\to J\cap I^{n}/I^{n+1}\to J\to \overline J\to 0,\]
    where $\overline{J}$ is the projection of $J$ in $A/I^n\langle T\rangle$, which implies that $J$ is finitely generated.
\end{proof}

As an important corollary, any artinian Tate ring is strongly noetherian. In particular, the ring $\Bpp{\alpha}(K)$, where $K$ is a perfectoid field over $\Q_p$ (See Subsection \ref{subsec: derham spec} for explicit definitions), is strongly noetherian.

\subsubsection*{The smoothoid case --} The definition of smoothoid spaces is introduced in \cite{heuer2022moduli}. In the paper, Heuer defined the module sheaf of differentials for a smoothoid space (over $\Q_p$) by using Hodge--Tate comparison. We will show that our definition coincides with his.

\begin{defi}[{\cite[Definition 2.2]{heuer2022moduli}}]
    An adic space $X$ is called \emph{smoothoid} if it admits an affinoid covering such that each of them admits a standard smooth map to a perfectoid affinoid space.
\end{defi}

Since locally on $X$, $X$ is isomorphic to the spectrum of a sous-perfectoid algebra, $X$ is \'etale sheafy. Hence, \'etale vector bundles are equivalent to analytic vector bundles.

Let $K$ be a perfectoid field over $\Q_p$. In \cite{heuer2022moduli}, Heuer defined the module of differentials for any smoothoid space over $K$ as follows.

\begin{defi}[{\cite[Definition 2.10]{heuer2022moduli}}]
    Let $X$ be a smoothoid space over $K$. Define the sheaf
    \[\widetilde{\Omega}_X:=R^1\nu_*\calO_{X_v},\]
    where $\calO_{X_v}$ is the structure sheaf of the $v$-topology of $X$ and $\nu$ is the push-forward from the $v$-topology to the \'etale topology.
\end{defi}

By \cite[Proposition 2.9]{heuer2022moduli}, $\widetilde{\Omega}_{X}$ is always a vector bundle.

\begin{theo}
    Let $X$ be a smoothoid space over $K$. Then the sheaf of differentials $ {\Omega}_{X/\Q_p}$ is isomorphic to the vector bundle $\widetilde{\Omega}_{X}\{-1\}$, where $\{-1\}$ is the Breuil--Kisin twist. 
\end{theo}

\begin{proof}
    Denote by $\bB^+_{\dR}$ the de Rham period sheaf on $X_v$. It admits a Fontaine's $\theta$-map
    \[\theta:\bB_\dR^+\to \calO_{X_v}.\]
    Consider the Bockstein differential $d_{B}:\calO_X\to \widetilde{\Omega}_X\{1\}$ associated to the sequence
    \[0\to \calO_{X_v}\{1\}\to \bB^+_{\dR}/\ker\theta^2\to \calO_{X_v}\to0.\]

    We claim that there exists a unique $\calO_X$-linear isomorphism
    \[I_X: \Omega_{X/\Q_p}\to \widetilde{\Omega}_{X}\{1\}\]
    that fits into the following commutative diagram:
    \[\begin{tikzcd}
        \calO_X\ar{rr}{d_B}\ar{dr}{d}&&\widetilde{\Omega}_X\{1\}\\
        & {\Omega}_{X/\Q_p}\ar{ur}{I_X}
    \end{tikzcd}\]
    where $d$ is the universal differential operator.

    We first check the uniqueness. It suffices to prove the statement analytically locally on $X$. Hence, we can assume that there exists a perfectoid Tate algebra $A$ over $K$ and $X=\Spa(R,R^+)$, where $R$ is standard smooth over $A$. We claim that for any complete adic module $M$ over $R$ and any continuous derivation $d: R \to M$, the derivation $d$ is $A$-linear, which is equivalent to $d(A) = 0$. Choose a bounded open submodule $M_0$ of $M$ such that $d(A^\circ)\subset M_0$. For any $a\in A^\circ$, there exists $a_1\in A^\circ$ such that $a_1^p-a\in pA^\circ$ and thus,
    \[d(a)\in pa_1^{p-1}d(a_1)+p\cdot d(A^\circ)\subseteq pM_0.\]
    By induction, we find that $d(A_0)\subseteq \bigcap_{n\geq 1} p^nM_0=0$. In particular, $ {\Omega}_{X/\Q_p}$ is generated by the image of $d$ as a sheaf of $\calO_X$-modules, and thus, the uniqueness follows.

    It remains to prove the existence of $I$. By the uniqueness of $I$, we only need to prove that such an $I$ exists analytically locally on $X$. Hence, we can assume that $X=\Spa(R,R^+)$ and that there exists a toric chart
    \[A\langle T_1^{\pm1},T_2^{\pm1},\dots,T_d^{\pm1}\rangle\to R\]
    in the sense of \cite[Definition 2.4]{heuer2022moduli}. Then, the claim is proved in the proof of \cite[Proposition 2.5]{Heuer_2022}.
\end{proof}

\section{Foundations of smooth adic spaces over $\Bp{\alpha}$}\label{Section: Foundation}

\subsection{The de Rham spectra of perfectoid spaces}\label{subsec: derham spec}

Let $(K,K^+)$ be a perfectoid Tate--Huber pair over $\Q_p$ and $\alpha\in\mathbb{Z}_{>0}$. Keep these notations until the end of this section.

\begin{defi}
    For any perfectoid Tate--Huber pair $(R,R^+)$ over $\Q_p$,
    define $\Bpp{\alpha}(R)$ as $\Bdrp{R}/\ker\theta^\alpha$. We emphasize that the topology on this ring is defined by the equality
    \[\Bpp{\alpha}(R)=\Ainf(R^+)/(\ker\theta)^{\alpha}[\frac{1}{p}],\]
    see the following remark for a detailed explanation.
    Moreover, we define $\Bpp{\alpha}^+(R^+)$ as the preimage of $R^+$ in $\Bpp{\alpha}(R)$ under the Fontaine's map $\theta$.

    Let $X$ be an \'etale sheafy adic space over $\Q_p$. We define $\Bpp{\alpha,X}$ as the $v$-sheaf
    \[\bB_{\dR,X}^+/\ker(\theta)^\alpha\]
    where $\bB_{\dR,X}^+$ is the de Rham period sheaf in the sense of \cite[Definition 6.1]{scholze2012padic}.
\end{defi}

\begin{rmk}
    We introduce the topology on $\Bpp{\alpha}(R)$. By definition, we have
    $$\Bpp{\alpha}(R)=\Ainf(R^+)/(\ker\theta)^\alpha\left[\tfrac{1}{p}\right].$$
    A topology is defined on $\Bpp{\alpha}(R)$ such that $\Ainf(R^+)/(\ker\theta)^\alpha$ is an open subring endowed with the $p$-adic topology.

    Furthermore, this topology is independent of the choice of $R^+$ (Corollary \ref{cor: good form of B alpha}).
\end{rmk}

\noindent\textbf{Notation: }{In the following of this section, we will use $\Bp{\alpha}$ (resp. $\Bp{\alpha}^+$, $\bdrp$, $\ainf$) to denote the ring $\Bpp{\alpha}(K)$ (resp. $\Bpp{\alpha}^+(K^+)$, $\Bdrp{K}$, $\Ainf(K^+)$).}

\begin{lemma}\label{lemma: topology on B alpha}
    Let $R$ be a perfectoid Tate ring over $\Q_p$ and $\vpi\in R^\flat$ be a pseudo-uniformizer. Let $R^+$ be an open bounded subring in $R$ which is integral perfectoid. Then, $\varpi$ lies in $R^{+,\flat}$ and the $p$-adic topology on $\Ainf(R^{+})/(\ker\theta)^\alpha$ is equal to the $[\vpi]$-adic topology on $\Ainf(R^{+})/(\ker\theta)^\alpha$ for any positive integer $\alpha$.
\end{lemma}

\begin{proof}
    Let $R^\circ$ be the subring of power bounded elements in $R$. Then the inclusion $R^{+}\to R^{\circ}$ is an almost isomorphism by \cite[Theorem 5.2]{scholze2012perfectoid}. In particular, $R^+$ contains all topologically nilpotent elements. Hence, $\varpi^\sharp\in R^+$.

    Note that the $[\vpi]$-adic topology does not depend on the choice of the pseudo-uniformizer $\vpi$. Hence, we can choose $\vpi$ so that $\xi=p+[\vpi]y$ is a generator of $\ker\theta$ for some $y\in\Ainf(R^+)$ (cf. \cite[lemma 6.2.8]{Scholze_2020}). Hence, in $\Ainf(R^+)/(\ker\theta)^\alpha$, we obtain the equation
    $$0 = \xi^\alpha = p^{\alpha} + [\vpi]y'$$
    for some $y'\in \Ainf(R^+)$,
    which implies that the $p$-adic topology is finer than the $[\vpi]$-adic topology.

    On the other hand, since $p$ itself is a pseudo-uniformizer of $R$, there exists some $z\in R^+$ and an integer $N>0$ such that
    $$\vpi^{\sharp N} = pz.$$
    Let $z_1$ be a preimage of $z$ in $\Ainf(R^+)$, then
    $$[\vpi]^N - pz_1 \in \xi\Ainf(R^+).$$
    Thus, $([\vpi]^N - pz_1)^\alpha = 0$ in $\Ainf(R^\circ)/(\ker\theta)^\alpha$, which is equivalent to saying that $[\vpi]^{N\alpha} \in p(\Ainf(R^{+})/(\ker\theta)^{\alpha})$. This proves that the $[\vpi]$-adic topology is finer than the $p$-adic topology.
\end{proof}

\begin{cor}\label{cor: good form of B alpha}
    Keep the notations in Lemma \ref{lemma: topology on B alpha}, we have
    $$\Bpp{\alpha}(R)=\Ainf(R^+)/(\ker\theta)^\alpha\big[\frac{1}{[\vpi]}\big].$$
    Furthermore, the topology on $\Bpp{\alpha}(R)$ does not depend on the choice of $R^+$.
\end{cor}

\begin{proof}
    The first claim is obvious. We only need to prove the second claim.
    
    Let $R^\circ$ be the subring of power bounded elements in $R$. It suffices to check that the topology defined by $R^+$ is equal to the topology defined by $R^\circ$. This follows from
    \[[\vpi]\Ainf(R^\circ)\subseteq \Ainf(R^+)\subseteq\Ainf(R^\circ)\]
    for any pseudo-uniformizer $\vpi\in R^\flat$.
\end{proof}

\begin{prop}\label{prop: etalebdrp}
    For $\alpha<\infty$, and any standard \'etale map $f:\Bp{\alpha}\to R$, the ring $\overline{R}=K\otimes_{B_\alpha}R$ is a perfectoid Tate ring. Moreover,
    there is a unique isomorphism of topological rings:
    $$i_R:R\cong \Bpp{\alpha}(\overline{R}),$$ such that the diagram
    $$\xymatrix{\Bp{\alpha}\ar[r]\ar[d]_-{f} &\Bpp{\alpha}(\overline{R})\ar[d]\\ R\ar[r]\ar[ur]^-{i_R} &\overline{R}}$$ commutes, where the upper row is the homomorphism induced by the functorality of $\Bpp{\alpha}(-)$, the lower row is the natural projection and the right column is the Fontaine's map.
\end{prop}

\begin{proof}
    Since $R$ is standard \'etale over $B_\alpha$, so is the morphism $K\to \overline R$. Hence, $\overline R$ is a perfectoid Tate ring.

    By the definition of standard \'etaleness, it suffices to consider the case when $f$ is finite \'etale or a rational localization. The finite \'etale case is trivial, and thus we only need to prove the claim when $f$ is a rational localization, say $R=\Bp{\alpha}\langle\tfrac{f_1,f_2,\dots, f_r}{g}\rangle$.

    Fix a generator $\xi$ of $\ker\theta$. By definition, the underlying space of $$\mathrm{Spa}(\Bp{\alpha},\Bp{\alpha}^+)$$ 
    is homeomorphic to $\mathrm{Spa}(K,K^+)$, and by analytic tilting correspondence, we can assume $f_i=[ \rf_i]$ and $g=[\rg]$ where $\rf_i$ and $\rg$ lies in $K^\flat$. Moreover, we may suppose $\rf_i$ and $\rg$ are in $K^{+,\flat}$ and $\rf_1=\vpi\in K^{+\flat}$ is a pseudo-uniformizer.

    Let $\xi\in W(K^{+\flat})$ be a generator of $\ker\theta$.

    By Corollary \ref{cor: good form of B alpha}, the rational localization is
    \begin{equation}\label{eq: section 2.3 derham localization commuate 1}
        W(K^{+\flat})/\xi^\alpha\left[\frac{[\rf_1],[\rf_2],\dots, [\rf_r]}{[\rg]}\right]^{\land}\left[\frac{1}{[\vpi]}\right]
    \end{equation} 
    where the completion is the $[\vpi]$-adic completion. Note that in the ring $W(K^{+\flat})[\frac{1}{[\rg]}]$, $$[\vpi]\frac{[\rf_i^{\frac{1}{p^l}}]}{[\rg^{\frac{1}{p^l}}]}=\frac{[\rf_1]}{[\rg]}[\rg^{\frac{p^l-1}{p^l}}][\rf_i^{\frac{1}{p^l}}].$$ 
    Hence, as subrings of $W(K^{+\flat})[\frac{1}{[\rg]}]$, we have 
    \[[\vpi]\cdot\left(\bigcup_{l\geq1}W(K^{+\flat})/\xi^\alpha\big{[}\frac{[\large\rf_1^{\frac{1}{p^l}}],[\rf_2^{\frac{1}{p^l}}],\dots, [\rf_r^{\frac{1}{p^l}}]}{[\rg^{\frac{1}{p^l}}]}\big]\right)\subseteq W(K^{+\flat})/\xi^\alpha\left[\frac{[\rf_1],[\rf_2],\dots, [\rf_r]}{[\rg]}\right].\]
    Use Corollary \ref{cor: good form of B alpha} again, the ring (\ref{eq: section 2.3 derham localization commuate 1})
    is equal to $$\left(\varprojlim_{m}\ \varinjlim_{l}W(K^{+\flat})/(p^m,\xi^\alpha)\left[\frac{[\large\rf_1^{\normalsize\frac{1}{p^l}}],[\large{\rf}_2^{\normalsize\frac{1}{p^l}}],\dots, [\large\rf_r^{\normalsize\frac{1}{p^l}}]}{[\large\rg^{\normalsize\frac{1}{p^l}}]}\right]\right)\left[\frac{1}{p}\right].$$

    We claim that
    $$\varinjlim_{l}\ W(K^{+\flat})/p^m\left[\frac{[\rf_1^{\frac{1}{p^l}}],[f_2'^{\frac{1}{p^l}}],\dots, [\rf_r^{\frac{1}{p^l}}]}{[\rg^{\frac{1}{p^l}}]}\right]\cong W\left(\varinjlim_{l}\ K^{+\flat}\left[\frac{\rf_1^{\frac{1}{p^l}},\rf_2^{\frac{1}{p^l}},\dots, \rf_r^{\frac{1}{p^l}}}{\rg^{\frac{1}{p^l}}}\right]\right)/p^m.$$ 
    In fact, the left hand side is a subring of 
    $$W(K^{+\flat})/p^m\left[\tfrac{1}{[g^{\flat}]}\right]$$ 
    and by the universal property of Witt vectors of a perfect ring, $$W(K^{+\flat})/p^m\left[\tfrac{1}{[g^{\flat}]}\right]\cong W\left(K^{+\flat}\left[\tfrac{1}{g^{\flat}}\right]\right)/p^m.$$ 
    By definition, 
    $$\varinjlim_{l}\ K^{+\flat}\left[\frac{\rf_1^{\frac{1}{p^l}},\rf_2^{\frac{1}{p^l}},\dots, \rf_r^{\frac{1}{p^l}}}{\rg^{\frac{1}{p^l}}}\right]$$ 
    is a perfect subring of $K^{+\flat}\left[\tfrac{1}{g^{\flat}}\right]$. Thus $$W\left(\varinjlim_{l} K^{+\flat}\left[\frac{\rf_1^{\frac{1}{p^l}},\rf_2^{\frac{1}{p^l}},\dots, \rf_r^{\frac{1}{p^l}}}{\rg^{\frac{1}{p^l}}}\right]\right)/p^m$$ 
    is a subring of $W\left(K^{+\flat}\left[\tfrac{1}{g^{\flat}}\right]\right)/p^m$. This proves that the natural homomorphism $$\varinjlim_{l}\ W(K^{+\flat})/p^m\left[\frac{[\rf_1^{\frac{1}{p^l}}],[\rf_2^{\frac{1}{p^l}}],\dots, [\rf_r^{\frac{1}{p^l}}]}{[\rg^{\frac{1}{p^l}}]}\right]\to W\left(\varinjlim_{l}\ K^{+\flat}\left[\frac{\rf_1^{\frac{1}{p^l}},\rf_2^{\frac{1}{p^l}},\dots, \rf_r^{\frac{1}{p^l}}}{\rg^{\frac{1}{p^l}}}\right]\right)/p^m$$ 
    is injective. Every element of $W\left(\varinjlim_{l} K^{+\flat}\left[\frac{\rf_1^{\frac{1}{p^l}},\rf_2^{\frac{1}{p^l}},\dots, \rf_r^{\frac{1}{p^l}}}{\rg^{\frac{1}{p^l}}}\right]\right)/p^m$ has a Teichmuller representation
    $$\sum_{i=0}^{m-1}[\mu_i]p^i,$$ so there exists an $l$ such that $$\mu_i\in K^{+\flat}\left[\tfrac{\rf_1^{\tfrac{1}{p^l}},\rf_2^{\tfrac{1}{p^l}},\dots, \rf_r^{\tfrac{1}{p^l}}}{\rg^{\tfrac{1}{p^l}}}\right]$$ for all $\mu_i$. This proves the natural homomorphism is surjective.

    Thus,
    \begin{align*}
        &\left(\varprojlim_{m}\varinjlim_{l}W(K^{+\flat})/(p^m,\xi^\alpha)\left[\frac{[\rf_1^{\frac{1}{p^l}}],[\rf_2^{\frac{1}{p^l}}],\dots, [\rf_r^{\frac{1}{p^l}}]}{[\rg^{\frac{1}{p^l}}]}\right]\right)\left[\frac{1}{p}\right]\\&\cong \left(\varprojlim_{m}W\left(\varinjlim_{l} K^{+\flat}\left[\frac{\rf_1^{\frac{1}{p^l}},\rf_2^{\frac{1}{p^l}},\dots, \rf_r^{\frac{1}{p^l}}}{\rg^{\frac{1}{p^l}}}\right]\right)/(p^m,\xi^\alpha)\right)\left[\frac{1}{p}\right]
    \end{align*}
    The right-hand-side is just the $\Bpp{\alpha}$ of the rational localisation.
\end{proof}

\begin{cor}\label{etshperfd}
    The Tate--Huber pair $\left(\Bp{\alpha},\Bp{\alpha}^+\right)$ is \'etale sheafy.
\end{cor}

\begin{proof}
    Since $\Bp{\alpha}\to K$ is a nilpotent thinkening, the standard \'etale topology of $\left(\Bp{\alpha},\Bp{\alpha}^+\right)$ is equal to $(K,K^+)$. The previous theorem implies that the structure presheaf is the sheaf $\Bpp{\alpha}$. In particular, the structural presheaf a sheaf.
\end{proof}

\begin{prop}
    For any perfectoid space $X$ over $\Q_p$, $R^1\nu_*\mathrm{GL}_n(\Bpp{\alpha})=1$.
\end{prop}

\begin{proof}
    When $\alpha<\infty$, 
    we prove it by induction on $\alpha$. The case $\alpha=1$ is well known as so called `almost \'etale descent', see \cite{scholze2012perfectoid} for the almost \'etale morphism between perfectoid spaces and \cite{berger2008familles} Chapter $3$ for the proof to almost \'etale descent.

    Assume the statement is already proved for $\alpha-1$. For $\alpha$, choose a generator $\xi$ of $\ker\theta$ and consider the following exact sequence
    $$0\to \widehat\calO_{X_\pet}^{n^2}=\mathrm{M}_n(\widehat\calO_{X_\pet})\xrightarrow{1+\xi^{\alpha-1}\cdot \mathrm{id}} \mathrm{GL}_{n}(\Bpp{\alpha})\to \mathrm{GL}_n(\Bpp{\alpha-1})\to 1.$$
    The proposition follows by taking the long exact sequence.
\end{proof}

\subsection{Properties of algebras with a toric chart}\label{toricalg}

In this subsection, we fix a generator $\xi$ of $\ker\big(\theta:\bdrp\to K\big)$. We emphasize that $\xi$ is not necessarily an element of $\ainf$.

\begin{defi}
    The \emph{toric algebra} of dimension $d$ over $\Bp{\alpha}$ is the topological algebra \[\Bp{\alpha}\langle T_1^{\pm1},T_2^{\pm1},\dots,T^{\pm1}_d\rangle,\] which we will denote by $\T$ in the sequel. We also define
    $$\T^+:=\Bp{\alpha}^+\langle T_1^{\pm1},T_2^{\pm1},\dots,T^{\pm1}_d\rangle.$$
\end{defi}

Moreover, we set $\overline{\T}=K\otimes_{\Bp{\alpha}} \T=K\langle T_1^{\pm1},T_2^{\pm1},\dots,T^{\pm1}_d\rangle$. For each positive integer $n$, define
$$\T_n=\Bp{\alpha}\langle T_1^{\pm\tfrac{1}{p^n}},T_2^{\pm\tfrac{1}{p^n}},\dots,T^{\pm\tfrac{1}{p^n}}_d\rangle;$$
$$\overline{\T}_n=\Bp{\alpha}\langle T_1^{\pm\tfrac{1}{p^n}},T_2^{\pm\tfrac{1}{p^n}},\dots,T^{\pm\tfrac{1}{p^n}}_d\rangle\otimes_{\Bp{\alpha}}K=K\langle T_1^{\pm\tfrac{1}{p^n}},T_2^{\pm\tfrac{1}{p^n}},\dots,T^{\pm\tfrac{1}{p^n}}_d\rangle.$$
Define $\T^+_n$ as
\[B_\alpha^+\langle T_1^{\pm\tfrac{1}{p^n}},T_2^{\pm\tfrac{1}{p^n}},\dots,T^{\pm\tfrac{1}{p^n}}_d\rangle\]
and $\overline{\T}^+_n$ as the image of $\T_n^+$ in $\overline{\T}_n$.

Let $\overline{\T}_{\infty}=\varinjlim_n\overline{\T}_n$, $\T_{\infty}=\varinjlim_n \T_n$, and let
$$\widehat{\overline{\T}}_{\infty}=K\langle T_1^{\pm\tfrac{1}{p^\infty}},T_2^{\pm\tfrac{1}{p^\infty}},\dots,T^{\pm\tfrac{1}{p^\infty}}_d\rangle\sim\varinjlim_n\overline{\T}_n$$
be the unique perfectoid tilde limit. Finally, define
$$\widehat{\T}_\infty=\Bpp{\alpha}(\widehat{\overline{\T}}_{\infty}).$$
We emphasize that $\widehat{\T}_\infty$ is not a completion of $\T_\infty$. However, we will prove that $\T_{\infty}$ can be embedded into $\widehat{\T}_\infty$ as a dense subalgebra.

In what follows, we define $\overline{T}_i^{\flat}$ to be the tilting element $(T_i, T_i^{\tfrac{1}{p}}, T_i^{\tfrac{1}{p^2}}, \dots) \in \widehat{\overline{\T}}_\infty^\flat$.

\begin{prop}\label{prop: Ainf torus}
    We have
    $$\Ainf(\widehat{\overline{\T}}_\infty^+)=\left(\bigcup_{l\geq0}\ainf\big[[(\overline{T}_i^{\flat})^{\frac{\pm1}{p^l}}]:i=1,2,\dots, d\big]\right)^\land,$$
    where the right-hand side denotes the $(p,\ker\theta)$-adic completion of the union of all sub-$\Ainf$-algebras generated by
    $$\{[(\overline{T}_i^{\flat})^{\frac{1}{p^l}}]:i=1,2,\dots, d\},$$
    which is a polynomial algebra.
\end{prop}

\begin{proof}
    Both rings are $p$-adically complete and $p$-torsion free. It suffices to check that $\Ainf(\widehat{\overline{\T}}_\infty^+)/p$ and $\left(\bigcup_{l\geq0}\Ainf(K^+)\big[[(\overline{T}_i^{\flat})^{\frac{1}{p^l}}]:i=1,2,\dots, d\big]\right)^\land/p$ are isomorphic as perfect rings, which is clear from the construction.
\end{proof}

We now explain how to embed $\T_\infty$ into $\widehat{\T}_\infty$.

\begin{theo}\label{theorem: toric normal trace}
    (1) For each $n$, there exists a closed immersion of topological $\Bp{\alpha}$-algebras
    $$\T_n\hookrightarrow \widehat{\T}_\infty$$
    sending $T_{i}^{\frac{1}{p^n}}$ to $[(\overline{T}_i^{\flat})^{\frac{1}{p^n}}]$. This embedding makes $\T_\infty$ a dense subalgebra of $\widehat{\T}_\infty$.

    (2) The normalized trace map $R_n: \T_\infty \to \T_n$ (that is, the colimit of $\tfrac{1}{p^{(m-n)d}} \mathrm{Tr}_{\T_m/\T_n}$ on $\T_m$) can be continuously extended to $\widehat{\T}_\infty$, and we still denote this extension by $R_n$.
\end{theo}

\begin{proof}
    This follows directly from Proposition \ref{prop: Ainf torus}.
\end{proof}

Consequently, if we denote by $\X_n(\T)$ the kernel of $R_n$, then we have a direct sum decomposition of topological $\T_n$-modules:
    $$\widehat{\T}_\infty = \T_n \oplus \X_n(\T).$$

Let $(A, A^+)$ be a Tate--Huber pair over $\Bp{\alpha}$ and let $f: (\T, \T^+) \to (A, A^+)$ be a standard \'etale homomorphism. The above constructions for toric algebras can be generalized to $A$.

\begin{defi}
    We maintain the notations above and define the following:
\begin{enumerate}[(1)]
    \item $A_n = \T_n \otimes_{\T} A$; this coincides with the completed tensor product since $\T_n$ is finite \'etale over $\T$;
    \item $\overline{A}_n = A_n \otimes_{\Bp{\alpha}} K$;
    \item $A_\infty = \varinjlim_{n \geq 1} A_n$;
    \item $\overline{A}_\infty = \varinjlim_{n \geq 1} \overline{A}_n = A_\infty \otimes_{\Bp{\alpha}} K$;
    \item $\widehat{\overline{A}}_\infty$ is the perfectoid tilde limit of $\overline{A}_n$, i.e., $\widehat{\overline{A}}_\infty = \widehat{\overline{\T}}_\infty \widehat{\otimes}_{\T} A$;
    \item $\widehat{A}_\infty = \Bpp{\alpha}(\widehat{\overline{A}}_\infty, \widehat{\overline{A}}_\infty^+)$.
\end{enumerate}
\end{defi}

\begin{lemma}\label{lemma: tensor formula of toric tower}
    Recall from Theorem \ref{theorem: toric normal trace} that $\T$ is a closed subalgebra of $\widehat{\T}_\infty$. Then there exists a unique isomorphism of topological algebras
    $$\widehat{\T}_\infty \widehat{\otimes}_\T A \cong \widehat{A}_\infty$$
    satisfying the following conditions:
    \begin{enumerate}[(1)]
        \item The restriction to $\widehat{\T}_\infty$ is induced by the functoriality of $\Bpp{\alpha}(-)$;
        \item After reduction modulo $\xi$, it recovers the natural isomorphism $\widehat{\overline{\T}}_\infty \widehat{\otimes}_{\overline{\T}} A \cong \widehat{\overline{A}}_\infty$.
    \end{enumerate}
\end{lemma}

\begin{proof}
    By Proposition \ref{prop: etalebdrp}, the homomorphism $\widehat{\T}_\infty \to \widehat{A}_\infty$ is standard \'etale. Hence, the lemma is a direct consequence of Theorem \ref{theo: formally etale of standard etale}.
\end{proof}

\begin{cor}\label{cor: direct sum toric algebras}
    For each $n > 0$, there is a natural isomorphism of topological $A_n$-modules
    $$\widehat{A}_\infty \cong A_n \oplus \X_n(A),$$
    where $\X_n(A) = A \widehat{\otimes}_\T \X_n(\T)$ (cf. Theorem \ref{theorem: toric normal trace}).

    The inclusion $A_n \to A_n \oplus \X_n(A) \cong \widehat{A}_\infty$ is given by
    $$A_n \to A_n \widehat{\otimes}_{\T_n} \widehat{\T}_{\infty} \cong \widehat{A}_\infty,$$
    and the projection $\widehat{A}_\infty \cong A_n \oplus \X_n(A) \to A_n$ is given by
    $$\widehat{A}_\infty \cong A_n \widehat{\otimes}_{\T_n} \widehat{\T}_\infty \xrightarrow{\mathrm{id} \widehat{\otimes} R_n} A_n.$$
\end{cor}

\begin{proof}
    This follows directly from Lemma \ref{lemma: tensor formula of toric tower} and Theorem \ref{theorem: toric normal trace}.
\end{proof}

We note that this corollary implies that the normalized trace $R_n: A_\infty \to A_n$ can be continuously extended to $\widehat{A}_\infty$. We will further show that these maps $R_n$ are uniformly continuous.

\begin{cor}
    For each $0\leq a<\alpha$, $t^a \widehat{A}_\infty\cap A=t^aA$ and the multiplication by $t$ map $t^aA/t^{a+1}A\to t^{a+1}A/t^{a+2}A$ is an isomorphism when $a<\alpha-1$.
\end{cor}

\begin{proof}
    Applying Corollary \ref{cor: direct sum toric algebras} to the case $n=1$ and $a<\alpha$, we see that $A/\xi^a$ is a direct summand of $\Bpp{a}(\widehat{\overline{A}}_\infty)$ (as $A$-modules). In particular, $t^a\widehat{A}_\infty\cap A=t^aA$.

    For the second claim, it suffices to show that the multiplication-by-$t$ map
    $$t^a\widehat{A}_\infty/t^{a+1}\widehat{A}_\infty\to t^{a+1}\widehat{A}_\infty/t^{a+2}\widehat{A}_\infty$$
    is an isomorphism. This follows from Proposition \ref{prop: etalebdrp}.
\end{proof}

\begin{theo}
    The Tate--Huber pair $(A,A^+)$ is strongly sous-\'etale sheafy (cf. Definition \ref{definition: sous-etale sheafy}). In particular, $(A,A^+)$ is \'etale sheafy.
\end{theo}

\begin{proof}
    Let $U_1,U_2,\dots,U_m$ be $m$ free variables. We need to prove that 
    $$\widetilde{A}:=A\langle U_1,U_2,\dots,U_m\rangle$$
    is sous-\'etale sheafy.

    Set $\overline{\widetilde{A}}=\widetilde{A}/\xi=\overline{A}\langle U_1,U_2,\dots,U_m\rangle$ and $\overline{\widetilde{A}}_\infty=\overline{A}_\infty\langle U_1^{\tfrac{1}{p^\infty}},U_2^{\tfrac{1}{p^\infty}},\dots,U_m^{\tfrac{1}{p^\infty}}\rangle$ (a perfectoid algebra).

    By the proof of Proposition \ref{prop: Ainf torus},
    $$\Ainf\left(\overline{A}^+_\infty\langle U_1^{\tfrac{1}{p^\infty}},U_2^{\tfrac{1}{p^\infty}},\dots,U_m^{\tfrac{1}{p^\infty}}\rangle\right)=\left(\bigcup_{l\geq0}\Ainf(\overline{A}^+_\infty)\left[[(U_i^{\flat})^{\tfrac{1}{p^l}}]:i=1,2,\dots, m\right]\right)^\land$$
    where $U_i^{\flat}=(U_i,U_i^{\tfrac{1}{p}},U_{i}^{\tfrac{1}{p^2}},\dots)$ and the completion is the $p$-adic completion. This shows that $\widehat{A}_\infty\langle U_1,U_2,\dots,U_m\rangle$ is a direct summand of $\Bpp{\alpha}(\overline{\widetilde{A}}_\infty)$ (see the proof of Corollary \ref{cor: direct sum toric algebras}). By Corollary \ref{cor: direct sum toric algebras}, $\widetilde{A}$ is a direct summand of $\widehat{A}_\infty\langle U_1,U_2,\dots,U_m\rangle$. Thus, $\widetilde{A}$ is a direct summand of $\Bpp{\alpha}(\overline{\widetilde{A}}_\infty)$, and the theorem follows from Corollary \ref{etshperfd}.
\end{proof}

We summarize the above consequences in the following diagram.

$$\xymatrix{ &\widehat{A}_\infty\ar@(r,u)[ddrr]^-{R_n} & & & \\ \widehat{\T}_{\infty}\ar@(d,l)[ddrr]_-{R_n}\ar[ur] & &A_\infty\ar[ul] & & \\ &\T_{\infty}\ar[ul]\ar[ur] & &A_n\ar[ul] & \\ & &\T_n\ar[ur]\ar[ul] & &A\ar[ul] \\ & & &\T\ar[ul]\ar[ur] & .}$$

\subsection{Canonical actions, locally analytic vectors and Sen theory of algebras with a toric chart}\label{ToricSentheory}

We discuss the decompletion theory of algebras over $B_\alpha$ with a toric chart in this subsection. Recall that we use $\Q_p^\cyc$ to denote the completion of the extension of $\Q_p$ by adding all $p$-power roots of unity. In this subsection, we assume that $K$ is defined over $\Q_p^\cyc$.

For each $n$, fix a primitive $p^n$-th root of unity $\zeta_{p^n}$ such that $\zeta_{p^{n+1}}^p=\zeta_{p^n}$. Let $\epsilon=(1,\zeta_p,\zeta_{p^2},\dots)\in K^\flat$ and $t=\log[\epsilon]\in\bdrp$.

\begin{rmk}
    We emphasize that each $\zeta_{p^n}$ in $K$ can be uniquely lifted to a primitive $p^n$-th root in $\bdrp$ since $(\bdrp,t)$ is a henselian pair. This lifting is still denoted by $\zeta_{p^n}$.
\end{rmk}

Let $\Gamma=\mathrm{Gal}(A_\infty/A)\cong \mathbb{Z}_p^d$, where the $i$-th basis element $\gamma_i$ acts on $A_\infty$ such that
$$\gamma_i(T_j^{\tfrac{1}{p^n}})=\begin{cases}
\zeta_{p^n}T_j^{\tfrac{1}{p^n}} & \text{if } i=j, \\
T_j^{\tfrac{1}{p^n}} & \text{if } i\neq j.
\end{cases}$$
This action naturally descends to $\overline{A}_\infty$ and extends continuously to $\widehat{\overline{A}}_\infty$. By the functoriality of $\Bpp{\alpha}(-)$, $\Gamma$ acts on $\widehat{A}_\infty$.

\begin{theo}
    The inclusion $A_n\subseteq \widehat{A}_\infty$ makes $A_n$ a $\Gamma$-invariant subalgebra.
\end{theo}

\begin{proof}
    When $f$ is an isomorphism, the claim follows from a direct calculation. In the general case, we use a formal \'etaleness argument. Specifically, for any $\gamma\in\Gamma$, there exists a unique isomorphism of $\Bp{\alpha}$-algebras $g_{\gamma}:A_n\to A_n$ lifting the Galois action of $\gamma$ on $\overline{A}_n$ that makes the following diagram commute:
    $$\xymatrix{ \T_n\ar[d]_{\gamma}\ar[r] & A_n\ar[d]^{g_\gamma} \\ \T_n\ar[r] & A_n.}$$
    Let $i_{\gamma}:A_n\to\widehat{A}_\infty$ be the composition $i_n\circ g_\gamma$. Since $f:\T_n\to A_n$ is \'etale, the action of $\gamma$ restricted to $A_n$ is the unique homomorphism lifting the Galois action of $\gamma$ on $\overline{A}_n$ and making the diagram:
    $$\xymatrix{ \T_n\ar[d]_{\gamma}\ar[r] & A_n\ar[d]^{g_\gamma} \\ \T_n\ar[r] & \widehat{A}_\infty}$$
    commute. Hence, $i_{\gamma}=g_{\gamma}$.
\end{proof}

Recall that there are two actions of $\Gamma$ on $A_\infty$: one is the restricted action from $\widehat{A}_\infty$ (denoted by $\gamma$ and called the \textbf{natural action}), and the other is the \textbf{Galois action} (denoted by $\gamma_*$). These two actions differ in general unless $\alpha=1$. The main obstruction to constructing a sheafified Riemann-Hilbert correspondence is this difference. We will compare these two actions explicitly in this subsection.

Let $\Omega_{A_n/\Bp{\alpha}}$ be the module of differentials. Then $\{T_i^{-1}dT_i:i=1,2,\dots,d\}$ is a basis, and let $T_i\partial_{i}$ be its dual.

We first consider the case of a torus.

\begin{lemma}
    The natural action of $\Gamma$ on $\T_n$ is $p^n\Gamma$-analytic.
\end{lemma}

\begin{proof}
    It suffices to prove that the action of $p^n\Gamma$ on $[(\overline{T}_i^{\flat})^{\tfrac{1}{p^n}}]$ is analytic. This follows from the fact that the map
    $$(a\mapsto [\epsilon]^a):\mathbb{Z}_p\to B_\alpha$$
    is analytic.
\end{proof}

\begin{defi}
    For any $\gamma\in \Gamma$ and $a\in \T_\infty$, define
    $$\partial_{\gamma}(a):=\lim_{u\in\mathbb{Z}_p\to 0}\frac{(u\cdot\gamma)(a)-a}{u},$$
    where the product is the scalar product. It is a derivation of $\T_\infty$ such that $\partial_{\gamma_i}=tT_i\partial_i$.
\end{defi}

Note that $\partial_\gamma$, regarded as a $\Q_p$-linear operator, is indeed nilpotent; thus we can consider $\exp (\partial_\gamma)$ (the exponential of an operator) for each $\gamma$. Moreover, by construction, the map
$$\gamma\mapsto \partial_\gamma$$
is a $\mathbb{Z}_p$-linear map from $\Gamma$ to $\End_{\Q_p}(\T_\infty)$. Thus, the map
$$(\gamma,a)\mapsto\exp(\partial_\gamma)(a)$$
for all $\gamma\in\Gamma$ and $a\in A_\infty$ defines a group action.

\begin{lemma}
    We have $\gamma(a)=\gamma_*\circ\exp(\partial_{\gamma})(a)=\exp(\partial_{\gamma})\circ\gamma_*(a)$ for all $\gamma\in\Gamma$ and $a\in \T_\infty$.
\end{lemma}

\begin{proof}
    It suffices to verify the identity for the generators $\gamma_i$. Thus, we need only check the compatibility of the action of $\gamma_i$ on $T_j$. If $i \neq j$, both actions fix $T_j$. If $i = j$, the identity follows from the well-known fact in $\bdrp$:
    $$[\ep^{\tfrac{1}{p^n}}]=\zeta_{p^n}\exp(\tfrac{t}{p^n}).$$
\end{proof}

Now we generalize these claims to $A$. Since $\T \to A$ is \'etale, each derivation $\partial_\gamma$ extends uniquely to $A$, still denoted by $\partial_\gamma$. Moreover, the uniqueness implies that
$$\partial_{\gamma_1+\gamma_2}=\partial_{\gamma_1}+\partial_{\gamma_2}$$
and
$$\partial_{u\gamma}=u\partial_\gamma.$$

\begin{lemma}
    For any $\gamma\in \Gamma$ and $a\in A_\infty$, we have $\gamma(a)=\gamma_*\circ\exp(\partial_{\gamma})(a)=\exp(\partial_{\gamma})\circ\gamma_*(a)$.
\end{lemma}

\begin{proof}
    The three actions coincide on $\T$ and also agree modulo $t$. For any $n$, since $f:\T\to A_n$ is standard \'etale, they are equal on $A_n$ by Theorem \ref{theo: formally etale of standard etale}.
\end{proof}

\begin{prop}\label{prop:locallyan}
    The inclusion $A_n\to \widehat{A}_\infty$ identifies $A_n$ with the $p^n\Gamma$-analytic vectors of $\widehat{A}_\infty$.
\end{prop}

\begin{proof}
    First, we prove that $A_n$ is $p^n\Gamma$-analytic. By definition of the Galois action, $\gamma_*(a)=a$ for any $\gamma\in p^n\Gamma$ and $a\in A_n$. Thus,
    $$\gamma(a)=\exp(\partial_\gamma)(a)$$
    for all $\gamma\in p^n\Gamma$ and $a\in A_n$, and this expression is clearly analytic.

    We now prove the claim by induction on $\alpha$. When $\alpha=1$, the result is well known. Suppose we have proved the case for $\alpha=a$; now assume $\alpha= a+1$. Consider the exact sequence
    $$0\to \widehat{A}_{\infty}/t^{a}\xrightarrow{t}\widehat{A}_\infty\to \widehat{A}_\infty/t\to 0.$$
    Taking $p^n\Gamma$-analytic vectors is a left exact functor, so we obtain the sequence
    $$0\to (\widehat{A}_{\infty}/t^{a})^{p^n\Gamma\mathrm{-an}}\xrightarrow{t}\widehat{A}_\infty^{p^n\Gamma\mathrm{-an}}\to (\widehat{A}_\infty/t)^{p^n\Gamma\mathrm{-an}}.$$
    By the induction hypothesis, we have the commutative diagram
    $$\xymatrix{0\ar[r] & A_n/t^a\ar[r]^-{t}\ar@{=}[d] &A_n\ar[r]\ar@{^(->}[d] &\overline{A}_n\ar[r]\ar@{=}[d] &0\\0\ar[r]&(\widehat{A}_{\infty}/t^{a})^{p^n\Gamma\mathrm{-an}}\ar[r]&\widehat{A}_\infty^{p^n\Gamma\mathrm{-an}}\ar[r]&(\widehat{A}_\infty/t)^{p^n\Gamma\mathrm{-an}}&}.$$
    Then the middle vertical arrow is an isomorphism by a straightforward diagram chase.
\end{proof}

Keeping the notations from the above subsection, we now prove the following theorem.

\begin{theo}\label{Sen toric}
    The algebra $\widehat{A}_\infty$ with its natural action by $\Gamma\cong \mathbb{Z}_p^d$, as well as the algebras
    $$\big(\widehat{A}_{\infty}\big)_n^i:=\Bp{\alpha}\langle T_1^{\pm \frac{1}{p^\infty}},\dots,T_{i-1}^{\pm\frac{1}{p^\infty}},T_{i}^{\pm\frac{1}{p^n}},T_{i+1}^{\pm\frac{1}{p^\infty}},\dots,T_d^{\pm\frac{1}{p^\infty}} \rangle\widehat{\otimes}_{\T}A,$$
    forms a $d$-dimensional Sen theory in the sense of \cite[Definition 2.2.5]{camargo2023geometric}. Moreover, this Sen theory is locally analytic.
\end{theo}

The proof is a direct calculation on a toric algebra and extends to the general case by using complete tensor formulas. Let us first recall the definitions.

\begin{defi}\label{sen}
    Let $\Gamma$ be a $p$-adic Lie group and assume that there is an isomorphism $\chi:\Gamma\cong \Z_p^d$. Suppose $\Lambda$ is a $\Q_p$-Tate algebra with a continuous action by $\Gamma$. Let $\Lambda_{*}$ be a definition subring.

    A $d$-dimensional Sen theory consists of the triple $(\Lambda,\Gamma,\chi)$ together with the following data:
    \begin{enumerate}[(1)]
        \item For each $n\in\mathbb{N}$ and $i\in\{1,2,\dots,d\}$, a closed subalgebra $\Lambda_n^i\subseteq \Lambda$.
        \item For each $n$ and $i$ as above, a $\Lambda_n^i$-linear continuous projection $R_n^i:\Lambda\to \Lambda_n^i$ and let
        $$X_n^i=\ker R_n^i.$$
    \end{enumerate}
    These data are required to satisfy the following conditions:
    \begin{enumerate}[(1)]
        \item $\Lambda_n^i\subseteq \Lambda_{n+1}^i$ and $R_{n}^i\circ R_{n+1}^i=R_n^i$ for all $n$ and $i$.
        \item $\Lambda_n^i$ is $\Gamma$-invariant and $R_n^i$ is $\Gamma$-equivariant for all $n$ and $i$.
        \item The projections $R_n^i$ commute pairwise.
        \item There exists a constant $c_2\in\Z$ such that $R_n^i(\Lambda^+)\subseteq p^{-c_2}\Lambda_{n *}^{i}$ for all $n$ and $i$, where $\Lambda_{n*}=\Lambda_*\cap\Lambda_n^i$. Moreover,
        $$x=\lim_{i\to\infty}R_n^i(x)$$
        for all $i$, a condition sometimes called the \textbf{CST2} condition.
        \item For any $i\leq d$, let $e_i$ be the $i$-th standard basis vector. For any $m\in\mathbb{N}$, let
        $$\gamma_i^{(m)}=\chi^{-1}(p^me_i).$$
        For any $n$, $i$ and $m\leq n$, the restriction
        $$(1-\gamma_i^{(m)})|_{X_n^i}$$
        is a bijection, and there exists a constant $c_3\in\Z$ such that
        $$(1-\gamma_i^{(m)})^{-1}(X_n^i\cap \Lambda_*)\subseteq p^{-c_3}(X_n^i\cap \Lambda_*)$$
        for all $n$ and $i$. This is sometimes called the \textbf{CST3} condition.
    \end{enumerate}
    Let $\mathcal{R}_n=\cap_{i=1}^d\Lambda_n^i$. This definition is independent of the choice of $\Lambda_*$.
\end{defi}

\begin{rmk}
    Our definition is slightly different from Camargo's original version. In fact, our definition is the special case when $H=1$ in Camargo's definition.
\end{rmk}

These data can be used to do decompletion.

\begin{theo}[\cite{berger2008familles},\cite{camargo2023geometric}]\label{decompletion}
    Keep the notations in Definition \ref{sen}.
    \begin{enumerate}[(1)]
        \item For any continuous cocycle $\rho:\Gamma\to GL_r(\Lambda)$, there exists an $N$ as well as a $V\in GL_r(\Lambda)$, such that the cocycle $$\gamma\mapsto V^{-1}\rho_\gamma\gamma(V)$$ takes values in $GL_r(\mathcal{R}_N)$.
        \item Let $N\geq 0\in \Z$ and $\rho_1,\rho_2:\Gamma\to GL_r(\mathcal{R}_N)$ be two cocycles such that they are equivalent as cocycles in $GL_r(\Lambda)$. There exists $N'>N$ such that $\rho_1$ and $\rho_2$ are isomorphic as cocycles in $GL_r(\mathcal{R}_{N'})$.
    \end{enumerate}
\end{theo}

In our case, some more assumptions hold and we can describe the decompleting process in a fancy way.

\begin{defi}
    Keep the notations in the Definition \ref{sen}. We say that the Sen theory is locally analytic if for any $n$, $\mathcal{R}_n$ is the set of $p^n\Gamma$-analytic vectors.
\end{defi}

Now we fix a continuous cocycle $\rho:\Gamma\to GL_r(\Lambda)$ and view it as a semi-linear action on $M=R^r$ via multiplication
$$\gamma v=\rho_\gamma\cdot \gamma(v)$$ where $\gamma(v)$ is just the $\gamma$ action of coefficients.

\begin{theo}\label{Sen}
    For sufficiently large $N$, the canonical homomorphism
    $$\Lambda\otimes_{\mathcal{R}_n}M^{p^n\Gamma-an}\to M$$ is isomorphic for all $n\geq N$.
\end{theo}

\begin{proof}
    By Theorem \ref{decompletion}, we may assume $N_0$ such that $\rho$ takes values in $GL_{r}(\mathcal{R}_{N_0})$. Choose $N>N_0$ such that $p^N\Gamma$ takes values in $1+p^2M_r(\mathcal{R}_{N_0*})$. We claim that this $N$ satisfies the assumption.

    Let $n\geq N$ and $$v=\begin{pmatrix}x_1\\x_2\\ \vdots \\ x_r\end{pmatrix}\in \Lambda^r.$$ We want to show that $v$ is $p^n\Gamma$-analytic is and only if $x_i\in \mathcal{R}_n$ for any $i$.
    By definition, $$\gamma\cdot v=\rho_\gamma\begin{pmatrix}\gamma(x_1)\\ \gamma(x_2)\\ \vdots \\ \gamma(x_r)\end{pmatrix}.$$ We first claim that $\gamma\mapsto \rho_\gamma$ is analytic on $p^N\Gamma$.
    We will consider $\mathcal{R}_{N_0}^r$ as a $\Q_p$-Banach space and $\mathcal{R}_{N_0*}^{r}$ a lattice. Let $L_i$ be the $\Q_p$-linear map $({\gamma_i}^{(N)}-1)$ hence they commute to each other. By assumption, $$L_i(\mathcal{R}_{N_0*}^{r})\subseteq p^2\mathcal{R}_{N_0*}^{r}.$$ Since $\rho$ is a homomorphism from $\Gamma$ to $GL_{\Q_p}(\mathcal{R}^r)$, $\rho\circ\chi^{-1}$ sends $(p^Nx_1,p^Nx_2,\dots,p^Nx_d)$ to $$(1+L_1)^{x_1}(1+L_2)^{x_2}\dots(1+L_d)^{x_d}=\sum_{r_1,r_2,\dots,r_d=0}^{\infty}\prod_{j=1}^d\binom{x_j}{r_j}L_j^{r_j}$$ for all $x_i\in\Z$. Note that $$v_p\left(\binom{n}{m}\right)\leq -v_p(m!)\leq -\tfrac{m}{p-1}.$$ Thus $$\prod_{j=1}^d\binom{x_j}{r_j}L_j^{r_j}$$ sends $\mathcal{R}_{N_0*}^{r}$ to $p^{r_1+r_2+\dots+r_d}\mathcal{R}_{N_0*}^{r}$. Thus by continuity of $\rho$, $$\rho\circ\chi^{-1}\left((p^Nx_1,p^Nx_2,\dots,p^Nx_d)\right)=\sum_{r_1,r_2,\dots,r_d=0}^{\infty}\prod_{j=1}^d\binom{x_j}{r_j}L_j^{r_j}$$ and is analytic.

    Returning to our theorem, since $$\gamma\cdot v=\rho_\gamma\begin{pmatrix}\gamma(x_1)\\ \gamma(x_2)\\ \vdots \\ \gamma(x_r)\end{pmatrix},$$ it is analytic on $p^n\Gamma$ if and only if $\gamma(x_i)$ is analytic on $p^n\Gamma$ for each $i$. Since our Sen theory is locally analytic, $x_i\in\mathcal{R}_n$ for each $i$.
\end{proof}

Now we prove the Theorem \ref{Sen toric}.

\begin{proof}[Proof to Theorem \ref{Sen toric}]
    It suffices to prove that the given data form a $d$-dimensional Sen theory.

    We first prove for $(\widehat{\T}_{\infty},\Gamma\cong\Z_p^d,\widehat{\T}_n^i,R_n^i)$ where $$\widehat{\T}_n^i=\Bp{\alpha}\langle T_1^{\pm \frac{1}{p^\infty}},\dots,T_{i-1}^{\pm\frac{1}{p^\infty}},T_{i}^{\pm\frac{1}{p^n}},T_{i+1}^{\pm\frac{1}{p^\infty}},\dots,T_d^{\pm\frac{1}{p^\infty}} \rangle$$ and $R_n^i$ is the normalised trace. Let
    $$\widehat{\T}_{\infty*}=\mathbb{A}_{inf}(K,K^+)/(\ker\theta)^\alpha\langle T_1^{\pm\frac{1}{p^\infty}},T_2^{\pm\frac{1}{p^\infty}},\dots,T_d^{\pm\frac{1}{p^\infty}}\rangle$$
    and thus
    $$\widehat{\T}_{n*}^i=\mathbb{A}_{inf}(K,K^+)/(\ker\theta)^\alpha\langle T_1^{\pm\frac{1}{p^n}},T_2^{\frac{1}{p^n}},\dots,T_d^{\pm\frac{1}{p^n}}\rangle.$$
    Note that $R_n^i$ sends $T_j^{\frac{1}{p^l}}$ to itself if $j\neq i$ or $l\leq n$ and $0$ otherwise. Thus, we can choose $c_2=0$. Moreover, by definition
    $$X_n^i=\widehat{\bigoplus}_{l>n}\T_n^iT^{\frac{1}{p^l}}_i$$ and thus $(1-\gamma_i^{(m)})$ sends $T_i^{\frac{1}{p^l}}$ to $(\zeta_{p^{l-m}}-1)T_i^{\frac{1}{p^l}}$.
    Note that
    $$v_p\left(\zeta_{p^{l-m}}-1\right)=\frac{1}{(p-1)p^{l-m-1}}.$$ Thus we can choose $c_3=1$.

    Now we go to the general case. Let $f:\T\to A$ be an algebra with a toric chart. Recall we have $$\widehat{A}_\infty=A\widehat{\otimes}_{\T}\widehat{\T}_\infty.$$ Fix a definition subring $A_*\subseteq A$ which contains the image of $\T_*$ defined in the above paragraph. Let $\widehat{A}_{\infty*}$ be the image of $A_*\widehat{\otimes}_{\T_*}\widehat{\T}_{\infty*}$ in $\widehat{A}_\infty$. By the consequences in the above paragraph,
    $$\widehat{\T}_{\infty}=\T_n^i\oplus X_n^i$$ and $$\widehat{\T}_{\infty*}=\T_{n*}^i\oplus \left(X_{n}^i\cap\widehat{\T}_{\infty*}\right).$$ Thus, $$\widehat{A}_{\infty*}=\widehat{A}_{n*}^i\oplus \left(A_*\widehat{\otimes}_{\T_*}\left(X_{n}^i\cap\widehat{\T}_{\infty*}\right)\right)$$ where the last tensor product is considered as its image in $\widehat{A}_\infty$. This calculation implies $A_{n*}^i=\widehat{A}_{\infty*}\cap A_n^i$ is equal to the image of $A_*\widehat{\otimes}_{\T_*}\T_{n*}^i$. Thus, we can still choose $c_2=0$ and $c_3=1$.
\end{proof}

\begin{rmk}
    The lattice $\widehat{A}_{\infty*}$ is in general different from $\mathbb{A}_{inf}(\widehat{\overline{A}}_\infty,\widehat{\overline{A}}_{\infty}^+)/(\ker\theta)^\alpha$. Hence, if we choose $\mathbb{A}_{inf}(\widehat{\overline{A}}_\infty,\widehat{\overline{A}}_{\infty}^+)/(\ker\theta)^\alpha$ as a new $\widehat{A}_{\infty*}$, $c_2$ and $c_3$ will change.

    We also want to mention that the property
    $$A_{n*}:=\widehat{A}_{\infty*}\cap A_n=im\big(A_*\widehat{\otimes}_{\T_*}{\widehat{\T}_{\infty*}}\to \widehat{A}_{\infty}\big)$$ highly depends on the existence of normalized traces. This does not holds for a general tower.
\end{rmk}

\subsection{The smooth adic spaces over $\Bp{\alpha}$}\label{subsec: smooth adic space over B_alpha}

The main object we will consider in the paper is smooth adic spaces over $\Bp{\alpha}$ for a perfectoid Tate--Huber pair $(K,K^+)$ over $\Q_p$.

The following proposition is the same as for rigid analytic spaces.

\begin{prop}
    Suppose $X$ is an adic space over $\mathrm{Spa}(\Bp{\alpha},{\Bp{\alpha}}^+)$. It is smooth over $\mathrm{Spa}(\Bp{\alpha},{\Bp{\alpha}}^+)$ if and only if there exists an open covering $$\{U_i\subseteq X:i\in I\}$$ such that:
    \begin{enumerate}[(1)]
        \item each $U_i$ is affinoid, say $U_i=\mathrm{Spa}(A_i,A_i^+)$;
        \item for each $i$, there exists a toric algebra of finite dimension $\T$ and \'etale morphism $\T\to A_i$.
    \end{enumerate}
\end{prop}

\begin{proof}
    Since the claim is local on $X$, we can assume that $X=\Spa(A,A^+)$ where $A$ is standard smooth over $B_\alpha$. Thus, there exists a standard \'etale morphism
    \[\T_d=B_\alpha\langle T_1,T_2,\dots,T_d \rangle\to A.\] Hence, it remains to prove the claim for $\T_d$.

    For any subset $S\subseteq \{1,2,\dots,d\}$, consider the rational localization
    \[\T_{d,S}:=\T_d\langle\tfrac{1}{T_i-\ep_i}:i=1,2,\dots,d\rangle\]
    where $\ep_i=0$ if $i\in S$ and $\ep_i$ otherwise.
    Then $\Spa(\T_d)=\bigcup_{S\subseteq \{1,2,\dots,d\}}\Spa(\T_{d,S})$ and the claim then follows.
\end{proof}

\section{Sheafified $p$-adic Riemann--Hilbert Correspondence}\label{Section:SheafRH}

Throughout this section, we keep the following assumptions.
\begin{itemize}
    \item Let $(K,K^+)$ be a perfectoid Tate--Huber pair over $\Q_p^{\cyc}$. Fix a positive integer $\alpha$ and denote by $\Bp{\alpha}$ the ring $\Bpp{\alpha}(K)$. In addition, denote by $B_{\alpha}^+$ the ring $\Bpp{\alpha}^+(K^+)$.
    \item Let $X$ be a smooth adic space over $\Bp{\alpha}$, and let $\overline{X}$ be the base change $X\times_{\Spa(B_\alpha)}\Spa(K)$. Put $\Omega_X=\Omega_{X/B_\alpha}$.
\end{itemize}

The goal of this section is to prove Theorem \ref{theorem: sheafRH}. Let us recall it:

\begin{theo}
    Let $r$ be a positive integer. There is a canonical isomorphism of sheaves
    $$\mathrm{RH}_{X,r}:R^1\nu_*\big(\mathrm{GL}_r({\mathbb{B}}_{\alpha,\overline{X}})\big)\cong \mathrm{MIC}_r(X)\{-1\},$$
    where
   $\mathrm{MIC}_r(X)\{-1\}$ is the sheaf on $X_{\et}$ associated to the presheaf which sends each $U\in X_{\et}$ to the isomorphic classes of pairs
        $$(M,\nabla:M\to M\otimes_{\inte{U}}\Omega_U\{-1\})$$
        where $M$ is a vector bundle of rank $r$ and $\nabla$ is an integrable connection with respect to the differential operator $\widetilde{d}$.
    Moreover, when $\alpha=1$, this isomorphism is the same as Heuer's sheafified $p$-adic Simpson correspondence.
\end{theo}

Recall the subscript $\alpha$ of the right hand side has already been contained in the assumption of $X$ and the left hand side depends only on the $\overline{X}$.

\subsection{The sheaf of integrable connections}

We will introduce $\mathrm{MIC}_r(X)\{-1\}$. 

\begin{defi}\label{defi: connections}
    Let $X/B_\alpha$ be a smooth adic space.
    
    An \emph{integrable connection} with respect to $d:\inte{X}\to \Omega_X\{-1\}$ is an analytic vector bundle $M$ endows with a morphism of abelian sheaves
    \[\nabla:M\to M\otimes_{\inte{X}}\Omega_X\{-1\},\]
    such that:
    \begin{itemize}
        \item For any open subset $Y\subset X$, $m\in\Gamma(Y,M)$ and $f\in \Gamma(Y,\inte{X})$, we have
        \[\nabla(fm)=f\nabla(m)+m\otimes d(f)\]
        \item The second order differential
        \[\nabla^{(1)}:M\otimes_{\inte{X}}\Omega\{-1\}\to \nabla^{(1)}:M\otimes_{\inte{X}}\Omega^{2}\{-2\}\]
        satisfies that $\nabla^{(1)}\circ \nabla=0$.
    \end{itemize}
\end{defi}

For any $\partial\in\Omega_X^\vee\{1\}$ and $f\in\inte{X}$, we define $\partial(f)$ as the derivation of $f$ by the image of $\partial$ in $\Omega_X^\vee$ induced by the natural map $\inte{X}\{1\}\to \inte{X}$. Furthermore, for any $\partial_1,\ \partial_2\in\Omega^\vee_X\{1\}$, define $[\partial_1,\partial_2]\in \Omega_X^\vee\{1\}$ as the image of the 
Lie bracket of $\partial_1$, $\partial_2$ (in $\Omega_{X}^\vee\{2\}$) in $\Omega_{X}^\vee\{1\}$.

Most of the time, we will use the dual form of the connections, which is also called `a derivation on $M$'.

\begin{defi}\label{defi: dual form connections}
    Let $M$ be a vector bundle on $X$. 
    
    A \emph{derivation} of $\Omega_{X}^\vee\{1\}$ on $M$ is a bilinear map of abelian sheaves
    \[\Omega_{X}^\vee\{1\}\times M\to M:\ (\partial,m)\mapsto \nabla_{\partial}(m)\]
    satisfying that:
    \begin{itemize}
        \item $\nabla_{f\partial}(m)=f\nabla_\partial(m)$ for all $f\in \inte{X},m\in M,\partial\in\Omega^\vee\{1\}.$

        \item $\nabla_\partial(fm)=\partial(f)m+f\nabla_{\partial}(m)$ for all $f\in \inte{X},m\in M,\partial\in\Omega^\vee\{1\}.$

        \item $[\nabla_{\partial_1},\nabla_{\partial_2}]=\nabla_{[\partial_1,\partial_2]}$ for any $\partial_1,\ \partial_2$ in $\Omega_{X}^\vee\{1\}$. The left hand side is the usual commutator of operators.
    \end{itemize}
\end{defi}

\begin{rmk}
    In the following, by a connection we always mean a connection with respect to $d:\inte{X}\to \Omega_X\{-1\}$.
\end{rmk}

The relationship between two notions can be described by the following lemma.

\begin{lemma}
    Let $M$ be a vector bundle on $X$. Then, taking duality induces a bijection between the set of connections $\nabla:M\to M\otimes_{\calO_X}\Omega_{X/B_\alpha}$ and the set of derivations of $\Omega_X\{1\}$ on $M$.
\end{lemma}

\begin{proof}
    This is a standard fact of $D$-modules. See, for example, \cite[Proposition 1.9]{etinghof}.
\end{proof}

Clearly, the restriction of any integrable connection on an \'etale open is integrable.

\begin{defi}[The sheaf $\mathrm{MIC}_{X,r}\{-1\}$]
Make the following definitions.
\begin{enumerate}[(1)]
    \item Let $U\in X_{\et}$ and $r$ be a positive integer, denote by $\mic_r(U)\{-1\}$ the set of isomorphic class of pairs $(M,\nabla:M\to M\otimes_{\calO_U}\Omega_{U/B_\alpha}\{-1\})$ where $M$ is a vector bundle of rank $r$ on $U$ and $\nabla$ is an integrable connection. For any morphism $F:U\to V$ in $X_\et$, the pull-back along $F$ induces a map $\mic_r(V)\{-1\}\to \mic_r(U)\{-1\}$. This defines a presheaf (of pointed sets) on $X_\et$.
    
    \item Keep the notations. Define the sheaf of integrable connections $\mathrm{MIC}_{X,r}\{-1\}$ as the sheafification of $\mic_r(-)\{-1\}$. For any $U\in X_\et$, we will use $\mathrm{MIC}_{X,r}(U)\{-1\}$ to denote $\Gamma(U,\mathrm{MIC}_{X,r}\{-1\})$.
\end{enumerate}
\end{defi}

\begin{defi}
    The sheaf of rank $r$ framed integrable connections $\mathrm{MIC}_{X,r}^{\mathrm{framed}}\{-1\}$ is defined as the subsheaf of $$\underline{\Hom}_\Z\left(\Omega_X^{\vee}\{1\},\underline{\End}_{\Z}(\inte{X}^r)\right)$$ consisting of the morphisms which satisfy the properties in Definition \ref{defi: dual form connections}. Equivalently, it is the sheaf sending each $Y\in\mathcal{X}$ to the isomorphic classes of $$(M,\nabla,\iota:\inte{X}^r|_Y\to M)$$ where $(M,\nabla)$ is an integrable connection and $\iota$ is an isomorphism of modules.
\end{defi}

An alternative definition of $\mathrm{MIC}_{X,r}\{-1\}$ is given by the following.

\begin{lemma}
    We have the equality $\mathrm{MIC}_{X,r}\{-1\}\cong \mathrm{MIC}_{X,r}^{\mathrm{framed}}/\mathrm{GL}_r(\inte{X})$ where $\mathrm{GL}_r(\calO_X)$ acts on $\mathrm{MIC}_{X,r}^{\mathrm{framed}}$ by acting on $\inte{X}^r$.
\end{lemma}

\begin{proof}
    Let $\Pi: \mathrm{MIC}_{X,r}^{\mathrm{framed}}\{-1\}\to \mathrm{MIC}_{X,r}\{-1\}$ be the map defined by forgetting $\iota$. This is a surjective sheaf homomorphism, because for any $U\in X_{\et}$ and any connection $(M,\lambda)$, $M$ is a free module sheaf on an analytic open cover of $U$. The remaining part can be obtained by a simple computation.
\end{proof}

\subsection{Extending cocycles on an \'etale covering}\label{maintech}

{From now on, we fix a system of $p^\infty$ roots of unity $\{\zeta_{p^n}\}$ in $\Q_p^\cyc$, see Subsection \ref{ToricSentheory} for explicit definitions.} Recall that we define
\begin{align*}
    \ep&=(1,\zeta_p,\zeta_{p^2},...)\in K^\flat\\
    t&=\log [\ep]\in \Bdrp{K} \\
\end{align*}

Let $f:\T\to A$ be a $\Bp{\alpha}$-algebra with a toric chart and $\Gamma$ as in Subsection \ref{toricalg}. For any standard map \'etale $A\to A'$, there is a unique extension of the derivation $\partial_{\gamma}$ on $A'$. For any such $A'$, an element $a\in A'$ and $\gamma\in \Gamma$, we define
$$\act{\gamma}{a}=(\exp\partial_\gamma)(a)=\sum_{i=0}^\infty \frac{\partial_\gamma^{\circ i}(a)}{i!}.$$ 
This is indeed a finite sum since $t$ is nilpotent in $A$.
More generally, for a matrix $P$ with coefficients in $A'$, define $\act{\gamma}{P}$ to be the $\exp\partial_\gamma$ action of each coefficient of $P$.

Fix an integer $M$ and define
$$M\Gamma:=\{Mx:x\in\Gamma\}.$$
We will prove the following theorem:
\begin{theo}\label{surject coh}
    For each cocycle $\Phi:M\Gamma\to \mathrm{GL}_m(A)$, there exists a standard \'etale cover $(A,A^+)\to(\widetilde{A},\widetilde{A}^+)$ such that $\Phi$ can be extended to a cocycle on $\Gamma$ with coefficients in $\widetilde{A}$.
\end{theo}

We only need to check on each \'etale stalk because a cocycle only depends on its values on generators of $\Gamma$. Suppose $x:\mathrm{Spa}(\bk,\bk^+)\to \mathrm{Spa}(A,A^+)$ is a geometric point and $A_x$ is its stalk. By definition, $A_x$ is a strictly henselian local ring, let $\m$ be the maximal ideal (see \cite{KEDLAYA_2018} Lemma 2.2.3).

\begin{defi}
    For any $r>0$, $$A_x\{t_1,t_2,\dots,t_r\}:=\varinjlim_{I,l}A'\langle \tfrac{t_1}{p^l},\tfrac{t_2}{p^l},\dots, \tfrac{t_r}{p^l}\rangle$$
    where $I$ is the index category consisting of all \'etale neighborhoods
    $$\xymatrix{\mathrm{Spa}(\bk,\bk^+)\ar[r]\ar[dr]_-x&\mathrm{Spa}(A',A'^+)\ar[d]\\ &\mathrm{Spa}(A,A^+)}$$
    and the colimit is taken along $I$ as well as all non-negative integers $l$. For any $f\in A_x\{t_1,t_2,\dots,t_r\}$ and $a_1,a_2,\dots,a_r\in A_x$,
    we say $f(a_1,a_2,\dots, a_r)$ converges if there exists $A'$ and $l$ such that $f\in A'\langle \tfrac{t_1}{p^l},\tfrac{t_2}{p^l},\dots, \tfrac{t_r}{p^l}\rangle$, $a_1,a_2,\dots, a_r\in A'$ and $f(a_1,a_2,\dots,a_r)$ (evaluate $t_i\mapsto a_i$ for each $i$) converges in $A'$.
\end{defi}

\begin{lemma}\label{lemma: converge max ideal}
    Let $(R,R^+)$ be a complete Tate--Huber pair and $f\in R\langle t_1,t_2,\dots,t_r\rangle.$ Then for any $a_1,a_2,\dots, a_r\in R^+$, $$f(a_1,a_2\dots, a_r)$$ converges.
\end{lemma}

\begin{proof}
    Since $a_i$ is power bounded for each $i$ and the product of finitely many bounded sets is bounded, we see that the set of all monomials of $a_1,a_2,\dots, a_r$ is bounded. Thus, $f(a_1,a_2,\dots, a_r)$ converges.
\end{proof}

\begin{lemma}
    For any $f\in A_x\{t_1,t_2,\dots,t_r\}$, and $a_1,a_2,\dots ,a_r\in \m$, $f(a_1,a_2,\dots,a_r)$ converges.
\end{lemma}

\begin{proof}
    Let $|-|_x$ be the (multiplicative) valuation at $x$. Then $|a_i|_x=0$ for any $i$ by definition. Thus, we can choose $A'$ such that $f\in A'\langle \tfrac{t_1}{p^l},\tfrac{t_2}{p^l},\dots, \tfrac{t_r}{p^l}\rangle$ and $a_i\in p^l{A'}^+$. The claim follows by Lemma \ref{lemma: converge max ideal}.
\end{proof}

\begin{lemma}\label{coprime}
    Suppose $R$ is a local ring with maximal ideal $\m$ and residue field $k$. Then the following statements hold.
    \begin{enumerate}[(1)]
        \item For any two monic polynomials $f,g$ which are coprime modulo $\m$, $f$ and $g$ are coprime in $R[x]$.

        \item For $r\geq 2$ and monic polynomials $f_1,f_2,\dots,f_r$ which are coprime to each other modulo $\m$, there exists polynomials $Q_1,Q_2,\dots,Q_r$ such that
    $$\sum_{i=1}^r Q_i\prod_{j\neq i}f_j=1.$$
    \end{enumerate}
\end{lemma}

\begin{proof}
    (1) For any positive integer $d$, let $$P_{R}^d=\{u\in R[x]:\deg u\leq d\}$$ and consider it as a free $R$-module. Then there is an $R$-module homomorphism
    $$F:P_R^{\deg g-1}\times P_{R}^{\deg f-1}\to P_{R}^{\deg f+\deg g-1}$$ such that $$F\left((u,v)\right)=uf+vg.$$ Since $f$ and $g$ are coprime in $k[x]$, $F$ is an isomorphism modulo $\m$. By Nakayama's lemma, $F$ is an isomorphism. Hence there exist $u,v\in R[x]$ such that $uf+vg=1$ and the claim follows.

    (2) We prove by induction on $r$. For $r= 2$, the proposition degenerates to (1). Suppose we have proved the result for $r-1$. By induction hypothesis, there are $q_1,q_2,\dots,q_{r-1}\in R[x]$ such that $$\sum_{i=1}^{r-1}q_i\prod_{j\leq r-1,j\neq i}f_i=1.$$ By (1), there exists $u,v\in R[x]$ such that $$u\prod_{i=1}^{r-1}f_i+vf_r=1.$$ Thus,
    \begin{align*}
        1&=u\prod_{i=1}^{r-1}f_i+vf_r\\ &=u\prod_{i=1}^{r-1}f_i+vf_r\sum_{i=1}^{r-1}\big(q_i\prod_{j\leq r-1,j\neq i}f_i\big)
    \end{align*}
    We may choose $Q_i=vq_i$ for $i<r$ and $Q_r=u$.
\end{proof}

In the following, for any matrix $T$ with coefficients in $A_x$, denote by $\overline{T}$ the projection of $T$ in the space of matrices with coefficients in $\bk$.

\begin{lemma}\label{first lemma}
    Let $A_x$ be as above and $\beta\in \Gamma$. Let $m_1,\dots, m_k$ be positive integers and $\Phi_i\in \mathrm{GL}_{m_i}(A_x)$ for any $i$. Suppose that each $\overline{\Phi}_i$ is upper triangular with only one eigenvalue $\alpha_i$ and $\alpha_i\neq \alpha_j,\ \forall i\neq j$.

    \begin{enumerate}[(1)]
        \item The map $Y\mapsto \Phi_iY-\act{\beta}{Y}\Phi_j$ from the space of $m_i\times m_j$ matrices over $A_x$ to itself is a bijection for any $i\neq j$.

        \item Let $H\in \mathrm{GL}_{m_1+\dots +m_k}(A_x)$ be a matrix of the form $$\begin{pmatrix}\Phi_1 & tX_{12} &\dots &tX_{1k}\\ tX_{21}&\Phi_2&\dots & tX_{2k}\\ \vdots&\vdots &&\vdots\\ tX_{k1} & tX_{k2}&\dots &\Phi_k\end{pmatrix}$$ where $X_{ij}$ has size $m_i\times m_j$. There exists a matrix $M$ of size $(\sum_{i=1}^km_i)\times (\sum_{i=1}^km_i)$ such that $(I+tM)^{-1}H\act{\beta}{(I+tM)}$ is block-diagonal
    $$diag(\Phi_1',\Phi_2',\dots, \Phi_k')$$ where $\Phi_i'$ has size $m_i\times m_i$ and $\Phi_{i}'\equiv \Phi_i\pmod{t}$.
    \end{enumerate}
\end{lemma}

\begin{proof}
    (1) We proceed by induction on $\alpha$ (recall $X$ is smooth over $\Bp{\alpha}$). For $\alpha=1$, $\exp\partial_\beta$ is the identity and by Nakayama's lemma, it suffices to prove that 
    the $k$-linear endomorphism \[Y\mapsto \overline{\Phi}_iY-Y\overline{\Phi}_j\]
    on the space of $m_i\times m_j$ matrices is an isomorphism.
    Let $\phi_i^l$ be the operator $Y\mapsto \overline{\Phi}_iY-\alpha_iY$ and $\phi_j^r$ be $$Y\mapsto Y\overline\Phi_j-\alpha_jY.$$ Then $$\overline\Phi_iY-Y\overline\Phi_j=(\alpha_i-\alpha_j)Y+\phi_i^l(Y)-\phi_j^r(Y).$$ Note that $\phi_i^l$ and $\phi_j^r$ are nilpotent operators and commute with each other, hence $$(\alpha_i-\alpha_j)+\phi_i^l-\phi_j^r$$ is invertible.

    Assume the claim holds for $\alpha-1$. Note that $\exp\partial_\beta$ acts trivially on $A_x/t$. Consider the exact sequence $$0\to \Mat_{m_i\times m_j}(A_x/t)\xrightarrow{t^{\alpha-1}} \Mat_{m_i\times m_j}(A_x)\xrightarrow{\mod t^{\alpha-1}} \Mat_{m_i\times m_j}(A_x/t^{\alpha-1})\to 0$$ the claim follows from the five lemma.

    (2) We first consider the case $k=2$. Let $$H=\begin{pmatrix}\Phi_1 & t^lX_{12}'\\ tX_{21}&\Phi_2\end{pmatrix}.$$
    By induction on $l\geq 1$, we claim that there exists a matrix $M_l$ such that $$(I+tM_l)^{-1}H\act{\beta}{(I+tM_l)}=\begin{pmatrix}\Phi_1' & t^lX_{12}'\\ tX_{21}&\Phi_2'\end{pmatrix}$$ for some matrix $X_{12}'$ of the same size as $X_{12}$, such that $\Phi_1'\equiv\Phi_1\pmod{t}$ and $\Phi_2'\equiv\Phi_2\pmod{t}$. Suppose we have proved the case for $l-1$. By an elementary calculation, $$\begin{pmatrix}I & t^{l-1}Y\\ 0&I\end{pmatrix}^{-1}\begin{pmatrix}\Phi_1 & t^{l-1}X_{12}\\ tX_{21}&\Phi_2\end{pmatrix} {\begin{pmatrix}I & t^{l-1}\act{\beta}{Y}\\ 0&I\end{pmatrix}}\equiv \begin{pmatrix}\Phi_1 & t^{l-1}(X_{12}+\Phi_1Y-Y\Phi_2)\\ tX_{21}&\Phi_2\end{pmatrix}\pmod{t^l}$$ and the bottom-left $m_2\times m_1$ block of
    $$\begin{pmatrix}I & t^{l-1}Y\\ 0&I\end{pmatrix}^{-1}\begin{pmatrix}\Phi_1 & t^{l-1}X_{12}\\ tX_{21}&\Phi_2\end{pmatrix} {\begin{pmatrix}I & t^{l-1}\act{\beta}{Y}\\ 0&I\end{pmatrix}}$$ is equal to $tX_{21}$.
    By (1), there exists a $Y$ such that $X_{12}+\Phi_1Y-Y\Phi_2\equiv 0 \pmod{t}$ and this proves the claim. Taking $l=\alpha$ (recall $t^\alpha=0$), we obtain a matrix $M_\alpha$ such that $$(I+tM_\alpha)^{-1}H\act{\beta}{(I+tM_\alpha)}=\begin{pmatrix}\Phi_1' & 0\\ tX_{21}&\Phi_2'\end{pmatrix}.$$ Then by an analogous induction argument, there exists a matrix $M_\alpha'$ such that
    $$(I+tM_{\alpha}')^{-1}\begin{pmatrix}\Phi_1' & 0\\ tX_{21}&\Phi_2'\end{pmatrix}\act{\beta}{(I+tM_\alpha')}=\begin{pmatrix}\Phi_1'' & 0\\0&\Phi_2''\end{pmatrix}$$ for some $\Phi_1''\equiv\Phi_1\pmod{t}$ and $\Phi_2''\equiv\Phi_2\pmod{t}$.

    For the general case, we use induction on $k$. Assume the result holds for $k-1\geq 2$; we now consider the case $k$. Let
    $$H=\begin{pmatrix}\Phi_1 & tX_{12} &\dots &tX_{1k}\\ tX_{21}&\Phi_2&\dots & tX_{2k}\\ \vdots&\vdots &&\vdots\\ tX_{k1} & tX_{k2}&\dots &\Phi_k\end{pmatrix}.$$
    Consider the $(\sum_{i=1}^{k-1}m_i)\times(\sum_{i=1}^{k-1}m_i)$ matrix
    $$H_1=\begin{pmatrix}\Phi_1 & tX_{12} &\dots &tX_{1,k-1}\\ tX_{21}&\Phi_2&\dots & tX_{2,k-1}\\ \vdots&\vdots &&\vdots\\ tX_{k-1,1} & tX_{k-1,2}&\dots &\Phi_{k-1}\end{pmatrix}.$$
    By the induction hypothesis, there is a $(\sum_{i=1}^{k-1}m_i)\times(\sum_{i=1}^{k-1}m_i)$ matrix $M_1$ such that $$(I_{\sum_{i=1}^{k-1}m_i}+tM_1)^{-1}H_1\act{\beta}{(I_{\sum_{i=1}^{k-1}m_i}+tM_1)}=\begin{pmatrix}\Phi_1' & 0 &\dots &0\\ 0&\Phi_2'&\dots & 0\\ \vdots&\vdots &&\vdots\\ 0 & 0&\dots &\Phi_{k-1}'\end{pmatrix}$$
    where $\Phi_i'\equiv \Phi_i\pmod{t}$. Hence if we let
    $$M_1'=\begin{pmatrix}M_1 &0\\0&0\end{pmatrix}$$
    be a $(\sum_{i=1}^km_i)\times(\sum_{i=1}^km_i)$ matrix whose top-left $\sum_{i=1}^{k-1}m_i$ square corner is equal to $M_1$ and the other elements are all $0$. Then
    $$(I_{\sum_{i=1}^{k}m_i}+tM_1')^{-1}H\act{\beta}{(I_{\sum_{i=1}^{k}m_i}+tM_1')}=\begin{pmatrix}\Phi_1' & 0 &\dots &0 &tX_{1k}\\ 0&\Phi_2'&\dots & 0& tX_{2k}\\ \vdots&\vdots &&\vdots &\vdots\\ 0 & 0&\dots &\Phi_{k-1}' &tX_{k-1,k}\\ tX_{k1}& tX_{k2} &\dots &tX_{k,k-1} &\Phi_k\end{pmatrix}.$$
    Next we apply the induction hypothesis to
    $$H_2=\begin{pmatrix}\Phi_2' & 0 &\dots &0 &tX_{2k}\\ 0&\Phi_3'&\dots & 0& tX_{3k}\\ \vdots&\vdots &&\vdots &\vdots\\ 0 & 0&\dots &\Phi_{k-1}' &tX_{k-1,k}\\ tX_{k2}& tX_{k3} &\dots &tX_{k,k-1} &\Phi_k\end{pmatrix}.$$
    Hence there exists a matrix $M_2\in M_{\sum_{i=2}^km_i}(A_x)$ such that
    $$(I_{\sum_{i=2}^{k}m_i}+tM_2)^{-1}H_2\act{\beta}{(I_{\sum_{i=2}^{k}m_i}+tM_2)}=\begin{pmatrix}\Phi_2'' & 0 &\dots &0\\ 0&\Phi_3''&\dots & 0\\ \vdots&\vdots &&\vdots\\ 0 & 0&\dots &\Phi_{k}'' \end{pmatrix}$$
    where $\Phi_i''\equiv \Phi_i\pmod{t}$. Let
    $$M_2'=\begin{pmatrix}0 &0\\0&M_2\end{pmatrix}$$ be the matrix in $M_{\sum_{i=1}^km_i}(A_x)$ whose bottom-right $\sum_{i=2}^{k}m_i$ square corner is equal to $M_2$ and the other elements are all $0$. We have
    \begin{align*}&(I_{\sum_{i=1}^{k}m_i}+tM_2')^{-1}(I_{\sum_{i=1}^{k}m_i}+tM_1')^{-1}H\act{\beta}{(I_{\sum_{i=1}^{k}m_i}+tM_1')}\act{\beta}{(I_{\sum_{i=1}^{k}m_i}+tM_2')}\\
    &=(I_{\sum_{i=1}^{k}m_i}+tM_2')^{-1}\begin{pmatrix}\Phi_1' & 0 &\dots &0 &tX_{1k}\\ 0&\Phi_2'&\dots & 0& tX_{2k}\\ \vdots&\vdots &&\vdots &\vdots\\ 0 & 0&\dots &\Phi_{k-1}' &tX_{k-1,k}\\ tX_{k1}& tX_{k2} &\dots &tX_{k,k-1} &\Phi_k\end{pmatrix}\act{\beta}{(I_{\sum_{i=1}^{k}m_i}+tM_2')}\\
    &=\begin{pmatrix}\Phi_1' & 0 &\dots &0 &tX_{1k}\\ 0&\Phi_2''&\dots & 0& 0\\ \vdots&\vdots &&\vdots &\vdots\\ 0 & 0&\dots &\Phi_{k-1}'' &0\\ tX_{k1}& 0 &\dots &0 &\Phi_k''\end{pmatrix}.\end{align*}
    Finally, we apply the case $k=2$ to
    $$\begin{pmatrix}\Phi_1' &tX_{1k}\\tX_{k1}& \Phi_k''\end{pmatrix}$$ and the claim follows.
\end{proof}

Recall we fix a basis $\{\gamma_i:i=1,2,\dots,d\}$ of $\Gamma$ such that
\begin{align*}\gamma_{i}(T_j)=\left\{\begin{matrix}[\ep]T_j &\text{if $i=j$}\\ T_j &\text{otherwise}\end{matrix}\right.\end{align*}

\begin{defi}[Simple form \& Canonical form]
For $\Phi\in \mathrm{GL}_m(A_x)$, we say that
\begin{enumerate}[(1)]
    \item $\Phi$ is of the simple form (or $\Phi$ is simple) if $\overline\Phi$ is upper triangular with constant diagonals.

    \item $\Phi$ is of the strict simple form (or $\Phi$ is strictly simple) if $\Phi$ is simple and the diagonals of $\overline{\Phi}$ are all equal to $1$.
\end{enumerate}
\end{defi}

\begin{lemma}\label{biaozhunhua}
    For each $i$, let $\psi_i=\Phi_{M\gamma_i}\in \mathrm{GL}_m(A_x)$. There exists a matrix $V\in \mathrm{GL}_m(A_x)$ and integers $d_1,\dots,d_u>0$ satisfying $\sum_{i=1}^ud_i=m$, such that for each $i$ $$V^{-1}\psi_i\act{M\gamma_i}{V}$$ is of the form $$\begin{pmatrix} \Phi_{1} &0 &\dots &0 \\ 0 &\Phi_{2} &\dots &0 \\ \vdots &\vdots &&\vdots\\ 0&0&\dots &\Phi_{u}\end{pmatrix}$$ where $\Phi_j$ is a $d_j\times d_j$ matrix such that $\overline{\Phi}_j$ is upper triangular with constant diagonal.
\end{lemma}

\begin{proof}
    We proceed by induction on $m=\mathrm{rank}(V)$. If $m=1$, there is nothing to prove. Assume the lemma holds for $m\leq r-1$; we now prove it for $m=r$. Note that the cocycle condition for $\Phi$ implies that $\psi_i\psi_j\equiv\Phi_{M\gamma_i+M\gamma_j}\equiv\psi_j\psi_i\pmod{t}$ for any $i,j$.

    If each $\overline{\psi_i}$ has a unique eigenvalue, since they commute with each other, there is a matrix $\overline{V}\in \mathrm{GL}_r(\bk)$ such that $\overline{V}^{-1}\overline{\psi_i}\overline{V}$ is upper triangular for all $i$. Thus any lifting of $\overline{V}$ in $\mathrm{GL}_r(A_x)$ satisfies the property.

    Suppose there is a $\psi_{i_0}$ such that $\overline{\psi_{i_0}}$ has at least two eigenvalues. Let $$P(T):=\det(TI_r-\psi_{i_0})\in A_x[T]$$ be the characteristic polynomial. Since $A_x$ is strictly henselian, we may assume $$P=P_1P_2\dots P_s$$ where each $P_i$ is monic and equals some $(T-\alpha_i)^{m_i}$ modulo $\m$ with $\alpha_i\neq \alpha_j$ for any $i\neq j$. By Lemma \ref{coprime}, there exist $Q_1,\dots, Q_s\in A_x[T]$ such that
    $$\sum_{i=1}^s Q_i\prod_{j\neq i}P_j=1.$$
    Hence, $$I=\sum_{i=1}^s Q_i(\psi_{i_0})\prod_{j\neq i}P_j(\psi_{i_0}).$$ Note that $\prod_{j=1}^sP_j(\psi_{i_0})=P(\psi_{i_0})=0$ by the Cayley–Hamilton theorem, we deduce that $$A_x^{m}=\oplus_{i=1}^s \ker P_{i}(\psi_{i_0}).$$
    In matrix terms, there exists a $V_1'\in \mathrm{GL}_m(A_x)$ such that $V_1'^{-1}\psi_{i_0}V_1'$ is of the form
    $$diag(\psi'_{i_0,1},\dots, \psi'_{i_0,s})$$
    such that each $\psi'_{i_0,j}$ is simple. Moreover, the eigenvalues of $\overline{\psi'_{i_0,j_1}}$ and $\overline{\psi'_{i_0,j_2}}$ are distinct whenever $j_1\neq j_2$. Hence,  $V_1'^{-1}\psi_{i_0}\act{M\gamma_{i_0}}{V_1'}$ is of the form
    $$\begin{pmatrix}\widetilde{\psi}_{i_0,1,1} &t\widetilde{\psi}_{i_0,1,2} &\dots &t\widetilde{\psi}_{i_0,1,s}\\ t\widetilde{\psi}_{i_0,2,1} &\widetilde{\psi}_{i_0,2,2} &\dots &t\widetilde{\psi}_{i_0,2,s}\\ \vdots &\vdots &&\vdots\\ t\widetilde{\psi}_{i_0,s,1}&t\widetilde{\psi}_{i_0,s,2}&\dots &\widetilde{\psi}_{i_0,s,s}\end{pmatrix}$$
    By Lemma \ref{first lemma}, we can choose a suitable $V_1\equiv V_1'\pmod{t}$ such that $V_1^{-1}\psi_{i_0}\act{M\gamma_{i_0}}{V_1}$ is of the form
    $$\begin{pmatrix}\widetilde{\psi}_{i_0,1,1} &0 &\dots &0\\ 0 &\widetilde{\psi}_{i_0,2,2} &\dots &0\\ \vdots &\vdots &&\vdots\\ 0&0&\dots &\widetilde{\psi}_{i_0,s,s}\end{pmatrix}$$

    For any $i\neq i_0$, write $V_1^{-1}\psi_{i}\act{M\gamma_i}{V_1}$ in the same block decomposition as $V_1^{-1}\psi_{i_0}\act{M\gamma_{i_0}}{V_1}$, say
    $$V_1^{-1}\psi_{i}\act{M\gamma_i}{V_1}=(\widetilde{\psi}_{i,u,v})_{1\leq u,v\leq d}.$$
    Then since $\psi_{i}\act{M\gamma_i}{\psi_{i_0}}=\Phi_{M\gamma_i+M\gamma_{i_0}}=\psi_{i_0}\act{M\gamma_{i_0}}{\psi_i}$ (the cocycle condition), $$\widetilde{\psi}_{i_0,u,u}\act{M\gamma_{i_0}}{\widetilde{\psi}_{i,u,v}}=\widetilde{\psi}_{i,u,v}\act{M\gamma_{i}}{\widetilde{\psi}_{i_0,v,v}}$$ for any $u,v$. Thus $\widetilde{\psi}_{i,u,v}=0$ for any $u\neq v$ by Lemma \ref{first lemma}. Finally, we apply the induction hypothesis to each diagonal block to complete the proof.
\end{proof}

\begin{lemma}
    Let $U\in \Mat_{n\times n}(A_x)$ be a nilpotent matrix such that $U^n=0$.

    \begin{enumerate}[(1)]
        \item Let $$(1+U+X)^{\tfrac{1}{M}}:=\sum_{a\geq 0}\binom{\tfrac{1}{M}}{a}(U+X)^a$$ which is viewed as an infinite sum of formal power series of the coordinates of $X=(x_{ij})_{1\leq i,j\leq n}$. Then it converges under the $(x_{ij}:1\leq i,j\leq n)$-adic topology in $$A_x[[x_{ij}:1\leq i,j\leq n]].$$

        \item The sum $(1+U+X)^{\tfrac{1}{M}}\in A_x[[x_{ij}:1\leq i,j\leq n]]$ lies in $$A_x\{x_{ij}:i,j=1,2,\cdot, n\}.$$
    \end{enumerate}
\end{lemma}

\begin{proof}
    (1) We expand $(U+X)^a$ as a noncommutative polynomial of $U$ and $X$. Note that each monomial is of the form
    $$U^{i_1}X^{j_1}U^{i_2}\dots U^{i_w}X^{j_w}\dots (i_1\geq 0,i_w>0\text{ for any $w>1$},j_w>0 \text{ for any $t>0$}).$$ Let $\deg_U=i_1+i_2+\dots$ and $\deg_X=j_1+j_2+\dots$. Since $U^n=0$, the monomial is nonzero only if each $i_w<n$. This implies all nonzero terms satisfy
    $$\deg_U\leq (n-1)(\deg_X+1)$$ and thus $\deg_X\geq \tfrac{a-n+1}{n}$. Hence the polynomial (of $x_{ij}$) $(U+X)^a$ lies in $$(x_{ij}:1\leq i,j\leq n)^{[\tfrac{a}{n}]-1}A_x[x_{ij}:1\leq i,j\leq n]$$ for all $a$. The claim follows.

    (2) Suppose $U\in \Mat_{n\times n}(A_x)$ for some standard \'etale neighborhood $\Spa(A',A'^+)$ of $x$.

    We expand $\binom{\tfrac{1}{M}}{a}(U+X)^a$ as a noncommutative polynomial of $U$ and $X$. Recall we have proved that each nonzero term in the expansion of $(U+X)^a$ satisfies $\deg_U\leq (n-1)(\deg_X+1)$. Hence $\binom{\tfrac{1}{M}}{a}(U+X)^a$ is a finite sum of
    $$cU^{i_1}X^{j_1}U^{i_2}\dots U^{i_w}X^{j_w}\dots$$ where $$c=\binom{\tfrac{1}{M}}{i_1+j_1+i_2+\dots }$$ and $$\deg_U=\sum_w i_w\leq (n-1)(\sum_{w}j_w+1)=(n-1)(\deg_X+1).$$
    Note that $$c=\binom{\tfrac{1}{M}}{\deg_U+\deg_X}=\tfrac{\prod_{b=1}^{\deg_U+\deg_X}\left(1-M(b-1)\right)}{M^{\deg_U+\deg_X}(\deg_U+\deg_X)!}$$ hence
    $$v_p(c)\geq -v_p(M)(\deg_U+\deg_X)-\tfrac{\deg_X+\deg_U}{p-1}\geq J\deg_X+H$$ for some constants $J$ and $H$ depending only on $M$ and $n$. Suppose the coefficients of $U$ lie in $p^{-l}A'_0$ for some open bounded ring $A'_0\subseteq A'$ and some integer $l$, the above argument implies the coefficient of $X^d$ lies in $p^{-\left(J+(n-1)l\right)d-H-nl+l}A'_0$ and thus $$(1+U+X)^{\tfrac{1}{M}}\in A'\langle \tfrac{x_{ij}}{p^{J+(n-1)l}}:i,j=1,2,\dots, n\rangle.$$
\end{proof}

Thus for any $U\in Mat_{n\times n}(A_x)$ such that $U^n=0$ and $X\in \Mat_{n\times n}(\m)$, we can define
$$(1+U+X)^{\tfrac{1}{M}}=\sum_{a\geq 0}\binom{\tfrac{1}{M}}{a}(U+X)^a.$$ By its form, this depends only on $U+X$, not on the choice of $U$ and $X$. Hence we define $\phi^{\tfrac{1}{M}}$ for any strictly simple $\phi$ as $\left(1+(\phi-1)\right)^{\tfrac{1}{M}}$.

\begin{lemma}\label{kaifang}
    Suppose $\beta\in \Gamma$. For any $\Phi\in \mathrm{GL}_m(A_x)$, there exists a $\Phi_1\in \mathrm{GL}_m(A_x)$ such that
    $$\Phi_1\cdot\act{\beta}{\Phi_1}\cdot\ \dots\ \cdot\act{(M-1)\beta}{\Phi_1}=\Phi.$$

    Moreover, if $\Phi$ and $\Phi_2$ are of the simple form such that $\Phi_2^M\equiv \Phi\mod t$, there exists a unique $\Phi_1\equiv\Phi_2\mod t$ satisfying the above property.
\end{lemma}

\begin{proof}
    We proceed by induction on $\alpha$ (see the proof to the lemma \ref{first lemma}). By lemma \ref{first lemma}, we may assume $\Phi$ is of the simple form.

    First, assume $\alpha=1$, which implies $t=0$ and $\exp\partial_\beta=id$. Since $M$ is invertible in $A_x$ and $A_x$ is strictly henselian, there exists $\lambda\in A_x$ such that $\lambda^{-M}\Phi$ is strictly simple. Thus we can choose $$\Phi_1=\lambda\left(\lambda^{-M}\Phi\right)^{\tfrac{1}{M}}.$$  The second claim is trivial in this case, since when $\alpha=1$, two elements in $A_x$ that are equal modulo $t$ must be equal.

    Suppose that we have already proved the case for $\alpha-1$. By induction hypothesis, we can find a $\Phi_2\in \mathrm{GL}_m(A_x)$ which is of the simple form such that
    $$\Phi_2\cdot\act{\beta}{\Phi_2}\cdot\ \dots\ \cdot\act{(M-1)\beta}{\Phi_2}\equiv\Phi\mod t^{\alpha-1}.$$
    We prove that there is a unique $Y\in Mat_m(A_x/t)$ such that $\Phi_1:=\Phi_2+t^{\alpha-1}Y$ satisfies
    $$\Phi_1\cdot\act{\beta}{\Phi_1}\cdot\ \dots\ \cdot\act{(M-1)\beta}{\Phi_1}=\Phi.$$
    In fact (note that $t^{\alpha}=0$),
    \begin{align*}&\ \ \ \ (\Phi_2+t^{a-1}Y)(\act{\beta}{\Phi_2}+t^{\alpha-1}\act{\beta}{Y})\dots (\act{(M-1)\beta}{\Phi_2}+t^{\alpha-1}\act{(M-1)\beta}{Y})\\&=(\Phi_2+t^{\alpha-1}Y)(\act{\beta}{\Phi_2}+t^{\alpha-1}Y)\dots (\act{(M-1)\beta}{\Phi_2}+t^{\alpha-1}Y)\\ &=\Phi_2\cdot\act{\beta}{\Phi_2}\cdot\ \dots\ \cdot\act{(M-1)\beta}{\Phi_2}+t^{\alpha-1}\sum_{u=1}^M\Phi_2\act{\beta}{\Phi_2}\dots\act{(u-1)\beta}{\Phi_2}Y\act{(u+1)\beta}{\Phi_2}\dots \act{(M-1)\beta}{\Phi_2}\\
    &=\Phi_2\cdot\act{\beta}{\Phi_2}\cdot\ \dots\ \cdot\act{(M-1)\beta}{\Phi_2}+t^{\alpha-1}\sum_{u=1}^M\Phi_2^{u-1}Y\Phi_2^{M-u}\end{align*}

    Thus we only need to prove that for any simple $B\in \mathrm{GL}_m(A_x)$ the homomorphism $$Y\mapsto\sum_{i=1}^M B^{i-1}YB^{M-i}$$ from the space of $m\times m$ matrices to itself is bijective. By Nakayama's lemma, we may assume that $A_x$ is a field (of characteristic $0$). The field case follows from the next lemma.
\end{proof}

\begin{lemma}
    Let $K$ be a field of characteristic $0$, $B$ be an invertible $n\times n$ upper triangular matrix with constant diagonal and $M\geq 1$ an integer. Then the linear endomorphism on the space of rank $n$ matrix
    $$F:Y\mapsto\sum_{i=1}^M B^{i-1}YB^{M-i}$$
    is an automorphism.
\end{lemma}

\begin{proof}
    It suffices to prove for an unipotent $B$. We can prove by induction on $n$. The case for $n=1$ is trivial. Assume we have proved for $n-1$. Let $$X=\begin{pmatrix}x_{11}&x_{12}&\cdots&x_{1n}\\ x_{21} &x_{22}&\cdots &x_{2n}\\ \vdots &\vdots &&\vdots\\ x_{n1}&x_{n2}&\cdots&x_{nn} \end{pmatrix}$$ such that $F(X)=0$. Look at the first column of $F(X)$, it is equal to
    $$\sum_{i=1}^M B^{i-1}\begin{pmatrix}x_{11}\\ x_{21}\\ \vdots\\ x_{n1}\end{pmatrix}$$ by a direct calculation. Thus $F(X)=0$ implies the first column of $X$ is zero. Now we can use the induction hypothesis for the matrix $$X'=\begin{pmatrix}x_{22}&x_{23}&\cdots&x_{2n}\\ x_{32} &x_{33}&\cdots &x_{3n}\\ \vdots &\vdots &&\vdots\\ x_{n2}&x_{n3}&\cdots&x_{nn} \end{pmatrix}.$$ Thus every element of $X$ except the first row is zero. Finally, another direct calculation implies $X=0$. Now we prove that $F$ is injection, hence it is a surjection since $F$ is an endomorphism.
\end{proof}

Now we are ready to prove the main theorem.

\begin{proof}[Proof of Theorem \ref{surject coh}]
    By the Lemma \ref{biaozhunhua}, we may assume $\psi_i$ is simple. For each $i$, choose $\lambda_i\in A_x^\times$ such that $\lambda_i^{-M}\psi_i$ is strictly simple.

    First suppose $\alpha=1$ and thus each $\exp\partial_{\gamma_i}=id$. We can choose $$\psi_i'=\lambda_i\left(\lambda_i^{-M}\psi_i\right)^{\tfrac{1}{M}}.$$ They commute to each other since $(\lambda_i^{-M}\psi_i)^{\tfrac{1}{M}}$ is a power series of $(\lambda_i^{-M}\psi_i-1)$.

    Come to general case.

    By the Lemma \ref{kaifang} and the case for $\alpha=1$, we can choose for each $i$ a simple $\psi_i'$ such that:
    $$
        \left\{\begin{matrix}\psi_i'\psi_j'\equiv \psi_j'\psi_i'\mod t \ \forall i,j\\
       \psi_i'\cdot\act{\gamma_i}{\psi_i'}\cdot\ \dots\ \cdot \act{(M-1)\gamma_i}{\psi_i'}=\psi_i\ \forall i\end{matrix}\right. 
$$

    Now it remains to prove that $\psi_i'\act{\gamma_i}{\psi_j'}=\psi_j'\act{\gamma_j}{\psi_i'}$. We prove it step by step. First we prove that \begin{equation}\label{eq:0001}\psi_i'\act{\gamma_i}{\psi_j}=\psi_j\act{M\gamma_j}{\psi_i'}.\end{equation}
    Note that both sides are congruent modulo $t$.
    By the uniqueness part of the Lemma \ref{kaifang} and
    \begin{align*}
        &\big(\act{M\gamma_j}{\psi_i'}\big)\act{\gamma_i}{\big(\act{M\gamma_j}{\psi_i'}\big)}\dots \act{(M-1)\gamma_i}{\big(\act{M\gamma_j}{\psi_i'}\big)}\\
        &=\act{M\gamma_j}{\psi_i'}\act{\gamma_i+M\gamma_j}{\psi_i'}\dots \act{(M-1)\gamma_i+M\gamma_j}{\psi_i'}\\
        &=\act{M\gamma_j}{\big({\psi_i'}\act{\gamma_i}{\psi_i'}\dots \act{(M-1)\gamma_i}{\psi_i'}\big)}\\
        &=\act{M\gamma_j}{\psi_i},
    \end{align*}
    it suffices to prove
    $$\big(\psi_j^{-1}\psi_i'\act{\gamma_i}{\psi_j}\big)\cdot\act{\gamma_i}{\big(\psi_j^{-1}\psi_i'\act{\gamma_i}{\psi_j}\big)}\dots \act{(M-1)\gamma_i}{\big(\psi_j^{-1}\psi_i'\act{\gamma_i}{\psi_j}\big)}=\act{M\gamma_j}{\psi_i}.$$
    In fact,
    \begin{align*}
        LHS&=\big(\psi_j^{-1}\psi_i'\act{\gamma_i}{\psi_j}\big)\cdot\big(\act{\gamma_i}{\psi_j}^{-1}\act{\gamma_i}{\psi_i'}\act{2\gamma_i}{\psi_j}\big)\dots \big(\act{(M-1)\gamma_i}{\psi_j}^{-1}\act{(M-1)\gamma_i}{\psi_i'}\act{M\gamma_i}{\psi_j}\big)\\
        &=\psi_j^{-1}\big(\psi_i'\cdot\act{\gamma_i}{\psi_i'}\cdot\ \dots\ \cdot \act{(M-1)\gamma_i}{\psi_i'}\big)\act{M\gamma_i}{\psi_j}\\
        &=\psi_j^{-1}\psi_i\act{M\gamma_i}{\psi_j}=RHS.
    \end{align*}
    The last equality is the cocycle condition of $\psi_i$.

    Next we prove that
    $$\psi_i'\act{\gamma_i}{\psi_j'}=\psi_j'\act{\gamma_j}{\psi_i'}\Leftrightarrow \psi_j'^{-1}\psi_i'\act{\gamma_i}{\psi_j'}=\act{\gamma_j}{\psi_i'}$$
    and this time it suffices to prove
    $$\big(\psi_j'^{-1}\psi_i'\act{\gamma_i}{\psi_j'}\big)\cdot \act{\gamma_i}{\big(\psi_j'^{-1}\psi_i'\act{\gamma_i}{\psi_j'}\big)}\dots \act{(M-1)\gamma_i}{\big(\psi_j'^{-1}\psi_i'\act{\gamma_i}{\psi_j'}\big)}=\act{\gamma_j}{\psi_i}$$
    by the uniqueness part of the Lemma \ref{kaifang}. Now do the same calculation as above.
    \begin{align*}
        LHS=\psi_j'^{-1}\psi_i\act{M\gamma_i}{\psi_j'}=\act{\gamma_j}{\psi_i}
    \end{align*}
    where the last equality holds by (\ref{eq:0001}).
\end{proof}
\subsection{Local construction}

Let $f:\T\to A$ be a $\Bp{\alpha}$ algebra with a toric chart. Denote by $X=\mathrm{Spa}(A,A^+)$ and $\overline{X}=\mathrm{Spa}(\overline{A},\overline{A}^+)$. In this section, we will construct a map $$\mathfrak{LRH}_f:\mathrm{H}^1\big(\overline{X}_{\pet},\mathrm{GL}_r(\mathbb{B}_{\alpha})\big)\to \mathrm{MIC}_r(X)\{-1\}.$$ By definition, the left hand side is the set of isomorphic classes of $v$-vector bundles on the ringed space $(\overline{X}_{\pet},\mathbb{B}_{\alpha})$.

Maintain the notations in Subsection \ref{ToricSentheory}. Recall we have a continuous action, which is called the \textbf{natural} action, coming from the functorlaity of $\Bpp{\alpha}$
$$\Gamma\times \widehat{A}_\infty\to \widehat{A}_\infty,\ (\gamma,x)\mapsto \gamma x.$$
Under this action, $A_n$ is equal the set of $p^n\Gamma$-analytic vectors. Take derivation, we get a Lie algebra action 
$$\Gamma[\tfrac{1}{p}]\cong \mathrm{Lie}(\Gamma)\to \mathrm{Der}_{B^+_{\dR,\alpha}}(A_\infty),\ \gamma\mapsto \partial_\gamma.$$

Let ${\Rep}_{r,\Gamma}^{\mathrm{cnt}}(\widehat{A}_{\infty})$ denote the category of rank $r$ projective $\widehat{A}_\infty$-modules with continuous semi-linear $\Gamma$-actions.

\begin{prop}
    There is a canonical equivalence between the category of $\mathbb{B}_{\alpha}$-vector bundles on $\overline X_v$ and the category $\Rep_{r,\Gamma}^{\mathrm{cnt}}(\widehat{A}_{\infty})$.
\end{prop}

\begin{proof}
    This is obvious since a $v$-vector bundle over $\mathrm{Spa}(\widehat{A}_\infty,\widehat{A}^+_\infty)$ is always an analytic bundle (See \cite{kedlaya2019relative} Theorem 3.5.8). The functor from the left to the right is taking global section on $\mathrm{Spa}(\widehat{A}_\infty,\widehat{A}^+_\infty)$.
\end{proof}

We fix a projective module $M$ of rank $r$ over $\widehat{A}_\infty$ with a $\Gamma$ action. Since $A_\infty$ is dense in $\widehat{A}_\infty$, by the general theory of adic spaces there exists an $w\geq 0$ as well as an analytic open covering $f:A_w\to A_w'$ such that $A_w'\widehat{\otimes}_{A_w}M$ is free on $A_w'\widehat{\otimes}_{A_w}\widehat{A}_\infty$. Recall we use $\gamma_*$ to denote the Galois action on $A_\infty$.

\begin{lemma}
    There exists an analytic covering $f_1:A_w\to A_w''$ such that for any $\gamma\in \Gamma/p^w\Gamma$, the homomorphism $f_1\circ\gamma_*:A_w\to A_{w}''$ factors through $f$.
\end{lemma}

\begin{proof}
    Note that the condition that $f_1\circ\gamma_*$ factors through $f$ is equivalent to $f_1$ factoring through $f\circ\gamma_*^{-1}$. Then we can choose
    $$A_w''=\bigotimes_{\gamma\in \Gamma/p^w\Gamma}\gamma_*A_w'$$
    where $\gamma_*A_w'$ is the $A_w$-algebra such that the underlying ring is $A_w'$ but is viewed as an $A_w$-algebra via
    $$f\circ\gamma_*^{-1}:A_w\to A_w'.$$
    For any $\gamma\in\Gamma/p^w\Gamma$, let $g:A_w'\to A_w''$ be the map
    $$x\mapsto 1\otimes1\otimes\dots\otimes x\otimes 1\dots\otimes 1$$ where the right-hand-side is the element such that there is an $x$ at the factor $\gamma_*A_w'$ and $1$ at the other factors. By definition, we have $g\circ f\circ\gamma_*^{-1}:A_w=f_1$.
\end{proof}

\begin{rmk}
    The same argument shows that for any \'etale covering $f:A_w\to A_w'$, there exists an \'etale covering $f_1:A_w\to A_w''$ such that for any $\gamma\in \Gamma/p^w\Gamma$, the homomorphism $f_1\circ\gamma_*$ factors through $f$.
\end{rmk}

Choose such an $f_1:A_w\to A_w''$, we have
$$A_w''\otimes_AA_w=A_w''\otimes_{A_w}(A_w\otimes_AA_w)=A_w''\otimes_{A_w}\prod_{\Gamma/p^w}A_w=\prod_{\Gamma/p^w}A_w''.$$
The above lemma implies that the natural map $A_w\xrightarrow{x\mapsto 1\otimes x}A_w''\otimes_AA_w$ factors through $f$. Since
$$A_w''\widehat{\otimes}_A\widehat{A}_\infty=(A_w''\otimes_AA_w)\widehat{\otimes}_{A_w}\widehat{A}_\infty$$
$A_w''\widehat{\otimes}_AM$ is free over $A_w''\widehat{\otimes}\widehat{A}_\infty$.

In conclusion, we have proved that:

\begin{lemma}\label{refine}
    There exists an \'etale covering $A\to A'$ such that $M_{A'}:=A'\widehat{\otimes}_AM$ is free over $A'\otimes_A\widehat{A}_\infty$.
\end{lemma}

\begin{rmk}
    The same calculation as well as the previous remark implies that: for any \'etale covering $f:A_w\to A_w'$, there exists an \'etale covering $A\to A'$ such that $A_w\xrightarrow{x\mapsto 1\otimes x}A_w\otimes_AA'$ factors through $f$.
\end{rmk}

By definition, $A'\widehat{\otimes}_A\widehat{A}_\infty\cong \widehat{A}_\infty'$ and thus by Sen theory (cf. Theorem \ref{Sen}):

\begin{theo}\label{theo:toric decomp}
    There exists an $n$ such that for any $n'\geq n$, the set of $p^{n'}\Gamma$-analytic vectors of $M$, which is denoted by $M_{A',n'}$ is a rank $r$ finite free module over $A_{n'}'$. Moreover, $$M_{A',n'}\otimes_{A'_{n'}}\widehat{A}_\infty'=M_{A'}.$$
\end{theo}

\begin{defi}
    For any $\gamma\in \Gamma$ and $m\in M_{A',n}$, let
    $$\nabla_{\gamma}(m):=\lim_{u\in\mathbb{Z}_p\to 0}\tfrac{(u\cdot\gamma)m-m}{u}$$ where the product is scalar product. By an elementary calculation, $$\nabla_{\gamma}(am)=\partial_{\gamma}(a)+a\nabla_{\gamma}(m).$$
\end{defi}

By the isomorphism $\Omega_{\mathrm{Spa}(A',A'^+)}^\vee\{1\}\cong A'\otimes_{\mathbb{Z}_p}\Gamma$ such that $\gamma_i$ maps to $tT_i\partial_i$, we can extend the above action to an integrable connection
$$\nabla:A'_{n}\otimes_{A'}\Omega^\vee_{\mathrm{Spa}(A',A'^+)}\{1\}\to \End_{B_{\alpha}^+}(M_{A',n}).$$
This is a connection since $\partial_{\gamma_i}=t T_i\partial_i$ on $A_n'$.

We fix an isomorphism of $A'_{n}$-modules ${A'_n}^r\cong M$. Let $e_1,e_2,\dots, e_r$ be the canonical basis of ${A'_n}^r$ and $\Gamma$ act on ${A'_n}^r$ via the isomorphism. Suppose
$$[\gamma e_1,\gamma e_2, \dots, \gamma e_r]=[e_1,e_2,\dots,e_r]\phi_{\gamma}.$$
By an elementary calculation,
$$\nabla_{\gamma}:[e_1,e_2,\dots,e_r]\mapsto [e_1,e_2,\dots,e_r]\phi_{\gamma}$$
where $\phi_{\gamma}=\lim_{u\in\mathbb{Z}_p\to 0}\tfrac{\Phi_{u\cdot\gamma}-I_r}{u}$.

We view the above connection as an element in $\mathrm{MIC}_r(A'_n)\{-1\}$. In the following, we will use $\nabla$ to denote this connection but viewed as an element in $\mathrm{MIC}_r(A'_n)\{-1\}$.

\begin{theo}
    The element $\nabla\in \mathrm{MIC}_r(A_n')\{-1\}$ descends to an element in $\mathrm{MIC}_r(A')\{-1\}$.
\end{theo}

\begin{proof}
    It is enough to prove that for each $\delta\in\Gamma$, there exists an \'etale covering $A'\to B$ such that the connection $\nabla_{B}$ (base change to $A'_n\otimes_A'B$) is isomorphic to $\delta_*(\nabla_{B})$. The latter one is the connection such that (via the isomorphism $\Omega_{\mathrm{Spa}(A',A'^+)}^\vee\{1\}\cong A'\otimes_{\mathbb{Z}_p}\Gamma$) $$\delta_*(\nabla_{B,\gamma_i}):[e_1,e_2,\dots,e_r]\mapsto [e_1,e_2,\dots,e_r]\cdot(\delta_*\otimes_{A'} id_{B})(\phi_{\gamma_i}).$$

    This is equivalent to showing that there exists an \'etale covering $A'\to B$ and a matrix $V\in \mathrm{GL}_r(A'_n\otimes_{A'}B)$ such that $$\delta_*(\phi_{\gamma_i})=V^{-1}\phi_{\gamma_i}V+V^{-1}\partial_{\gamma_i}(V)$$ for any $i=1,2,\dots,d$.

    Observe that the operator $\exp(\partial_\gamma)$ reproduces the action of $\gamma$ on $A'_n$ for any $\gamma\in p^{n}\Gamma$. By the subsection $3.2$ and the remark after Lemma \ref{refine}, we know that there exists an \'etale covering $A'\to B$ such that $\Phi|_{p^n\Gamma}$ can be extended to a cocycle (of the action $\exp(\partial_{-})$ on $A'_n\otimes_{A'}B$) $\Phi':\Gamma\to \mathrm{GL}_r(A'_n\otimes_A'B)$.

    By the cocycle conditions, $$\Phi_{\delta}\delta(\Phi_{\gamma})=\Phi_{\delta+\gamma}=\Phi_{\gamma}\gamma(\Phi_{\delta}).$$ Take (direction $\gamma$) derivation,
    $$\Phi_{\delta}\delta(\phi_{\gamma})=\phi_{\gamma}\Phi_{\delta}+\partial_{\gamma}(\Phi_{\delta}).$$ Act $\exp(-\partial_{\delta})$ on both side,
    $$\delta_*(\phi_{\gamma})=\exp(-\partial_{\delta})(\Phi_{\delta}^{-1})\cdot\exp(-\partial_{\delta})(\phi_{\gamma})\cdot\exp(-\partial_{\delta})(\Phi_{\delta})+\exp(-\partial_{\delta})(\Phi_{\delta})^{-1}\cdot\partial_{\gamma}\left(\exp(-\partial_{\delta})(\Phi_{\delta})\right)$$
    Similarly,
    $$\Phi'(-\delta)\exp(-\partial_{\delta})\left(\phi_{\gamma}\right)=\phi_{\gamma}\Phi'(-\delta)+\partial_{\gamma}\left(\Phi'(-\delta)\right)$$

    Take
    $$V=\Phi'(-\delta)\cdot\exp(-\partial_{\delta})(\Phi_{\delta}),$$ the above calculation implies $\delta_*(\Phi_{i})=V^{-1}\phi_{i}V+V^{-1}\partial_{\gamma_i}(V).$
\end{proof}

\begin{prop}
    The element $\nabla$ descends an element in $\mathrm{MIC}_r(A)\{-1\}$. Moreover, this element does not depend on the choice of $A\to A'$.
\end{prop}

\begin{proof}
    The above construction is clearly functorial for $A\to A'$. That is, for any diagram
    $$\xymatrix{
        A_1'\ar[rr]^-{g} &&A_2'\\
        &A\ar[ul]^-{f_1}\ar[ur]_-{f_2}
    }$$
    the diagram
    $$\xymatrix{
        \mathrm{H}^1\big(\Gamma,\mathrm{GL}_r(\widehat{A}_{1,\infty})\big)\ar[d]\ar[rr] &&\mathrm{H}^1\big(\Gamma,\mathrm{GL}_r(\widehat{A}_{2,\infty})\big)\ar[d]\\
        \mathrm{MIC}_{r}\big(\Spa(A_1',A_1'^{+})\big)\{-1\}\ar[rr]&&\mathrm{MIC}_{r}\big(\Spa(A_2',A_2'^{+})\big)\{-1\}
    }$$
    commutes. Thus, by the sheaf property of $\mathrm{MIC}_r(-)\{-1\}$, the element $\nabla$ lies in $\mathrm{MIC}_r\big(\Spa(A,A^+)\big)\{-1\}$.
\end{proof}

\subsection{Independence of toric charts}

In the previous subsection, we constructed a map for any algebra with a toric chart $f:\T\to A$ 
$$\mathfrak{LRH}_f:\mathrm{H}^1\big(\overline{X}_{\pet},\mathrm{GL}_r(\mathbb{B}_{\alpha})\big)\to \mathrm{MIC}_r(X)\{-1\}.$$
It is necessary to prove that this map is independent of the choice of $f$. The proof is analogous to Heuer's reconstruction of Camargo's Higgs field (\cite{heuer2023padic} Theorem 4.8). Let $X=\mathrm{Spa}(A,A^+)$ and $\overline{X}=\mathrm{Spa}(\overline{A},\overline{A}^+)$ as defined above.

Suppose there are two charts $$\xymatrix{\T_1=\Bp{\alpha}\langle T_1^{\pm1},T_2^{\pm1},\dots,T^{\pm1}_d\rangle\ar[dr]_-{f_1} &&\T_2=\Bp{\alpha}\langle S_1^{\pm1},S_2^{\pm1},\dots,S^{\pm1}_d\rangle\ar[dl]^-{f_2}\\ &A}.$$

We define $$A_{m,n}:=A[T_1^{\frac{1}{p^m}},T_2^{\frac{1}{p^m}},\dots, T_d^{\frac{1}{p^m}},S_1^{\frac{1}{p^n}},S_2^{\frac{1}{p^n}}\dots,S_d^{\frac{1}{p^n}}].$$
Similarly, $$\widehat{\overline{A}}_{\infty,n}=\widehat{\overline{A}}_{\infty,0}\otimes_AA_{0,n}$$ and $$\widehat{\overline{A}}_{m,\infty}=\widehat{\overline{A}}_{0,\infty}\otimes_AA_{m,0}$$ by almost purity, they are perfectoid algebras.
We also define
$$\widehat{A}_{\infty,n}=\Bpp{\alpha}(\widehat{\overline{A}}_{\infty,n});\ \widehat{\overline{A}}_{m,\infty}=\Bpp{\alpha}(\widehat{\overline{A}}_{m,\infty}).$$
We apply the consequences in \ref{toricalg} to $\T_1\to A_{0,n}$ and $\T_2\to A_{m,0}$, $A_{\infty,n}$ is a subalgebra of $\widehat{A}_{\infty,n}$ and $A_{m,\infty}$ is a subalgebra of $\widehat{A}_{m,\infty}$.
Finally, $$\widehat{\overline{A}}_{\infty,\infty};\ \widehat{A}_{\infty,\infty}$$ are the (unique) perfectoid tilde colimit of $\overline{A}_{m,n}$ and its $\Bpp{\alpha}$.

We have the following diagram:
$$\xymatrix{&&\widehat{A}_{\infty,\infty}&&\\&\widehat{A}_{\infty,n}\ar[ur]&&\widehat{A}_{m,\infty}\ar[ul]&
\\
\widehat{A}_{\infty,0}\ar[ur]&&A_{m,n}\ar[ul]\ar[ur]&&\widehat{A}_{0,\infty}\ar[ul]\\&A_{m,0}\ar[ul]\ar[ur]&&A_{0,n}\ar[ul]\ar[ur]&\\&&A\ar[ur]\ar[ul]&&}.$$

Let $\Gamma_1$ and $\Gamma_2$ be two copies of $\Z_p^{d}$ corresponding to two toric charts $f_1$ and $f_2$. Thus, by definition, there exist base change homomorphisms
$$\Rep_{r,\Gamma_1}^{\mathrm{cnt}}(A_{m,0})\to \Rep_{r,\Gamma_1}^{\mathrm{cnt}}(\widehat{A}_{\infty,0})\to \Rep_{r,\Gamma_1\times\Gamma_2}^{\mathrm{cnt}}(\widehat{A}_{\infty,\infty})$$ and similarly for $\Gamma_2$ and $A_{0,n}$.

Let $\gamma_i^{(1)}$ denote the $i$-th canonical basis vector. Define $\gamma_i^{(2)}\in\Gamma_2$ similarly. Recall from Subsection \ref{ToricSentheory}, we have actions
$\gamma_i^{(\epsilon)}$, $\gamma_{i,*}^{(\epsilon)}$ as well as
$\partial_{\gamma_i^{(\epsilon)}}$ for $\epsilon=1,2$, corresponding the algebras with toric charts 
$$\Bp{\alpha}\langle T_1^{\pm1},T_2^{\pm1},\dots,T^{\pm1}_d\rangle\xrightarrow{f_1} A$$ and $$\Bp{\alpha}\langle S_1^{\pm1},S_2^{\pm1},\dots,S^{\pm1}_d\rangle\xrightarrow{f_2} A.$$

\begin{theo}
    Suppose there exist $\Phi_1\in \Rep_{r,\Gamma_1}^{\mathrm{cnt}}(A_{m,0})$ and $\Phi_2\in \Rep_{r,\Gamma_2}^{\mathrm{cnt}}(A_{0,n})$ such that they are equivalent in $\Rep_{r,\Gamma_1\times\Gamma_2}^{\mathrm{cnt}}(\widehat{A}_{\infty,\infty})$, then the connections induced by them are isomorphic after restricting on $A_{m,n}$.
\end{theo}

\begin{proof}
    In the proof, we regard $\Gamma_1$ as the subgroup $\Gamma\times 1\subseteq \Gamma_1\times\Gamma_2$, and similarly for $\Gamma_2$.

    By assumption, there exists a matrix $V\in \mathrm{GL}_r(\widehat{A}_{\infty,\infty})$ such that for any $\delta_1\in \Gamma_1$ and $\delta_2\in \Gamma_2$,
    $$V\Phi_1(\delta_1)\left((\delta_1,\delta_2)(V)\right)^{-1}=\Phi_2(\delta_2).$$
    This implies that 
    $$\Phi_1(\delta_1)=V\delta_1(V)^{-1} \quad \text{and} \quad \Phi_2(\delta_2)=V\delta_2(V)^{-1}$$
    for all $\delta_1\in \Gamma_1$ and $\delta_2\in\Gamma_2$.

    Thus $T\in \mathrm{GL}_r(\widehat{A}_{\infty,\infty}^{\Gamma_1\times \Gamma_2-la})$. By taking derivatives, we obtain for each $\delta_1\in\Gamma_1$ and $\delta_2\in\Gamma_2$
    $$\phi_1(\delta_1)=V\partial_{\delta_1}(V^{-1});\ \phi_2(\delta_2)=V\partial_{\delta_2}(V^{-1}),$$
    where $\phi_1(\delta_1)=\lim_{c\in\Z_p\to 0}\tfrac{\Phi_1(c\gamma_1)-\Phi_1(0)}{c}$, and similarly for $\phi_2$. This equality holds in $\widehat{A}_{\infty,\infty}^{\Gamma_1\times\Gamma_2-la}$. However, since each coefficient lies in $A_{m,n}$, the equality also holds in $A_{m,n}$.

    Suppose \begin{equation}\label{dTdS}[T_1^{-1}dT_1,\dots, T_d^{-1}dT_d]=[S_1^{-1}dS_1,\dots, S_d^{-1}dS_d]U\end{equation} for some $U\in \mathrm{GL}_r(A)$. Thus 
    $$[T_1\partial_{1}^{(1)},\dots, T_d\partial_{d}^{(1)}]=[S_1\partial_{1}^{(2)},\dots, S_d\partial_{d}^{(2)}]U^{-t}.$$ 
    Since $\partial_{\gamma_i^{(1)}}= tT_i\partial_i^{(1)}$ and $\partial_{\gamma_i^{(2)}}=tS_i\partial_i^{(2)}$ (see Subsection \ref{ToricSentheory}), we have
    $$[\partial_{\gamma_1^{(1)}},\dots, \partial_{\gamma_d^{(1)}}]=[\partial_{\gamma_1^{(2)}},\dots, \partial_{\gamma_d^{(2)}}]U^{-t}.$$
    Hence,
    \begin{align}
    [\phi_1(\gamma_1^{(1)}),\phi_1(\gamma_2^{(1)}),\dots,\phi_1(\gamma_d^{(1)})]&=[V\partial_{\gamma_1^{(1)}}(V^{-1}),V\partial_{\gamma_2^{(1)}}(V^{-1}),\dots,V\partial_{\gamma_d^{(1)}}(V^{-1})]\nonumber\\
    &=[V\partial_{\gamma_1^{(2)}}(V^{-1}),V\partial_{\gamma_2^{(2)}}(V^{-1}),\dots,V\partial_{\gamma_d^{(2)}}(V^{-1})]U^{-t}\nonumber\\
    &=[\phi_2(\gamma_1^{(2)}),\phi_2(\gamma_2^{(2)}),\dots,\phi_2(\gamma_d^{(2)})]U^{-t}\label{gamma}
    \end{align}
    
    By our construction, the connection $\nabla_1:A_{m,0}^r\to A_{m,0}^{r}\otimes_A\Omega_{A}\{-1\}$ associated to $\Phi_1$ satisfies that
    $$\nabla_1(e_i)=t^{-1}\sum_{j=1}^d[e_1,e_2,\dots,e_d]\phi_1(\gamma_j^{(1)})\cdot \tfrac{dT_j}{T_j}$$
    and the connection $\nabla_2:A_{0,n}^r\to A_{0,n}^{r}\otimes_A\Omega_{A}\{-1\}$ associated to $\Phi_2$ satisfies that
    $$\nabla_2(e_i)=t^{-1}\sum_{j=1}^d[e_1,e_2,\dots,e_d]\phi_2(\gamma_j^{(2)})\cdot \tfrac{dS_j}{S_j}.$$
    By \eqref{dTdS} and (\ref{gamma}), $\nabla_1$ and $\nabla_2$ are the same as connections on the $A_{m,n}$-module $A_{m,n}^r$.
\end{proof}

\begin{cor}
    The maps $\mathfrak{LRH}_{f_1}$ and $\mathfrak{LRH}_{f_2}$ from $\mathrm{H}^1\big(\overline{X}_{\pet},\mathrm{GL}_r(\mathbb{B}_{\alpha})\big)$ to  $\mathrm{MIC}_r(X)\{-1\}$ coincide.
\end{cor}

\begin{proof}
    The $v$-descent of vector bundles over $\mathrm{Spa}(\widehat{A}_{\infty,\infty},\widehat{A}_{\infty,\infty}^+)$ shows that the left-hand side coincides with the isomorphism classes of finite rank projective modules as well as a semi-linear action of $\Gamma_1\times\Gamma_2$. An argument analogous to that in Lemma \ref{refine} shows that any projective module becomes free after some \'etale base change over $A$ and then reduces to the finite free case. The above theorem implies this claim.
\end{proof}

\subsection{Proof of Theorem \ref{theorem: sheafRH}}

During the above two subsections, we have shown that there is a canonical map $$RH:\mathrm{H}^1\big(\overline{X}_{\pet},\mathrm{GL}_r(\mathbb{B}_{\alpha})\big)\to \mathrm{MIC}_r(X)\{-1\}$$ for any $X$ smooth over $\Bp{\alpha}$. Take sheafification, we get a sheaf morphism $$RH:R^1\nu_*\big(\mathrm{GL}_r(\mathbb{B}_{\alpha})\big)\to \mathrm{MIC}_r(X)\{-1\}.$$ It remains to prove that this is an isomorphism.

By sheaf property, we only need to prove it when $X=\mathrm{Spa}(A,A^+)$ and $A$ has a toric chart $f:\T\to A$.

\begin{theo}\label{theo: injlocalrh}
    For any $\Phi_1,\Phi_2\in \mathrm{H}^1\big(\overline{X}_{\pet},\mathrm{GL}_r(\mathbb{B}_{\alpha})\big)$, such that $$\mathfrak{LRH}_f(\Phi_1)=\mathfrak{LRH}_f(\Phi_2),$$ there exists a \'etale covering $A\to A'$ such that $\Phi_1=\Phi_2$ in $$\mathrm{H}^1\left(\mathrm{Spa}(\overline{A}',\overline{A}'^+)_{\pet},\mathrm{GL}_r(\mathbb{B}_{\alpha})\right).$$
\end{theo}

\begin{proof}
    By Lemma \ref{refine}, we can choose a covering $A\to A'$ such that $\Phi_1,\Phi_2$ comes from generalized representations. Then we choose an integer $m>0$ such that $\Phi_1$ and $\Phi_2$ can be decompleted to $A'_m$ and thus we can assume 
    $$\Phi_1(p^m\Gamma),\Phi_2(p^m\Gamma)\subseteq 1+p^2M_r({A'}_m^+)$$
    by enlarging $m$.

    Let 
    $$\nabla_{i}:\Gamma\to \End_{\Bpp{\alpha}(K,K^+)}({A'}_m^r)\ (i\in\{1,2\})$$ 
    be the connection $\mathfrak{LRH}(\Phi_i)$. By the definition of $\mathfrak{LRH}$, $\nabla_{i,\gamma}=\log(\Phi_{i,\gamma})$ for any $\gamma\in p^m\Gamma$ (here $\Phi_{i,\gamma}$ is considered as a semi-linear action but not just a matrix) for $i=1,2$. By the definition of $\mathrm{MIC}\{-1\}$, we can choose \'etale covering $A_m'\to A''$ such that there exists $V\in \mathrm{GL}_{A''}({A''}^r)$ such that $V\circ\nabla_{1,\gamma}\circ V^{-1}=\nabla_{2,\gamma}$ for any $\gamma\in p^m\Gamma$ and take $\exp$ causes that $\Phi_1$ and $\Phi_2$ are equivalent in $\Rep_r(p^m\Gamma,A'')$. Hence, they are equal in $\mathrm{H}^1\big(\mathrm{Spa}(\overline{A}'',\overline{A}''^+)_{\pet},\mathrm{\mathrm{GL}}_r(\mathbb{B}_{\alpha}^+)\big)$.
\end{proof}

\begin{cor}
    The map $RH$ is an injection.
\end{cor}

\begin{proof}
    This is obvious by Theorem \ref{theo: injlocalrh}.
\end{proof}

\begin{theo}
    The map $RH$ is a surjection.
\end{theo}

\begin{proof}
    For any $\nabla:\Gamma\to \End_{\Bpp{\alpha}(K,K^+)}(A^r)$, we choose a large $m$ such that $\exp\nabla_\gamma$ converges on $p^m\Gamma$ and let $\Phi_{\gamma}$ be its $\exp$. Then $\Phi$ is a cocycle of $p^m\Gamma$ on $\mathrm{\mathrm{GL}}_r(A)\subseteq \mathrm{\mathrm{GL}}_r(A_m)$. By definition, the connection of this cocycle is the base change of $\nabla$ to $A_m$. This proves the surjectivity.
\end{proof}

\section{Moduli spaces}\label{section: moduli}

Throughout this section, we fix a perfectoid field $C$ over $\Q_p$ and a system of $p^n$-th roots $\zeta_{p^n}$. We use $\Bp{\alpha}$ to denote $\Bpp{\alpha}(C,C^+)$. Suppose $X$ is a smooth adic space over $\Bp{\alpha}$ and $\overline{X}$ as above. Let $\Perf$ be the category of affinoid perfectoid spaces over $C$ and consider its $v$-topology. We will sometimes omit the ring of integers when we write an adic spectrum; see Subsection \ref{Subsection: Notations} for the definition of the rings of integers.

Recall the following definitions.

\begin{defi}
    For any integer $r>0$, define the following:
    \begin{enumerate}[(1)]
        \item Let $\M_{\dR,\alpha,r,X}$ be the prestack sending $\mathrm{Spa}(A,A^+)\in \Perf$ to the groupoid of pairs $(M,\nabla)$, where $M$ is a rank $r$ vector bundle on $X_A$ and 
        $$\nabla:M\to M\otimes \Omega\{-1\}$$
        is a connection with respect to $\widetilde{d}$ (see Example \ref{eg:BKdiff}).

        \item Let $\M_{{\mathbb{B}_{\alpha}},r,X}$ be the prestack sending $\mathrm{Spa}(A,A^+)\in \Perf$ to the groupoid of rank $r$ locally free $\Bpp{\alpha}$-sheaves on the $v$-topology of $\overline{X}_{A}$. Here, $\overline{X}=X\otimes_{B_\alpha}C$ and $\overline{X}_A$ is the usual base change.
    \end{enumerate}
\end{defi}

We will prove the following theorem:

\begin{theo}[=Theorem \ref{theorem: modulismall}]
    The prestacks $\M_{\dR,\alpha,r,X}$ and $\M_{{\mathbb{B}_{\alpha}},r,X}$ are small $v$-stacks.
\end{theo}

\begin{rmk}
    The prestack $\M_{{\mathbb{B}_{\alpha}},r,X}$ is by definition a $v$-stack.
\end{rmk}

Before proving Theorem \ref{theorem: modulismall}, we emphasize that by a connection on $X$ we always mean a connection on $X$ with respect to $d:\inte{X}\to \Omega_X\{-1\}$ (cf. Definition \ref{defi: connections}).

\subsection{More on algebras}

We recall some basic facts about perfectoid algebras.

\begin{theo}
    For any diagram of perfectoid Tate rings
    $$\xymatrix{ &A\\ B &D\ar[l]\ar[u]},$$ the complete tensor product $A\widehat{\otimes}_DB$ is a perfectoid Tate ring.
\end{theo}

\begin{proof}
    See \cite[Proposition 6.18]{scholze2012perfectoid}.
\end{proof}

Our first goal is to prove that the functor $\Bpp{\alpha}$ commutes with complete tensor products (Theorem \ref{bdrtensor}).

\begin{lemma}\label{witttensor}
    Suppose 
    $$\xymatrix{ &P\\ Q &R\ar[l]\ar[u]}$$ 
    is a diagram of perfect rings of characteristic $p$. Then the canonical map $$W_n(P)\otimes_{W_n(R)} W_n(Q)\to W_n(P\otimes_RQ)$$ is an isomorphism.
\end{lemma}

\begin{proof}
    Recall that for any perfect ring $R$ and a $\Z/p^n$-algebra $T$, the modulo $p$ map
    $$\Hom\left(W_n(R),T\right)\to \Hom(R,T/p^n)$$ is a bijection. The claim follows from this universal property.
\end{proof}

\begin{lemma}\label{wittcompletion}
    Let $R$ be a perfect ring of characteristic $p$ and $r\in R$. Let $\widehat{R}$ be the $r$-adic completion and $i:R\to \widehat{R}$ be the canonical map; we still use $r$ to denote $i(r)$ in $\widehat{R}$. Then for any $m,n\geq 1$, the canonical map $$W_n(R)/[r]^m\to W_n(\widehat{R})/[r]^m$$ is an isomorphism.
\end{lemma}

\begin{proof}
    Since the $r$-adic completion is equivalent to the $r^m$-adic completion, it suffices to consider the case $m=1$.

    The canonical map $R/r\to \widehat{R}/r$ is an isomorphism, so $i^{-1}(r\widehat{R})=rR$. Since each element in $W_n(R)$ can be uniquely written as $$\sum_{j=0}^{n-1}[t_j]p^j,$$ its image in $W_n(\widehat{R})$ is divisible by $r$ if and only if $i(t_j)\in r\widehat{R}$ for each $j$. Hence $t_j\in rR$, and injectivity follows.

    For any $$\sum_{j=0}^{n-1}[s_j]p^j$$ in $W_n(\widehat{R})$, we can choose $t_j\in R$ such that $i(t_j)-s_j\in r\widehat{R}$, and thus the image of $$\sum_{j=0}^{n-1}[t_j]p^j$$ in $W_n(R)/[r]$ maps to $\sum_{j=0}^{n-1}[s_j]p^j$ in $W_n(\widehat{R})/[r]$. Surjectivity follows.
\end{proof}

Now we return to the diagram
$$\xymatrix{ &A\\ B &D\ar[l]\ar[u]}$$ of perfectoid Tate rings.

\begin{theo}\label{bdrtensor}
    Under the above assumption, the canonical homomorphism $$\Bpp{\alpha}(A)\widehat{\otimes}_{\Bpp{\alpha}(D)}\Bpp{\alpha}(B)\to \Bpp{\alpha}(A\widehat{\otimes}_DB)$$ is an isomorphism.
\end{theo}

\begin{proof}
    Fix a set of generators $t$ of $\ker\big(\theta:W(\inte{D}^{+,\flat})\to \inte{D}\big)$, and let $\pi\in\inte{D}^{{+,\flat}}$ be a pseudo-uniformizer. By definition, it suffices to prove that
    $$\left(W(A^{{+,\flat}})/t^\alpha\otimes_{W(D^{{+,\flat}})/t^\alpha}W(B^{{+,\flat}})/t^\alpha\right)^\land\to W(A^{{+,\flat}}\widehat{\otimes}_{D^{{+,\flat}}}B^{{+,\flat}})/t^\alpha$$ is an isomorphism, where the outer completion is $p$-adic and the inner one is $\pi$-adic. Now it suffices to prove for any $N$ that
    $$\left(W(A^{{+,\flat}})/t^\alpha\otimes_{W(D^{{+,\flat}})/t^\alpha}W(B^{{+,\flat}})/t^\alpha\right)/p^N\cong \left(W(A^{{+,\flat}}\widehat{\otimes}_{D^{{+,\flat}}}B^{{+,\flat}})/t^\alpha\right)/p^N.$$

    By Lemma \ref{wittcompletion}, the canonical map $$W_N(A^{{+,\flat}}\widehat{\otimes}_{D^{{+,\flat}}}B^{{+,\flat}})/[\pi]^m\to W_N(A^{{+,\flat}}\otimes_{D^{{+,\flat}}}B^{{+,\flat}})/[\pi]^m$$ is an isomorphism for any $m$, and hence so is $$W_N(A^{{+,\flat}}\widehat{\otimes}_{D^{{+,\flat}}}B^{{+,\flat}})/(t^\alpha,[\pi]^m)\to W_N(A^{{+,\flat}}\otimes_{D^{{+,\flat}}}B^{{+,\flat}})/(t^\alpha,[\pi]^m).$$ Since for sufficiently large $m$, $[\pi^m]\in(p^N,t^\alpha)$, the canonical map $$\left(W(A^{{+,\flat}}\otimes_{D^{{+,\flat}}}B^{{+,\flat}})/t^\alpha\right)/p^N\to \left(W(A^{{+,\flat}}\widehat{\otimes}_{D^{{+,\flat}}}B^{{+,\flat}})/t^\alpha\right)/p^N$$ is an isomorphism.

    By Lemma \ref{witttensor}, $$\left(W(A^{{+,\flat}})/t^\alpha\otimes_{W(D^{{+,\flat}})/t^\alpha}W(B^{{+,\flat}})/t^\alpha\right)/p^N\cong \left(W(A^{{+,\flat}}\otimes_{D^{{+,\flat}}}B^{{+,\flat}})/t^\alpha\right)/p^N.$$ Thus the claim is proved.
\end{proof}

\begin{cor}
    Suppose $(A,A^+)\in \Perf$, $\pi:\mathrm{Spa}(B,B^+)\to \mathrm{Spa}(A,A^+)$ is a $v$-covering, and $f:\T_C\to R$ is a $C$-algebra with a toric chart. Then the homomorphism $$\Bpp{\alpha}(A)\widehat{\otimes}_{\Bp{\alpha}}R\to \Bpp{\alpha}(B)\widehat{\otimes}_{\Bp{\alpha}}R$$ is injective.
\end{cor}

\begin{proof}
    We have known that $R$ is a direct summand of $\widehat{R}_\infty$. Hence, it suffices to prove that $$\Bpp{\alpha}(A)\widehat{\otimes}_{\Bp{\alpha}}\widehat{R}_\infty\to \Bpp{\alpha}(B)\widehat{\otimes}_{\Bp{\alpha}}\widehat{R}_\infty$$ is injective. This is a consequence of Theorem \ref{bdrtensor}.
\end{proof}

The next goal is to establish some useful lemmas to prove smallness. These lemmas are already used by Fargues and Scholze in \cite{fargues2021geometrization}.

\begin{lemma}
    Suppose $(A,A^+)$ is a perfectoid Tate--Huber pair and $\varpi$ is a pseudo-uniformizer. For any closed subalgebra $B\subseteq A$ that is topologically countably generated and contains $\varpi$, there exists a topologically countably generated closed subalgebra $B_1\subseteq A$ containing $B$ such that $$B^+/\varpi\subseteq (B_1^+/\varpi)^p,$$ where $B^+=B\cap A^+$ and $B^+_1=B_1\cap A^+$.
\end{lemma}

\begin{rmk}
    By definition, $\varpi B^+=\varpi A^+\cap B$, which implies that $B^+/\varpi$ is a subalgebra of $B^+_1/\varpi$.
\end{rmk}

\begin{proof}
    Let $S\subseteq B$ be a countable dense subset and define $S^+=S\cap B^+$. For any $x\in B^+$ and $n>0$, there exists an $s\in S$ such that $x-s\in \varpi^nB^+$, and thus $s\in B^+$. This proves that $S^+$ is dense in $B^+$.

    For each $s\in S^+$, since $A^+$ is integrally perfectoid, we can choose an $s_1\in A^+$ such that $s_1^p-s\in \varpi A^+$. Let $S_1=S\cup \{s_1:s\in S^+\}$ and $B_1$ be the closure of the algebra generated by $S_1$. Then $B_1\supset B$ by the density of $S$ in $B$ and satisfies the condition. By the definition of $\{s_1:s\in S^+\}$, the image of $S^+$ in $B^+/\varpi$ lies in $(B_1^+/\varpi)^p$, and hence so does $B^+/\varpi$ since $S$ is a set of topological generators.
\end{proof}

\begin{prop}\label{count}
    Suppose $(A,A^+)$ is a perfectoid Tate--Huber pair. Then any countable subset of $A$ is contained in a perfectoid topologically countably generated closed subalgebra.
\end{prop}

\begin{proof}
    Fix a pseudo-uniformizer $\varpi$ such that $\varpi^p\mid p$.

    Assume $S\subseteq A$ is a countable subset containing $\varpi$, and let $B$ be the closed subalgebra topologically generated by $S$. By the above lemma, there exists a sequence of closed subalgebras $\{B_i:i\geq 0\}$ such that $B_0=B$ and $B_i^+/\varpi\subseteq (B_{i+1}^+/\varpi)^p$. Let $B_{\infty}$ be the union of $B_i$ and $\widehat{B}_{\infty}$ its closure. It suffices to prove that $\widehat{B}_\infty$ is a perfectoid topologically countably generated closed subalgebra.

    Let $S_i\subseteq B_i$ be a countable dense subset for each $i$. Then $\bigcup S_i$ is a countable dense subset of $\widehat{B}_\infty$, proving that $\widehat{B}_\infty$ is topologically countably generated. Moreover, observe that $B_\infty^+/\varpi=\bigcup B_i^+/\varpi$, and its Frobenius is surjective by the definition of $B_i$. Hence $\widehat{B}_\infty$ is perfectoid.
\end{proof}

\begin{rmk}
    This proposition can be generalized to the relative case, just change $S$ to a countably generated dense subalgebra.
\end{rmk}

\begin{cor}\label{cor: countable subset countable generated}
    Under the above assumptions, any countable subset of $\Bpp{\alpha}(A)$ is contained in the $\Bpp{\alpha}$ of a perfectoid topologically countably generated closed subalgebra.
\end{cor}

\begin{proof}
    Induction on $\alpha$. When $\alpha=1$, it has been proved above. Suppose we have already proved the case for $\alpha-1$. Let $(A,A^+)$ be a perfectoid Tate--Huber pair and $S\subseteq \Bpp{\alpha}(A)$ be a countable subset. By the induction hypothesis, there exists a topologically countably generated perfectoid subalgebra $A_1\subseteq A$ such that $\Bpp{\alpha-1}(A_1)$ contains the projection of $S$ in $\Bpp{\alpha-1}(A)$.

    Thus for any $s\in S$, there exists a $s'\in \Bpp{\alpha-1}(A_1)$ such that $s'-s\in t^{\alpha-1}\Bpp{\alpha}(A_1)$ and let $u_s$ denote an element in $A$ such that $t^{\alpha-1}u_s=s-s'$. Then, enlarge $A_1$ so that it is perfectoid, topologically countably generated and contains all $u_s$, then $A_1$ satisfies the assumption.
\end{proof}

\subsection{$V$-descent of connections}

Suppose $\pi:\mathrm{Spa}(B,B^+)\to \mathrm{Spa}(A,A^+)$ is a $v$-cover of affinoid perfectoid space over $C$. We thus have a coequalizer diagram:
$$0\to A\to B\rightrightarrows B\widehat{\otimes}_AB.$$ Define $X_B:=\mathrm{Spa}(\Bpp{\alpha}(B,B^+))\times_{\mathrm{Spa}(\Bpp{\alpha}(A,A^+))}X$ and similarly $X_{B\widehat{\otimes}_AB}.$ Thus we have two morphisms $$\pi_{1,2}:X_{B\widehat{\otimes}_AB}\to X_B.$$

\begin{theo}\label{vstack}
    For any vector bundle with an integrable connection $(M,\nabla)$ on $X_B$ as well as a descent datum $$\phi: \pi_{1}^{*}(M,\nabla)\cong \pi_2^*(M,\nabla),$$ there exists a unique vector bundle with an integrable connection $(N,\nabla')$ as well as an isomorphism $\iota:\pi^*(N,\nabla')\cong (M,\nabla)$. Such that $\phi$ is equiped with the natural morphism
    $$\Bpp{\alpha}(\mathbf{trans})\otimes id_N:\Bpp{\alpha}(B\widehat{\otimes}_AB)\otimes_{\Bpp{\alpha}(A)}N\cong \Bpp{\alpha}(B\widehat{\otimes}_AB)\otimes_{\Bpp{\alpha}(A)}N$$ under $\iota$, where $\mathbf{trans}$ is the map $x\widehat{\otimes}y\to y\widehat{\otimes}x$.
\end{theo}

\begin{proof}
    We may assume that $X$ is affinoid, say $\mathrm{Spa}(R,R^+)$ and $R$ admits an \'etale homomorphism
    $$f:\T=\Bp{\alpha}\langle T_1^{\pm1},\dots,T_d^{\pm1}\rangle\to R.$$ Use $R_A$ to denote $\Bpp{\alpha}(A,A^+)\widehat{\otimes}R$ and similarly, the map $f_A:\T_A\to R_A$.

    Recall our notations $R_{N},\ R_{\infty},\ \widehat{R}_{\infty}$ in Subsection \ref{toricalg} and $R_{*,-},\ \widehat{R}_{*,\infty}$ are the same construction for $f_*:\T_*\to R_*$ where $*$ is an object in $\Perf$ and $-\in\{N,\ \infty\}$. $\Gamma\cong\mathbb{Z}_p^d$ is the Galois group.

    Fix a descent datum $(V,\nabla,i:V\widehat{\otimes}_R R_B\cong R_B\widehat{\otimes}_RV)$ such that $V$ is free over $R_B$, $\nabla$ is a connection on $V$ and $i$ is an isomorphism between $R_B\widehat{\otimes}_RR_B=R_{B\widehat{\otimes}_AB}$ modules.

    Consider the diagram
    $$\xymatrix{0\ar[r] &R\ar[r]\ar[d]&R_B\ar@<0.5ex>[r]\ar@<-0.5ex>[r]\ar[d] & R_{B\widehat{\otimes_A}B}\ar[d]\\ 0\ar[r] &R_{N}\ar[r]&R_{B,N}\ar@<0.5ex>[r]\ar@<-0.5ex>[r]& R_{B\widehat{\otimes}_AB,N}}$$ where the bottom ones are a $\Gamma/p^N\Gamma$-Galois cover of the upper ones and all maps are $\Gamma$-equivariant (under the natural action, see Subsection \ref{ToricSentheory} ). Similarly, we have
    $$\xymatrix{0\ar[r] &R\ar[r]\ar[d]&R_B\ar@<0.5ex>[r]\ar@<-0.5ex>[r]\ar[d] & R_{B\widehat{\otimes}_AB}\ar[d]\\
    0\ar[r] &\widehat{R}_{\infty}\ar[r]&\widehat{R}_{B,\infty}\ar@<0.5ex>[r]\ar@<-0.5ex>[r]& \widehat{R}_{B\widehat{\otimes}_AB,\infty}}.$$

    It suffices to prove that there exists an $N$ such that $(V,\nabla,i)\otimes_RR_N$ can descent to $R_N$. In fact, if there exists such an $N$. Then the solution of the descent datum $(V,\nabla,i)\otimes_RR_N$ must be
    $$V'_N:=\{v\in V:i(v\widehat{\otimes}1)=1\widehat{\otimes}v\text{ in } R_{B,N}\otimes V\}$$
    and the connection $\nabla'_N$ must be the restriction of $\nabla_{R_N}$ on $V_N'$. Hence, $(V'_N,\nabla'_N)$ must be invariant under the Galois action of $\Gamma/p^N\Gamma$. By the usual \'etale descent of vector bundles, $(V_N',\nabla_N')$ can descent to $X$.

    Recall, we may write $\nabla$ as a $\mathbb{Z}_p$ linear map $\phi_{-}:\Gamma\to \End_{\Bp{\alpha}}(V)$ such that $$\phi_{\gamma}(ax)=\partial_{\gamma}(a)x+a\phi_{\gamma}(x).$$ Choose $N_1$ large enough such that $\exp(\phi_{\gamma})$ converges uniformly on $p^{N_1}\Gamma$. Base change to $\widehat{R}_{B,\infty}$, we get a finite projective module $\widehat{R}_{B,\infty}\otimes_{R_B}V$ as well as a $p^{N_1}\Gamma$'s semi-linear action, which is defined as the tensor product of the natural action on $\widehat{R}_{\infty}$ and $\exp(\phi_{-})$ on $V$, plus a descent datum (of generalised representation) along
    $$\xymatrix{
    0\ar[r] &\widehat{R}_{\infty}\ar[r]&\widehat{R}_{B,\infty}\ar@<0.5ex>[r]\ar@<-0.5ex>[r]& \widehat{R}_{B\widehat{\otimes}_AB,\infty}}.$$ By $v$-descent of $\Bpp{\alpha}$-bundle of perfectoid spaces, the above datum descents to a generalised representation of $p^{N_1}\Gamma$ over $\widehat{R}_{\infty}$. The local construction of the Riemann--Hilbert correspondence provides an $N\geq N_1$ and a vector bundle with a connection over $R_N$, and this is a descent of $(V,\nabla,i)\otimes_{R}R_N$.
\end{proof}

\subsection{Proof of the smallness}

The proof to smallness is essentially the same as Heuer's proof when $\alpha=1$. We will use the following important fact due to Fargues and Scholze.

\begin{prop}[\cite{fargues2021geometrization}, Proposition \uppercase\expandafter{\romannumeral3}.1.3]
    Suppose $X$ is a $v$-stack over $\Perf$, then $X$ is small if for any $\mathrm{Spa}(R,R^+)\in \Perf$,
    $$X(R,R^+)=\varinjlim X(R_i,R_i^+)$$ where the colimit takes among all topologically countably generated closed subalgebras.
\end{prop}

\begin{proof}(Fargues \& Scholze)
    The isomorphic classes of perfectoid topologically countably generated Tate rings forms a set and let $\mathcal{C}$ a set of representative elements. We can choose $$T=\bigsqcup_{\mathrm{Spa}(R_i,R_i^+),\ *\in X(R_i,R_i^+)}\mathrm{Spa}(R_i,R_i^+)$$ and it covers $X$ (cf. Proposition \ref{count}). The equivalent relation $T\times_{X}T$ satisfies the same limit property and thus $X$ is small.
\end{proof}

\begin{lemma}\label{smalldr}
    For any $\mathrm{Spa}(A,A^+)\in \Perf$ and a vector bundle $V/X_A$, there exists a perfectoid topologically countably generated closed subalgebra $A'\subseteq A$, such that there exists a vector bundle $V'/X_{A'}$ and $V$ is isomorphic to the pull back of $V'$.
\end{lemma}

\begin{proof}
    Since $X$ is quasi-compact, there exists a finite affinoid covering of $X$ and we may assume that $X\cong \mathrm{Spa}(R,R^+)$ is affinoid which admits a toric chart. There is a rational open covering $\mathcal{U}$ of $X_A$ trivialize $V$, we first prove that there exists a perfectoid topologically countably generated closed subalgebra $A'\subseteq A$ such that $\mathcal{U}$ comes from a pull back of an open covering of $X_{A'}$. Choose a set of topological generators $\{x_1,x_2,\dots,x_s\}$ of $R$ and by definition, $$\Bpp{\alpha}(A)[x_1,x_2,\dots,x_s]$$ is a dense subalgebra of $R_A$. By general facts about adic spaces, each rational open of $X_A$ can be defined by finitely many elements in $\Bpp{\alpha}(A)[x_1,x_2,\dots,x_s]$ and thus by finitely many elements of $A$. Then, Corollary \ref{cor: countable subset countable generated} implies the claim.

    Second, we need to prove the gluing data can be define over some $X_{A'}$ for some perfectoid topologically countably generated closed subalgebra $A'\subseteq A''\subseteq A$. This is equivariant to proving
    $$\varinjlim_i R_{A_i}=R_{A}.$$ Since $\Bpp{\alpha}(A)[x_1,x_2,\dots,x_s]$ is dense in $R_A$, any element in $R_A$ is a (countable) limit of some elements in $\Bpp{\alpha}(A)[x_1,x_2,\dots,x_s]$. Thus any element of $R_A$ is determined by countably many elements in $\Bpp{\alpha}(A)$ and Corollary \ref{cor: countable subset countable generated} also implies the claim.
\end{proof}

\begin{theo}
    The $v$-stacks $\M_{\dR,\alpha,r,X}$ and $\M_{{\mathbb{B}_{\alpha}},r,X}$ are small.
\end{theo}

\begin{proof}
    A similar argument of Lemma \ref{smalldr} proves that 
    $$\M_{\dR,\alpha,r,X}(R,R^+)=\varinjlim_{i}\M_{\dR,\alpha,r,X}(R_i,R_i^+)$$ 
    where the colimit takes among all perfectoid topologically countably generated closed subalgebras, Hence, $\M_{\dR,\alpha,r,X}$ is small.

    To prove $\M_{{\mathbb{B}_{\alpha}},r,X}$ is small, we may still assume $X$ is affinoid which admits a toric chart. A similar argument of Theorem \ref{vstack} as well as the above paragraph implies that for any $\Bpp{\alpha}$-bundle over $v$-topology of $X_A$, there exists a perfectoid topologically countably generated closed subalgebra $A'\subseteq A$ and an \'etale covering $X'\to X_{A'}$ such that $V|_{X'\otimes_{\Bpp{\alpha}(A')}\Bpp{\alpha}(A)}$ can be defined over $X'$. Then use the same argument, we can enlarge $A'$ to permit the descent datum of $v$-bundles along $X'\otimes_{\Bpp{\alpha}(A')}\Bpp{\alpha}(A)\to X_A$ can be defined over $X_{A'}$, thus $V$ can be defined over $X_{A'}$. The same argument also proves that isomorphisms can be defined over some perfectoid topologically countably generated closed subalgebra.
\end{proof}

\bibliographystyle{alpha}
\bibliography{ref}

\end{document}